\documentclass[12pt]{amsart}
\usepackage{amssymb,amsmath, amsthm,latexsym,verbatim}
\usepackage{a4wide}
\usepackage{mathrsfs}
\usepackage[backref=page]{hyperref}
\usepackage{enumitem}
\usepackage{todonotes}
\usepackage{xcolor}

\theoremstyle{plain}

\newtheorem{lem}{Lemma}[section]
\newtheorem{theo}[lem]{Theorem}
\newtheorem{prop}[lem]{Proposition}
\newtheorem{cor}[lem]{Corollary}
\newtheorem{remark}[lem]{Remark}

\newtheorem{conjecture}[lem]{Conjecture}

\parskip=\medskipamount
\font\k=cmr7
\font\rm=cmr12

  \newcommand {\di}{\mbox{\k dis}}
  \newcommand {\fin}{{\mbox{\k fin}}}
  \newcommand {\sico}{\mbox{\k sc}}
  \newcommand{\unip}{\operatorname{unip}}

  \newcommand {\reg}{\text{reg}}
  
  \newcommand {\spec}{\text{spec}}
  \newcommand {\geo}{\mbox{\k geom}}
  
  \newcommand {\temp}{\mbox{\k temp}}

  \newcommand {\tri}{\mbox{\k tri}}
  \newcommand {\C}{{\mathbb C}}
  
  \newcommand {\N}{{\mathbb N}}
  \newcommand {\R}{{\mathbb R}}
  \newcommand {\Z}{{\mathbb Z}}
  \newcommand {\Q}{{\mathbb Q}}
  \newcommand {\A}{{\mathbb A}}
  
  \newcommand{\bG}{G}
  \newcommand{\bK}{{\bf K}}
  \newcommand {\af}{{\mathfrak a}}
  \newcommand {\gf}{{\mathfrak g}}
  
  \newcommand {\kf}{{\mathfrak k}}

  \newcommand {\of}{{\mathfrak o}}
  \newcommand {\nf}{{\mathfrak n}}
  \newcommand {\lf}{{\mathfrak l}}
  \newcommand {\uf}{{\mathfrak u}}
  
  \newcommand {\pg}{{\mathfrak p}}
   \newcommand {\pf}{{\mathfrak p}}

\renewcommand {\H}{{\mathcal H}}
  
  \newcommand {\M}{{\mathcal M}}
  \newcommand {\CmF}{{\mathcal F}}
  \newcommand {\Co}{{\mathcal C}}
 \newcommand {\cO}{{\mathcal O}}
 
\newcommand {\cF}{{\mathcal F}}
 
  \newcommand {\G}{G}

\newcommand {\CmL}{{\mathcal L}}

 \newcommand {\cH}{{\mathcal H}}
 \newcommand {\cP}{{\mathcal P}}
 \newcommand {\cL}{{\mathcal L}}
 \newcommand {\cA}{{\mathcal A}}
 
\newcommand {\cT}{{\mathcal T}}
\newcommand  {\cZ}{{\mathcal Z}}
\newcommand {\bs}{\backslash}

 \newcommand {\ov}{\overline}

\newcommand{\levis}{{\mathcal L}}
\newcommand{\Ai}{A_M(\R)^0}

\renewcommand{\Im}{\operatorname{Im}}
\renewcommand{\Re}{\operatorname{Re}}
\newcommand{\Tr}{\operatorname{Tr}}

\newcommand{\End}{\operatorname{End}}

\newcommand{\tr}{\operatorname{tr}}
\newcommand{\Id}{\operatorname{Id}}
\newcommand{\Hom}{\operatorname{Hom}}

\newcommand{\Ind}{\operatorname{Ind}}

\newcommand{\inj}{\operatorname{inj}}
\newcommand{\proj}{\operatorname{proj}}

\newcommand{\rk}{\operatorname{rank}}
\newcommand{\vol}{\operatorname{vol}}

\newcommand{\SL}{\operatorname{SL}}
\newcommand{\GL}{\operatorname{GL}}
\newcommand{\SO}{\operatorname{SO}}

\newcommand{\Ad}{\operatorname{Ad}}

\newcommand{\rO}{\operatorname{O}}
\newcommand{\supp}{\operatorname{supp}}

\renewcommand{\det}{\operatorname{det}}

\newcommand{\Rep}{\operatorname{Rep}}

\newcommand{\Liesl}{\mathfrak{sl}}
\newcommand{\Liegl}{\mathfrak{gl}}

\newcommand{\norm}[1]{\lVert#1\rVert}
\newcommand{\abs}[1]{\lvert#1\rvert}
\newcommand{\eps}{\epsilon}
\newcommand{\one}{\mathbf 1}
\newcommand{\oP}{\overline P}

\newcommand{\aaa}{\mathfrak{a}}
  \newcommand {\K}{{\bf K}}
  
  \newcommand {\bU}{{\bf U}}
  
  \newcommand{\Ht}{H}
\newcommand{\sprod}[2]{\left\langle#1,#2\right\rangle}
\newcommand{\PPP}{\mathcal{P}}
\newcommand{\FFF}{{\mathcal F}}
\newcommand{\rts}{\Sigma}

\newcommand{\disc}{\operatorname{disc}}
\newcommand{\srts}{\Delta}
\newcommand{\modulus}{\delta}
\newcommand{\AF}{{\mathcal A}}
\newcommand{\zzz}{\mathfrak{Z}}
\newcommand{\iii}{{\mathrm i}}
\newcommand{\LieG}{\mathfrak{g}}
\newcommand{\bases}{\mathfrak{B}}
\newcommand{\bss}{\underline{\beta}}
\newcommand{\dtup}{\mathcal{X}}
\newcommand{\card}[1]{\lvert#1\rvert}

\newcommand{\ka}{\mathfrak{a}}
\newcommand{\kg}{\mathfrak{g}}
\newcommand{\kh}{\mathfrak{h}}
\newcommand{\kk}{\mathfrak{k}}
\newcommand{\ks}{\mathfrak{s}}
\newcommand{\kq}{\mathfrak{q}}
\newcommand{\kn}{\mathfrak{n}}
\newcommand{\CmP}{\mathcal{P}}
\newcommand{\CmU}{\mathcal{U}}
\newcommand{\Mid}{\operatorname{id}}

\newcommand{\CmO}{\mathcal{O}}
\newcommand{\CmN}{\mathcal{N}}

\newcommand{\km}{\mathfrak{m}}
\newcommand{\kv}{\mathfrak{v}}
\newcommand{\cpt}{\mathbf{K}}
\newcommand{\One}{\mathbf{1}}
\newcommand{\level}{\operatorname{level}}

\newcommand{\FP}{\operatorname{FP}}

\newcommand{\wX}{\widetilde{X}}
\newcommand{\1}{{\bf 1}}

\begin{document}

\title[]
{Analytic torsion for arithmetic locally symmetric manifolds and approximation of $L^2$-torsion}
\date{\today}

\author{Jasmin Matz}
\thanks{Second author partially supported by the Israel Science Foundation (grant no.\ 1676/17)}
\address{Department of Mathematical Science, University of Copenhagen, Universitetsparken 5, DK-2100 Copenhagen, Denmark}
\email{matz@math.ku.dk}

\author{Werner M\"uller}
\address{Universit\"at Bonn\\
Mathematisches Institut\\
Endenicher Allee 60\\
D -- 53115 Bonn, Germany}
\email{mueller@math.uni-bonn.de}


\begin{abstract}
In this paper we define a regularized version of the analytic torsion for 
quotients of a symmetric space of non-positive curvature by arithmetic lattices.
The definition is based on
the study of the renormalized trace of the corresponding heat operators, which 
is defined as the geometric side of the Arthur trace formula applied to the 
heat kernel. Then we  study the limiting behavior of the analytic torsion as the
lattices run through a sequence of congruence subgroups of a fixed arithmetic
subgroup. Our main result states that for sequences of principal congruence
subgroups, which converge to 1 at a fixed finite set of places and
strongly acyclic flat bundles, the logarithm of the analytic torsion, divided
by the index of the subgroup, converges to the $L^2$-analytic torsion.

\end{abstract}

\maketitle
\setcounter{tocdepth}{1}
\tableofcontents

\section{Introduction}

The main purpose of this paper is to define a regularized analytic torsion for
locally symmetric manifolds of finite volume which are quotients of a symmetric
space of non-positive curvature by an arithmetic group and to study the
limiting behavior of the renormalized analytic torsion for
sequences of finite coverings of a given locally symmetric space.
 This is motivated 
by the applications of the Ray-Singer analytic torsion \cite{RS} of a compact
manifold to the study of the growth of torsion in the cohomology of cocompact
arithmetic groups by Bergeron and Venkatesh \cite{BV}. Other applications
of a similar nature are given in  \cite{MaM}, \cite{MP4}. Since many important 
arithmetic groups are not cocompact, it is very desirable to extend these 
results to non-cocompact lattices. There exist some results for hyperbolic
manifolds of finite volume \cite{PR}, \cite{Ra1}, \cite{Ra2}, \cite{MR2}.

\subsection{Analytic torsion}
To explain the approach in more detail, we briefly recall the definition of 
the Ray-Singer
analytic torsion for a compact Riemannian manifold $X$ of dimension $n$. Let
$\rho\colon\pi_1(X)\to \GL(V)$ be a finite dimensional complex representation 
of the
fundamental group of $X$ let $E_\rho\to X$ be the flat vector bundle associated
to $\rho$. Choose a Hermitian fiber metric in $E_\rho$. Let $\Delta_p(\rho)$
be the Laplace operator on $E_\rho$-valued $p$-forms with respect to the metrics
on $X$ and in $E_\rho$. It is an elliptic differential operator, which is 
formally self-adjoint and non-nagative. Let 
$h_p(\rho):=\dim\ker\Delta_p(\rho)$. Using the trace of the
heat operator $e^{-t\Delta_p(\rho)}$, the zeta function $\zeta_p(s;\rho)$ 
of $\Delta_p(\rho)$ (see \cite{Sh}) can be defined by
\begin{equation}\label{zeta-fct}
\zeta_p(s;\rho):= \frac{1}{\Gamma(s)}\int_0^\infty 
\left(\Tr\left(e^{-t\Delta_p(\rho)}\right)-h_p(\rho)\right)t^{s-1} \;dt.
\end{equation} 
The integral converges for $\Re(s)>n/2$ and admits a meromorphic extension to
the whole complex plane, which is holomorphic at $s=0$. 
Then the analytic torsion $T_X(\rho)\in\R^+$, introduced by Ray and Singer 
\cite{RS}, is defined by
\begin{equation}\label{analtor0}
\log T_X(\rho)=\frac{1}{2}\sum_{p=1}^d (-1)^p p 
\frac{d}{ds}\zeta_p(s;\rho)\big|_{s=0}.
\end{equation}
Bergeron and Venkatesh \cite{BV} have applied the analytic torsion in the case 
of a locally symmetric space $X:=\Gamma\bs G/K$, where $G$ is a semisimple 
Lie group of non-compact type, $K$ a maximal compact subgroup and 
$\Gamma\subset G$ a torsion free, cocompact discrete subgroup. The 
representations $\rho$ of $\Gamma$ considered in this case, are of the form 
$\rho:=\tau|_\Gamma$, where $\tau$ is an irreducible finite-dimensional complex 
representation of $G$. Denote by $T_X(\tau)$ (resp. $T^{(2)}_X(\tau)$)
the analytic torsion (resp. the $L^2$-torsion \cite{Lo}, \cite{MV}) taken with 
respect to $\tau|_\Gamma$.
 Let $\wX:=G/K$ be the universal covering of $X$. Since the
heat kernels of the corresponding Laplacians on $\wX$ are $G$-invariant, one has
\begin{equation}\label{l2-tor}
\log T^{(2)}_X(\tau)=\vol(X) t^{(2)}_{\widetilde X}(\tau),
\end{equation}
where $t^{(2)}_{\widetilde X}(\tau)$ is a constant that depends only on 
$\widetilde X$ and $\tau$. One of the main results of \cite{BV} is the 
approximation of the $L^2$-torsion by the renormalized analytic torsion of
finite coverings of $X$ under a certain non-degeneracy condition on $\tau$.
Representations which satisfy this condition are called 
{\it strongly acyclic}. Let $X_i\to X$, $i\in\N$, be a sequence of finite coverings of $X$. 
Let $\tau$ be strongly acyclic. Let $\inj(X_i)$ denote the injectivity radius
of $X_i$ and assume that $\inj(X_i)\to\infty$ as $i\to\infty$. Then
by \cite[Theorem 4.5]{BV} one has
\begin{equation}\label{limittor}
\lim_{i\to\infty}\frac{\log T_{X_i}(\tau)}{\vol(X_i)}=t^{(2)}_{\widetilde
X}(\tau).
\end{equation} 
Let $\delta(\widetilde X):=\rk_\C(G)-\rk_\C(K)=1$. The constant 
$t^{(2)}_{\widetilde X}(\rho)$ has been computed in
\cite[Proposition 5.2]{BV}. As a result if follows that 
$t^{(2)}_{\widetilde X}(\rho)\neq0$ if and only if $\delta(\widetilde X)=1$. 
Combined with the equality of analytic torsion and Reidemeister torsion 
\cite{Mu2}, Bergeron and Venkatesh \cite{BV} used this result in the
case $\delta(\widetilde X)=1$ to show that 
torsion in the cohomology of cocompact arithmetic groups growth exponentially.

\subsection{The regularized trace of the heat operator}
The definition of the analytic torsion is based on the compactness of the 
underlying manifold. Without this assumption, the 
heat operator $e^{-t\Delta_p(\rho)}$ is in general not a trace class operator.
So in order to generalize the above method to non-compact manifolds, the first 
problem is to define an appropriate regularized trace of the heat operators.
For hyperbolic manifolds of finite volume a regularized trace has been
defined by means of the renormalized trace of the heat 
kernel. This method goes back to Melrose \cite{Me} and has been used, for
example, in \cite{Pa}, \cite{CV}, \cite{PR}, 
\cite{MP1}, \cite{MP3}, \cite{MP4}. The method is based on the truncation of 
the heat kernel with the help of an appropriate height function. The upshot is
that the renormalized trace of the heat kernel equals the geometric side of 
the Selberg trace formula applied to the heat kernel.

In the higher rank case we proceed in the same way as in the case of 
hyperbolic manifolds. To truncate the heat kernels we use Arthur's truncation
operator \cite{Ar1}. In \cite{MM1} we have dealt with the case  $\bG=\GL(n)$. 
As in the case of hyperbolic manifolds, the regularized trace
of the heat operator equals the
geometric side of the Arthur trace formula applied to the heat kernel. We use
this as the definition of the regularized trace for a general reductive group
$\bG$. Let us explain this approach in more detail.
To this end we pass 
to the adelic framework. For simplicity assume 
that $\bG$ is a connected semisimple algebraic group defined over $\Q$. Assume 
that $\bG(\R)$ is not compact. Let $\K_\infty$ be a maximal compact subgroup of 
$\bG(\R)$. Put $\widetilde X=\bG(\R)/\K_\infty$. Let $\A$ be
the ring of adeles of $\Q$ and $\A_f$ the ring of finite adeles. Let 
$K_f\subset \bG(\A_f)$
be an open compact subgroup. We consider the adelic quotient
\begin{equation}\label{adelic-quot}
X(K_f)=\bG(\Q)\bs (\widetilde X\times \bG(\A_f)/K_f).
\end{equation}
This is the adelic version of a locally symmetric space. In fact, 
$X(K_f)$ is the disjoint union of finitely many locally symmetric spaces, i.e.,
there exists lattices $\Gamma_i\subset \bG(\R)$, $i=1,\dots,l$, such that 
\begin{equation}\label{adelic-quot-1}
X(K_f)=\Gamma_1\bs\wX\sqcup\cdots\sqcup\Gamma_l\bs\wX.
\end{equation}
(see Section \ref{sec-analtor}). 
If $\bG$ is simply connected, then by strong approximation we have
\[
X(K_f)=\Gamma\bs \widetilde X,
\]
where $\Gamma=(\bG(\R)\times K_f)\cap \bG(\Q)$. We will assume that $K_f$ is 
neat
so that $X(K_f)$ is a manifold. Let $\nu\colon \K_\infty\to \GL(V_\nu)$ be a 
finite dimensional unitary representation. It induces a homogeneous 
Hermitian vector bundle $\widetilde E_\nu$ over $\widetilde X$,  which is 
equipped with the canonical connection $\nabla^\nu$. Being homogeneous,
 it can be pushed down to every connected component $\Gamma_i\bs \widetilde X$ 
of $X(K_f)$ resulting in a locally homogeneous vector bundle over $X(K_f)$. 
Let $\widetilde\Delta_\nu$ (resp. $\Delta_\nu$) be the associated
Bochner-Laplace operator acting in the space of smooth section of 
$\widetilde E_\nu$ (resp. $E_\nu$). Using the kernel of the
heat semigroup $e^{-t\widetilde\Delta_\nu}$, $t>0$, at the Archimedean place and
the normalized characteristic function at the finite places, we get a
smooth function $\phi_t^\nu$ on $\bG(\A)$ which belongs to $\Co(\bG(\A);K_f)$, 
the adelic version of the Schwartz space (see Section \ref{sec-prelim} for its 
definition). Let $J_{\geo}$ be the 
geometric side of the Arthur trace formula introduced in \cite{Ar1}. Following
\cite[Definition 11.1]{MM1}, we define the regularized trace of 
$e^{-t\Delta_\nu}$  by
\begin{equation}\label{reg-trace}
\Tr_{\reg}\left(e^{-t\Delta_\nu}\right):=J_{\geo}(\phi^\nu_t)
\end{equation}
 In order to define the zeta function 
by the analogous formula \eqref{zeta-fct} we need to determine the asymptotic
behavior of $\Tr_{\reg}\left(e^{-t\Delta_\nu}\right)$ as $t\to0$ and $t\to\infty$,
respectively. To this end we use the Arthur trace formula. 

Our first main result is concerned with the small time behavior of the 
regularized trace which is described in the following theorem.
\begin{theo}\label{thm:asymptotic:geom}
Let $\bG$ be a reductive algebraic group over $\Q$, $K_f\subset\bG(\A_f)$ an 
open compact subgroup, and $\nu:\K_\infty\longrightarrow \GL(V_\nu)$ a finite 
dimensional unitary representation of $\K_\infty$. Let $\Delta_\nu$ be the 
associated Bochner-Laplace operator on the adelic quotient $X(K_f)$. Suppose 
that $K_f$ is neat, and let $r$ be the split semisimple rank of $\bG$. 
There exist 
constants $a_j, b_{ij}\in\C$, $j\ge0$, $0\le i\le r$, depending on $\nu$ and 
$K_f$ such that
 \begin{equation}\label{asex}
  \Tr_{\text{reg}}(e^{-t \Delta_\nu })
  \sim t^{-d/2}\sum_{j=0}^\infty a_j t^j  
  + t^{-(d-1)/2} \sum_{j=0}^\infty \sum_{i=0}^r b_{ij} t^{j/2} (\log t)^{i}
 \end{equation}
as $t\searrow 0$. 
\end{theo}
For $\bG=\GL(n)$ or $\SL(n)$ this result was proved in \cite{MM1}, and for
hyperbolic manifolds in \cite{Mu4}.  In these cases the arguments are 
simplified considerably.

To study the large time behavior we restrict attention to twisted Laplace
operators, which are relevant for our purpose. 
As before, let $\tau\colon \bG(\R)^1\to \GL(V_\tau)$ be a finite dimensional complex 
representation. Let $\Gamma_i\bs \widetilde X$, $i=1,\dots,l$,  be the
components of $X(K_f)$. The restriction of $\tau$ to $\Gamma_i$ induces a flat 
vector bundle $E_{\tau,i}$ over $\Gamma_i\bs \widetilde X$. The disjoint union is
 a flat vector bundle $E_\tau$ over $X(K_f)$. By \cite{MM} it is isomorphic to 
the locally homogeneous vector bundle associated to $\tau|_{\K_\infty}$. It 
can be equipped with a fiber metric induced from the homogeneous 
bundle. Let $\Delta_p(\tau)$ be the corresponding twisted Laplace operator on
$p$-forms with values in $E_\tau$. Let $\Ad_{\pf}\colon \K_\infty\to \GL(\pf)$ 
be the adjoint representation of $\K_\infty$ on $\pf$, where $\pf=\kf^\perp$, and 
$\nu_p(\tau)=\Lambda^p\Ad_{\pf}^\ast\otimes\tau|_{\K_\infty}$. 
Up to a vector bundle endomorphism, $\Delta_p(\tau)$ equals the 
Bochner-Laplace operator $\Delta_{\nu_p(\tau)}$. So 
$\Tr_{\reg}\left(e^{-t\Delta_p(\tau)}\right)$ is well defined. The large time 
behavior of the regularized trace is described by the following proposition.

\begin{theo}\label{prop-lt}
Let $K_f\subset \bG(\A_f)$ be an open compact subgroup. Assume that $K_f$ is 
neat. Let $\tau\in\Rep(\bG(\R)^1)$ be irreducible. 
Assume that $\tau\not\cong\tau_\theta$. Then we have
\begin{equation}\label{largetime}
\Tr_{\reg}\left(e^{-t\Delta_p(\tau)}\right)=O(e^{-ct})
\end{equation}
as $t\to\infty$ for all $p=0,\dots,d$. 
\end{theo} 
\subsection{Regularized analytic torsion}
By Theorems~\ref{thm:asymptotic:geom} and \ref{prop-lt} we can define the zeta
function $\zeta_p(s;\tau)$ of $\Delta_p(\tau)$ as in \eqref{zetafct}, using the regularized
trace of $e^{-t\Delta_p(\tau)}$ in place of the usual trace in \eqref{zeta-fct}. The corresponding 
Mellin transform converges absolutely
and uniformly on compact subsets of the half-plane $\Re(s)>d/2$ and admits
a meromorphic extension to the whole complex plane by Theorems~\ref{thm:asymptotic:geom} and \ref{prop-lt}. Because of the presence
of the log-terms in the expansion \eqref{asex}, the zeta function may have a
pole at $s=0$ so that we need to use the finite part of $\zeta_p(s;\tau)$ in the definition of the analytic torsion of $X(K_f)$. More precisely, for a meromorphic function $f(s)$  on $\C$ and $s_0\in\C$, 
let $f(s)=\sum_{k\ge k_0}a_k(s-s_0)^k$
be the Laurent expansion of $f$ at $s_0$, and put $\FP_{s=s_0}f:=a_0$. Now we 
define
the analytic torsion $T_{X(K_f)}(\tau)\in\C\setminus\{0\}$ by
\begin{equation}\label{analtor}
\log T_{X(K_f)}(\tau)=\frac{1}{2}\sum_{p=0}^d (-1)^p p 
\left(\FP_{s=0}\frac{\zeta_p(s;\tau)}{s}\right).
\end{equation}

\begin{remark}
Let $\delta(\widetilde X)=\rk_\C G(\R)^1-\rk_\C {\bf K}_\infty$. We recall that in
the co-compact case, the analytic torsion is trivial, unless 
$\delta(\widetilde X)=1$. As the example of a hyperbolic manifold of even 
dimension shows \cite{MP1}, this does not need to be so in the non co-compact 
case.
\end{remark}
\subsection{Approximation of $L^2$-torsion}
Let $\{K_j\}_{j\in\N}$ be a sequence of open
compact subgroups of $\bG(\A_f)$. We say that $K_j\rightarrow 1$ as 
$j\to\infty$, if for every open compact subgroup $U$ of $\bG(\A_f)$ 
there exists $N\in\N$ such  that $K_j\subset U$ for all $j\ge N$.
We wish to study the limiting behavior of $\log T_{X(K_j)}(\tau)/\vol(X(K_j))$
as $j\to\infty$. 
The main goal is to generalize the results of 
\cite{MM2} to other reductive groups. For an adelic quotient $X:=X(K_f)$ let 
$T^{(2)}_X(\tau)$
be the $L^2$-torsion \cite{Lo}, \cite{MV}. Since the heat kernels on
$\wX$ are $G(\R)$-invariant, one has
\begin{equation}
\log T^{(2)}_X(\tau)=\vol(X) t^{(2)}_{\wX}(\tau),
\end{equation}
where $t^{(2)}_{\wX}(\tau)$ depends only on $\wX$ and $\tau$. 
 Based on the known results in the compact case \cite{BV}, we make 
the following conjecture.

\begin{conjecture}\label{conj1}
Let $\tau\in\Rep(G(\R)^1)$ be irreducible and assume that 
$\tau\not\cong\tau_\theta$. Let 
$\{K_{j}\}_{j\in\N}$ be a sequence of open compact subgroups of $G(\A_f)$
with $K_{j}\to 1$ as $j\to\infty$. Then
\begin{equation}
\lim_{j\to\infty}\frac{\log T_{X(K_j)}(\tau)}{\vol(X(K_j))}= t^{(2)}_{\wX}(\tau). 
\end{equation}
\end{conjecture}
In \cite{MM2} we established this conjecture for principal congruence subgroups
of $\GL(n)$ and $\SL(n)$. 
The constant $t^{(2)}_{\wX}(\tau)$ has been computed in
\cite[Proposition 5.2]{BV}. It is shown that $t^{(2)}_{\wX}(\tau)\neq0$ if and
only if $\delta(\wX)=1$.

Due to problems related to the geometric side of the trace formula, 
we are unable to prove Conjecture \ref{conj1} in full generality.
The problem is concerned with the global coefficients in
Arthur's fine expansion of the geometric side. For $\GL(n)$ estimations of
the global coefficients are known \cite{Ma1}.  For  groups other than
$\GL(n)$, currently very little is known about these coefficients. 
That is why at moment we can only deal with a restricted class of sequences of 
congruence groups which were considered in \cite[Sect. 2]{Cl}.
Let $S$ be a finite set of primes. Let $\{K_j\}_{j\in\N}$ be a
sequence of open compact subgroups of $G(\A_f)$. We say that $\{K_j\}$ 
converges to $1$ at $S$, denoted by $K_j\underset{S}{\rightarrow} 1$, if
there exists an open compact subgroup $K^S=\prod_{p\not\in S}K_p$ of $G(\A^S)$
such that
\begin{equation}\label{conv-s}
K_j=K_{S,j}\times K^S, \quad\text{with}\quad K_{S,j}=\prod_{p\in S}K_{p,j},\quad
\text{and}\quad K_{S,j}\underset{j\to\infty}{\longrightarrow} 1,
\end{equation}
where the latter condition means that for every open compact subgroup $U$
of $G_S:=\prod_{p\in S}G(\Q_p)$ there exists $N_0\in\N$ such that $K_{S,j}\subset
U$ for $j\ge N_0$. 
An example are the principal congruence subgroups $\Gamma_j:=\Gamma(p^j)$
of $\SL(n,\Z)$ of level $p^j$, where $p$ is a fixed prime. If $S=S_{\fin}$, 
the set of all finite places, we will simply write $K_j\to 1$.  

We need additional assumption which are related to the spectral side of the 
trace formula. Recall that in \cite[Section 5]{FiLaMu}, for a reductive group
$G$ we introduced and studied two natural properties, called (TWN) and (BD).
The abbreviations stand for {\it tempered winding numbers} and 
{\it bounded degree}, respectively. These properties concern the behavior of the
intertwining operators associated to proper parabolic subgroups of $\bG$. They
are trivially satisfied if $\bG$ is anisotropic modulo the center. By
\cite[Prop. 5.5, Theorem 5.15]{FiLaMu} they hold if $\bG$ equals
$\GL(n)$ or $\SL(n)$. In many other cases they have been established in 
\cite{FL2}, \cite{FL4}. 
The main result is the following
theorem.
\begin{theo}\label{theo-approx}
Let $G$ be a reductive algebraic group over $\Q$ which satisfies (TWN) and 
(BD). Let $\tau\in\Rep(G(\R)^1)$ be irreducible and assume that 
$\tau\not\cong\tau_\theta$. Let $S$ be a finite
set of primes. Let $\{K_j\}_{j\in\N}$ be a sequence of open
compact subgroups of $G(\A_f)$ with $K_j\underset{S}{\rightarrow} 1$ as 
$j\to\infty$. Let $X_j:=X(K_j)$. Then
\begin{equation}\label{approx-l2tor}
\lim_{j\to\infty}\frac{\log T_{X_j}(\tau)}{\vol(X_j)}=t^{(2)}_{\wX}(\tau).
\end{equation}
\end{theo}
As mentioned above, we have $t^{(2)}_{\wX}(\tau)\neq0$ if and only if 
$\delta(\wX)=1$.
Examples of groups with $\delta(\wX)=1$ are  $\SL(3)$, $\SL(4)$ and $\SO(p,q)$ 
with $pq$  odd. Examples of reductive groups that satisfy properties (TWN) are 
given in
\cite{FL2} and property (BD) is discussed in \cite{FL4}. Using
\cite[Theorem 3.11]{FL2} and \cite[Corollary 1]{FL4}, we get
\begin{cor}\label{cor-conv1} 
\begin{enumerate}
\item Let $\bG$ be an inner form of $\GL(n)$ or $\SL(n)$ defined over $\Q$.
Let $\tau\in\Rep(G(\R)^1)$ be irreducible satisfying $\tau\not\cong\tau_\theta$.
Let $S$ be a finite set of primes with $2\notin S$.  
Then 
\eqref{approx-l2tor} holds for every sequence $\{K_j\}_{j\in\N}$ of open subgroups
of a given open compact subgroup ${\bf K}_0$ of $\bG(\A_{f})$ satisfying
 $K_j\underset{S}{\rightarrow} 1$ as $j\to\infty$. 
\item Let $\bG$ be a quasi-split classical group defined over $\Q$.
Let $\tau\in\Rep(G(\R)^1)$ be irreducible satisfying $\tau\not\cong\tau_\theta$.
Let $S$ be a finite set of primes with $2\notin S$. 
Then \eqref{approx-l2tor}
holds for every sequence of open subgroups $\{K_j\}_{j\in\N}$ of a given open 
compact subgroup ${\bf K}_0$ of $\bG(\A_f)$ satisfying 
$K_j\underset{S}{\rightarrow} 1$ as $j\to\infty$.
\end{enumerate}
\end{cor}
We note that among the classical groups, the special orthogonal groups 
$\SO(p,q)$ with $pq$ odd satisfy $\delta(\widetilde X)=1$, which implies that
$t^{(2)}_{\widetilde X}(\tau)\neq 0$. Furthermore, $\SO(p,q)$ is quasi-split iff
$|p-q|\le 2$.

Suppose that  $\bG$ satisfies strong approximation. Then Theorem 
\ref{theo-approx} can be restated in terms of arithmetic subgroups 
of $\bG(\R)$. Let $K\subset \bG(\A_f)$ be an open compact subgroup of
$\bG(\A_f)$ and let $\Gamma_K:=G(\Q)\cap K\subset \bG(\R)$. Then $\Gamma_K$
is a discrete subgroup of $\bG(\R)$. By the
strong approximation theorem \cite[Theorem 7.12]{PlRa} it follows that the
$\bG(\R)$-spaces $G(\Q)\bs G(\A)/K$ and $\Gamma_K\bs G(\R)$ are 
canonically isomorphic. A subgroup $\Gamma\subset\bG(\Q)$ is called a
{\it congruence subgroup} if there exists an open compact subgroup $K_\Gamma$
of $\bG(\A_f)$ such that $\Gamma=\bG(\Q)\cap K_\Gamma$.

\begin{cor}\label{corr1}
Suppose that $\bG$ is a simply connected $\Q$-simple group satisfying (TWN) and
(BD) and that $\bG(\R)^1$ is not compact. Let ${\bf K}_0\subset \bG(\A_f)$
be a compact open subgroup and $\Gamma_0:=\bG(F)\cap {\bf K}_0$. Let 
$\{\Gamma_j\}_{j\in\N}$ be a sequence of congruence subgroups of $\Gamma_0$. 
Let $S$ be a finite set of primes and assume that $K_j\underset{S}{\rightarrow} 1$ as $j\to\infty$.
Let $\tau\in\Rep(\bG(\R)^1)$ be irreducible and assume that $\tau\not\cong
\tau_\theta$. Then 
\begin{equation}\label{approx-l2tor-2}
\lim_{j\to\infty}\frac{\log T_{\Gamma_j\bs\widetilde X}(\tau)}
{\vol(\Gamma_j\bs\widetilde X)}=t^{(2)}_{\wX}(\tau).
\end{equation}
\end{cor}
The groups $\SL(n)$ are simple and simply connected and satisfy properties 
(TWN) and (BD) \cite[Prop. 5.5, Theorem 5.15]{FiLaMu}. Moreover, for $n\ge 3$ 
every finite index subgroup of $\SL(n,\Z)$ is a congruence subgroup \cite{BLS}, 
\cite{Mn}. Thus for every finite index subgroup $\Gamma$ of $\SL(n,\Z)$, 
$n\ge 3$, there exists an open compact subgroup $K_\Gamma$ of $\bG(\A_f)$ such 
that $\Gamma=\bG(\Q)\cap K_\Gamma$. By Corollary \ref{corr1} we get
\begin{cor}
Let $\widetilde X_n:=\SL(n,\R)/\SO(n)$. Let $\tau\in\Rep(\SL(n,\R))$ be
irreducible and assume that $\tau\not\cong\tau_\theta$. Let $S$ be a finite set 
of primes. Let $\{\Gamma_j\}_{j\in\N}$ be a sequence of finite index subgroups
of $\SL(n,\Z)$ satisfying $K_{\Gamma_j}\underset{S}{\rightarrow} 1$ as 
$j\to \infty$. Then
\begin{equation}\label{approx-l2tor-3}
\lim_{j\to\infty}\frac{\log T_{\Gamma_j\bs\widetilde X_n}(\tau)}
{\vol(\Gamma_j\bs\widetilde X_n)}=t^{(2)}_{\wX_n}(\tau).
\end{equation}
\end{cor}
\begin{remark}
It is very desirable to remove in the results above the assumption that the
sequence of open compact subgroups $\{K_j\}$ converges to $1$ at $S$. 
This assunption is due to the insufficient knowledge of the coeffcients of
the fine expansion of the geometric side of Arthur's trace formula. For 
$\GL(n)$ there exist appropriate estimations of the coefficients \cite{Ma1}.
This was used in \cite{MM2} to established \eqref{approx-l2tor-3} 
for the sequence of principal congruence subgroups $\{\Gamma(N)\}_{N\in\N}$
of $\SL(n,\Z)$ without any restriction, which was deduced from the 
corresponding result for $\GL(n)$.  Hoffmann \cite{Ho} has described a 
conjectural approach to express the geometric side of the Arthur trace formula in terms of
zeta integrals attached to prehomogeneous vector spaces. A successful
completion of this program would lead to expressions of the coefficients in 
terms zeta functions attached to prehomogeneous vector spaces, which, in turn,
 would imply that the restriction on the sequences of congruence
subgroups can be lifted in general. For certain unipotent classes in $\GL(n)$ results have been proven in \cite{Ch}.
\end{remark}
The approach described by Hoffmann \cite{Ho} has been successfully carried out
 for classical groups of absolute rank up to two, especially for $\SL(3)$,
which is one of the cases where $t^{(2)}_{\wX}(\tau)\neq0$. 
For $\SL(3)$, the global coefficients are determined in \cite{HW} which can be 
used to get appropriate estimates for the global coefficients in the case of 
$\SL(3)$.  Thus for $n=3$  we get
\begin{cor}\label{cor:sl3}
Let $\widetilde X:=\SL(3,\R)/\SO(3)$. Let $\tau\in\Rep(\SL(3,\R))$ be
irreducible and assume that $\tau\not\cong\tau_\theta$. Then for every sequence
$\{\Gamma_j\}_{j\in\N}$ of finite index subgroups of $\SL(3,\Z)$ with 
$\vol(\Gamma_j\bs\widetilde X)\to\infty$ we have
\begin{equation}
\lim_{j\to\infty}\frac{\log T_{\Gamma_j\bs\widetilde X}(\tau)}
{\vol(\Gamma_j\bs\widetilde X)}=t^{(2)}_{\widetilde X}(\tau)>0.
\end{equation}
\end{cor}

\subsection{Method of proof}
Now we briefly explain our method to prove Theorems~\ref{thm:asymptotic:geom}
 and \ref{prop-lt}.  To determine the asymptotic behavior of the regularized 
trace as $t\to +0$, we start with its definition \eqref{reg-trace} as the 
geometric side of trace formula. Let $J_{\unip}(f)$ be the unipotent contribution
to the trace formula \cite{Ar85}. Then we show that $\phi_t^\nu$ 
can be replaced by a  compactly supported function 
$\widetilde \phi_t^\nu\in C_c^\infty(G(\A)^1)$ with support in a sufficiently 
small neighborhood of $1$ such that
\begin{equation}
\Tr_{\reg}\left(e^{-t\Delta_\nu}\right):=J_{\unip}(\widetilde\phi_t^\nu)+
O\left(e^{-c/t}\right)
\end{equation}
as $t\to +0$. To analyze $J_{\unip}(\widetilde\phi_t^\nu)$, we use the fine
geometric expansion \cite{Ar85} which expresses 
$J_{\unip}(\widetilde\phi_t^\nu)$ in terms of weighted orbital integrals. This
reduces the proof of Theorem~\ref{thm:asymptotic:geom} to the study of 
weighted orbital integrals. In \cite{MM1} we dealt with this problem for
the group $\GL(n)$. 
In this case all unipotent orbits are Richardson, which simplifies the 
analysis considerably. To deal with the weighted orbital integrals for an 
arbitrary reductive group $G$ we rely on \cite{Ar88}. Using
the proof of \cite[Corollary~6.2]{Ar88}, we obtain an appropriate integral
expression for the weighted orbital integrals. Again the main issue is the 
analysis of the weight function
and the proof that it has a certain logarithmic scaling behavior. Then,
as in the case of $\GL(n)$ we insert a standard parametrix for the heat kernel
into the weighted integral \eqref{eq:unipotent:combined} and determine its
asymptotic behavior as $t\to 0$. This leads to the proof of Theorem~\ref{thm:asymptotic:geom}. 

To prove Theorem~\ref{prop-lt}, we use the spectral side of the trace formula.
Let $\phi_t^{\tau,p}$ be the function in $\Co(G(\A)^1;K_f)$, which is defined in 
the same way
as $\phi_t^\nu$ in terms of the kernel of the heat operator $e^{-t\widetilde \Delta_p(\tau)}$ on the universal
covering. Then by the trace formula we have
\[
\Tr_{\reg}\left(e^{-t\Delta_p(\tau)}\right)=J_{\spec}(\phi^{\tau,p}_t).
\]
The key tool to deal with the spectral side is the refinement of the spectral
 expansion of the Arthur trace formula established in \cite{FLM1}
(see Theorem~\ref{thm-specexpand}). It states that
\[
J_{\spec}(f)=\sum_{[M]}J_{\spec,M}(f),\quad f\in\Co(G(\A)^1).
\]
Here $[M]$ runs over the conjugacy classes of Levi subgroups of $G$ and 
$J_{\spec,M}$ is a distribution associated to $M$. The term
associated to $G$ is $\Tr R_{\di}(f)$, where
$R_{\di}$ denotes the restriction of the regular representation of $G(\A)^1$ 
in $L^2(G(\Q)\bs G(\A)^1)$ to the discrete subspace. Using
our assumption that $\tau\not\cong\tau_\theta$, we obtain
$\dim\ker\Delta_p(\tau)=0$. Then it follows as in the compact case
that there exists $c>0$ such that 
\[
\Tr R_{\di}(\phi^{\tau,p}_t)=O(e^{-ct}), \quad\text{as}\;\; t\to\infty.
\]
For a proper Levi subgroup $M$ of $G$, $J_{\spec,M}(f)$ is an integral whose
main ingredient are logarithmic derivatives of intertwining operators. 
The determination of the asymptotic behavior of $J_{\spec,M}(\phi^{\tau,p}_t)$ as 
$t\to\infty$ relies on two properties, one global and one local, of the
intertwining operators. The global property is a uniform 
bound on the winding number of the normalizing factors of the intertwining
operators in the co-rank one case. The bound that we need is \eqref{log-deriv},
which was established in \cite[Theorem~5.3]{Mu2}. The local property is 
concerned with the estimation of integrals of logarithmic derivatives of 
normalized local intertwining operators $R_{\overline{P},P}(\pi_v,s)$, 
which are uniform in $\pi_v$. 

To prove Theorem \ref{theo-approx} we follow the approach used in \cite{MM2}
in the case of $\SL(n)$. The new ingredients are the results of \cite{FL2},
\cite{FL3} and \cite{FL4} concerning the spectral side of the trace formula.
These are estimations of the logarithmic derivatives of the global normalizing 
factors of the intertwining operators and bounds on the degrees of coefficients of local intertwining  operators.
At the moment,  we can treat the geometric side only for $\bG=\SL(3)$ and
$\bG=\GL(n)$ in all 
generality. For other reductive groups  estimations of the global 
coefficients are not available. This is why we need to make restrictions on the 
sequences of congruence subgroups. 

The paper is organized as follows. In Section \ref{sec-prelim} we fix notations 
and recall some basic facts. In Section \ref{sec-asyexp} we begin with the 
study of the asymptotic expansion of the regularized trace of the heat operator.
We show that for the derivation of the asymptotic expansion one can replace 
the geometric side of the trace formula by the unipotent contribution. 
Sections \ref{sec:unipotent}, \ref{sec:weighted} and \ref{sec:test:fcts}
contain some preparatory material related to weighted unipotent orbital 
integrals. In Section \ref{sec:asympt:weight:orb} we show that the weighted unipotent 
orbital integrals with respect to test functions derived from the heat kernel
admit an asymptotic expansion as $t\to0$. In Section \ref{sec:global}
we use this result combined with Arthur's fine geometric expansion to prove
Theorem~\ref{thm:asymptotic:geom}. In Section \ref{sec-trform} we recall the
the refined spectral expansion of Arthur's trace formula. Section 
\ref{sec-locint} is concerned with the study of logarithmic derivatives of
local intertwining operators. In Section \ref{sec-analtor} we use
the spectral side of the Arthur trace formula to prove Theorem~\ref{prop-lt}
which concerns the large time asymptotic behavior of the regularized trace of
the heat operator. 
Together with Theorem~\ref{thm:asymptotic:geom} this enables us to define the
regularized analytic torsion.  In the final section \ref{sec-towers} we 
study the limiting behavior of the renormalized logarithm of the analytic
torsion and prove Theorem \ref{theo-approx}.

\section{Preliminaries}\label{sec-prelim}

Let $\bG$ be a reductive algebraic group defined
over $\Q$. We fix a minimal parabolic subgroup $P_0$ of $\bG$ 
defined over $\Q$ and a Levi decomposition $P_0=M_0 U_0$, both defined 
over $\Q$.  Let $\cF$ be the set of parabolic subgroups of $G$ which contain 
$M_0$ and are defined over $\Q$. Let $\cL$ be the set of subgroups of $\bG$ 
which contain $M_0$ and are Levi components of groups in $\cF$. 
For any $P\in\cF$ we write
\[
P=M_PN_P,
\]
where $N_P$ is the unipotent radical of $P$ and $M_P$ belongs to $\cL$. 

Let $M\in\cL$. Denote by $A_M$ the $\Q$-split component of the center of $M$. 
Put $A_P=A_{M_P}$. Let $L\in\cL$ and assume that $L$ contains $M$. Then $L$ is
a reductive group defined over $\Q$ and $M$ is a Levi subgroup of $L$. We 
shall denote the set of Levi subgroups of $L$ which contain $M$ by $\cL^L(M)$.
We also write $\cF^L(M)$ for the set of parabolic subgroups of $L$, defined 
over $\Q$, which contain $M$, and $\cP^L(M)$ for the set of groups in $\cF^L(M)$
for which $M$ is a Levi component. Each of these three sets is finite. If 
$L=\bG$, we shall usually denote these sets by $\cL(M)$, $\cF(M)$ and $\cP(M)$.

Let $X(M)_\Q$ be the group of characters of $M$ which are defined over $\Q$. 
Put
\begin{equation}\label{liealg}
\af_{M}:=\Hom(X(M)_\Q,\R).
\end{equation}
This is a real vector space whose dimension equals that of $A_M$. Its dual 
space is
\[
\af_{M}^\ast=X(M)_\Q\otimes \R.
\]
 We shall write, 
\begin{equation}\label{liealg1}
\af_P=\af_{M_P},\;A_0=A_{M_0}\quad\text{and}\quad \af_0=\af_{M_0}.
\end{equation}
For $M\in\cL$ let $A_M(\R)^0$ be the connected component of the identity of
the group $A_M(\R)$. 
Let $W_0=N_{\bG(\Q)}(A_0)/M_0$ be the Weyl group of $(\bG,A_0)$,
where $N_{\bG(\Q)}(H)$ is the normalizer of $H$ in $\bG(\Q)$.
For any $s\in W_0$ we choose a representative $w_s\in \bG(\Q)$.
Note that $W_0$ acts on $\levis$ by $sM=w_s M w_s^{-1}$. For $M\in\cL$ let
$W(M)=N_{\bG(\Q)}(M)/M$, which can be identified with a subgroup of $W_0$.

For any $L\in\cL(M)$ we identify $\af_L^\ast$ with a subspace of $\af_M^\ast$.
We denote by $\af_M^L$ the annihilator of $\af_L^\ast$ in $\af_M$. Then $r=\dim \af_0^{\bG}$ is the semisimple rank of $\bG$.
We set
\[
\levis_1(M)=\{L\in\levis(M):\dim\aaa_M^L=1\}
\]
and
\begin{equation}\label{f1}
\cF_1(M)=\bigcup_{L\in\levis_1(M)}\cP(L).
\end{equation}
We shall denote the simple roots of $(P,A_P)$ by $\Delta_P$. They are
elements of $X(A_P)_\Q$ and are canonically embedded in $\af_P^\ast$. Let
$\Sigma_P\subset \af_P^\ast$ be the set of reduced roots of $A_P$ on the Lie
algebra of $\bG$. For any $\alpha\in\rts_M$ we denote by $\alpha^\vee\in\aaa_M$
the corresponding co--root. Let $P_1$ and $P_2$ be parabolic subgroups with
$P_1\subset P_2$. Then $\af_{P_2}^\ast$ is embedded into $\af_{P_1}^\ast$, while
$\af_{P_2}$ is a natural quotient vector space of $\af_{P_1}$. The group
$M_{P_2}\cap P_1$ is a parabolic subgroup of $M_{P_2}$. Let $\Delta_{P_1}^{P_2}$
denote the set of simple roots of $(M_{P_2}\cap P_1,A_{P_1})$. It is a subset
of $\Delta_{P_1}$. For a parabolic subgroup $P$ with $P_0\subset P$ we write
$\Delta_0^P:=\Delta_{P_0}^P$. 

Let $\A$ (resp. $\A_f$) be the ring of adeles (resp. finite adeles) of $\Q$. 
We fix a maximal compact subgroup $\K=\prod_v \K_v = \K_\infty\cdot 
\K_{f}$ of $\bG(\A)=\bG(\R)\cdot \bG(\A_{f})$. We assume that the maximal 
compact subgroup $\K \subset \bG(\A)$ is admissible with respect to 
$M_0$ \cite[\S 1]{Ar88}. Let $\Ht_M: M(\A)\rightarrow\aaa_M$ be the 
homomorphism given by
\begin{equation}\label{homo-M}
e^{\sprod{\chi}{\Ht_M(m)}}=\abs{\chi (m)}_\A = \prod_v\abs{\chi(m_v)}_v
\end{equation}
for any $\chi\in X(M)$ and denote by $M(\A)^1 \subset M(\A)$ the kernel 
of $\Ht_M$. 

Let $\gf$ and $\kf$ denote the Lie algebras of $\bG(\R)$ and $\K_\infty$,
respectively. Let $\theta$ be the Cartan involution of $\bG(\R)$ with respect to
$\K_\infty$. It induces a Cartan decomposition $\mathfrak{g}= 
\mathfrak{p} \oplus \mathfrak{k}$. 
We fix an invariant bi-linear form $B$ on $\mathfrak{g}$ which is positive 
definite on $\mathfrak{p}$ and negative definite on $\mathfrak{k}$.
This choice defines a Casimir operator $\Omega$ on $\bG(\R)$,
and we denote the Casimir eigenvalue of any $\pi \in \Pi (\bG(\R))$ by 
$\lambda_\pi$. Similarly, we obtain
a Casimir operator $\Omega_{\K_\infty}$ on $\K_\infty$ and write $\lambda_\tau$ for 
the Casimir eigenvalue of a
representation $\tau \in \Pi (\K_\infty)$ (cf.~\cite[\S 2.3]{BG}).
The form $B$ induces a Euclidean scalar product $(X,Y) = - B (X,\theta(Y))$ on 
$\mathfrak{g}$ and all its subspaces.
For $\tau \in \Pi (\K_\infty)$ we define $\norm{\tau}$ as in 
\cite[\S 2.2]{CD}. Note that the restriction of the scalar product 
$(\cdot,\cdot)$ on $\gf$ to $\af_0$ gives $\af_0$ the structure of a 
Euclidean space. In particular, this fixes Haar measures on the spaces 
$\af_M^L$ and their duals $(\af_M^L)^\ast$. We follow Arthur in the 
corresponding normalization of Haar measures on the groups $M(\A)$ 
(\cite[\S 1]{Ar1}). 

Let $L^2_{\disc}(\Ai M(\Q)\bs M(\A))$ be the discrete part of 
$L^2(\Ai M(\Q)\bs M(\A))$, i.e., the
closure of the sum of all irreducible subrepresentations of the regular 
representation of $M(\A)$.
We denote by $\Pi_{\disc}(M(\A))$ the countable set of equivalence classes of 
irreducible unitary
representations of $M(\A)$ which occur in the decomposition of the discrete 
subspace $L^2_{\disc}(\Ai M(\Q)\bs M(\A))$ into irreducible representations.

Let $H$ be a topological group. We will denote by $\Pi(H)$ the set of equivalence
classes of irreducible unitary representations of $H$.

Next we introduce the space $\Co(G(\A)^1)$ of Schwartz functions. 
For any compact open subgroup
$K_f$ of $G(\A_f)$ the space $G(\A)^1/K_f$ is the countable disjoint union of
copies of $G(\R)^1=G(\R)\cap G(\A)^1$ and therefore, it is a differentiable
manifold. Any element $X\in\mathcal{U}(\gf^1_\infty)$ of the universal 
enveloping algebra of the Lie algebra $\gf_\infty^1$ of $G(\R)^1$ defines a
left invariant differential operator $f\mapsto f\ast X$ on $G(\A)^1/K_f$. Let
$\Co(G(\A)^1;K_f)$ be the space of smooth right $K_f$-invariant functions on
$G(\A)^1$ which belong, together with all their derivatives, to $L^1(G(\A)^1)$.
The space $\Co(G(\A)^1;K_f)$ becomes a Fr\'echet space under the seminorms
\[
\|f\ast X\|_{L^1(G(\A)^1)},\quad X\in\mathcal{U}(\gf^1_\infty).
\]
Denote by $\Co(G(\A)^1)$ the union of the spaces $\Co(G(\A)^1;K_f)$ as $K_f$ 
varies over the compact open subgroups of $G(\A_f)$ and endow 
$\Co(G(\A)^1)$ with the inductive
limit topology.

\section{Asymptotic expansion of the regularized trace}\label{sec-asyexp}
\setcounter{equation}{0}
Let $\bG$ be a reductive quasi-split group over $\Q$. We assume that its center $Z_{\bG}$ is $\Q$-split and let $A_{\bG}$ be the identity component of $Z_{\bG}(\R)$. Then $\bG(\A)= \bG(\A)^1\times A_{\bG}$ and $\bG(\R)=\bG(\R)^1\times A_{\bG}$ with $\bG(\R)^1= \bG(\A)^1\cap \bG(\R)$, a semisimple real Lie group. Let $\theta:\bG(\R)\longrightarrow \bG(\R)$ be a Cartan involution and $\K_\infty= \bG(\R)^{\theta}$ its fixed points. For each prime $p$ fix some maximal compact subgroup $\cpt_p$ of $\bG(\Q_p)$ and let $\cpt_f= \prod_{p}\cpt_p$. Let $K_f$ be a finite index subgroup of $\cpt_f$ and $X(K_f)=\bG(\Q)\backslash \bG(\A)^1/\K_\infty^0\cdot K_f$, where $\K_\infty^0$ is the connected identity component of $\K_\infty$. Let $r$ denote the split semisimple rank of $\bG$ so that $r=\dim\af_0^{\bG}$. 

We recall the definition of  the regularized trace. For that we adopt the 
notation from \cite[\S 11-12]{MM1}.
Let $\nu\colon \K_\infty\to \GL(V_\nu)$ be a finite dimensional unitary 
representation. Let $\tilde\Delta_\nu$ be the Bochner-Laplace operator  
attached to $\nu$ on the universal covering $\tilde X= \bG(\R)^1/\K_\infty^0$ of 
$X(K_f)$. Let $H_t^\nu: \bG(\R)^1\longrightarrow \GL(V_\nu)$ be the convolution 
kernel associated with $\tilde\Delta_\nu$, and let $h_t^\nu=\Tr H_t^\nu$. 
We extend $h_\nu^t$ to $G(\R)$ by $h_t^\nu(ag)=h_t^\nu(g)$ for all $a\in A_{\bG}$, $g\in \bG(\R)^1$. Let $\One_{K_f}: \bG(\A_f)\longrightarrow\C$ be the 
characteristic function of $K_f$. Put 
\begin{equation}\label{char-funct}
\chi_{K_f}:=\frac{\One_{K_f}}{\vol(K_f)}
\end{equation}
and
\[
 \phi^\nu_t(g)
 = h_t^\nu(g_\infty) \chi_{K_f}(g_f)
\]
for $g=g_\infty\cdot g_f\in \bG(\A)= \bG(\R)\cdot \bG(\A_f)$. 
Let $J_{\geo}$ denote the geometric side of Arthur's trace formula. The regularized trace of $e^{-t\Delta_\nu}$ is defined by
\begin{equation}\label{eq:def:reg:trace}
 \Tr_{\text{reg}}(e^{-t\Delta_\nu})
 = J_{\geo}(\phi_t^\nu).
\end{equation}
This is well-defined because $\phi_t^\nu$ and all its derivatives are in $L^1(\bG(\A)^1)$ so that $J_{\geo}(\phi_t^\nu)$ is well-defined by \cite{FLM1}.

\subsection{Reduction to unipotent distributions}
The proof of Theorem~\ref{thm:asymptotic:geom} rests on an asymptotic expansion of certain unipotent distributions $J_M^G(\CmO,\cdot)$, which will be introduced in \S \ref{sec:weighted} and  which are defined only for compactly supported  test functions. To state this result we first need to construct compactly supported test functions from $\phi_t^\nu$.

Let $d(\cdot, \cdot): \tilde X\times\tilde X\longrightarrow [0,\infty)$ be the geodesic distance on $\tilde X$, and put $r(g_\infty)=d(g_\infty x_0, x_0)$ where $x_0=\K_\infty\in \tilde X$ is the base points. Let $0<a<b$ be sufficiently small real numbers and let $\beta:\R\longrightarrow [0,\infty)$ be a smooth function supported in $[-b,b]$ such that $\beta(y)=1$ for $0\le |y|\le a$, and $0\le\beta(y)\le1$ for $|y|>a$. Define 
\begin{equation}\label{eq:def:psi}
 \psi^\nu_t(g_\infty) 
 = \beta(r(g_\infty)) h_t^\nu(g_\infty).
\end{equation}
and
\begin{equation}\label{eq:def:phi:tilde}
 \tilde\phi^\nu_t(g)
 = \psi^\nu_t(g_\infty)\chi_{K_f}(g_f)
\end{equation}
for $g=g_\infty\cdot g_f\in G(\A)=G(\R)\cdot G(\A_f)$.
Then $\tilde\phi^\nu_t\in C_c^\infty(\bG(\A)^1)$ and $\psi^\nu_t\in C^\infty_c(\bG(\R)^1)$.
By \cite[Proposition~12.1]{MM1} there is some $c>0$ such that for every $0<t\le 1$ we have
\begin{equation}\label{eq:geom:approx}
 \left|J_{\geo}(\phi_t^\nu)- J_{\geo}(\tilde\phi_t^\nu)\right|
 \ll e^{-c/t}.
\end{equation}
We note that in \cite[Sect. 12]{MM1} we made the assumption that $\bG=\GL(n)$
or $\bG=\SL(n)$. However, the proof of the  proposition holds without any 
restriction on $\bG$. The next result reduces the considerations to the 
unipotent contribution to the geometric side.
Before we state it, we recall the coarse geometric expansion of Arthur's trace formula \cite[\S 10]{Ar05}: Two elements $\gamma_1, \gamma_2\in G(\Q)$ are called coarsely equivalent if their semisimple parts (in the Jordan decomposition)  are conjugate in $\bG(\Q)$. Then for any $f\in C^\infty_c(\bG(\A)^1)$ we have
\[
 J_{\geo}(f)=\sum_{\of} J_{\of} (f),
\]
where $\of$ runs over the coarse equivalence classes in $\bG(\Q)$, and the distribution $J_{\of} $ is supported in the set of all $g\in \bG(\A)^1$ whose semisimple part is conjugate in $\bG(\A)$ to some semisimple element in $\of$. If $\of\neq\of'$, the supports of $J_{\of}$ and $J_{\of'}$ are disjoint.
Note that the set of unipotent elements in $\bG(\Q)$ constitute a single equivalence class $\of_{\text{unip}}$ and we write $J_{\text{unip}}=J_{\of_{\text{unip}}}$.

\begin{prop}\label{prop:replace:geom:unip}
If $K_f$ is neat and the support of $\beta$ is sufficiently small, then
\begin{equation}\label{eq:geom:unip}
 J_{\geo}(\tilde\phi_t^\nu)
 = J_{\text{unip}}(\tilde\phi_t^\nu).
\end{equation}
\end{prop}
\begin{proof}
Let $\rho:\bG\longrightarrow\GL(N)$ be a faithful representation of $\bG$. For each prime $p$  we can find $\nu_p\ge 0$ such that $\cpt_p\subseteq \rho^{-1}(\GL_N(p^{-\nu_p}\Z_p))$, and $\nu_p=0$ for all but finitely many $p$. Hence $K_f\subseteq \rho^{-1}(\GL_N(M^{-1}\hat{\Z}))$ with $M=\prod p^{\nu_p}$.

Let $\chi:\GL_N(\A)\longrightarrow \A^N$ be defined by mapping elements of $\GL_N(\A)$ onto the sequence of coefficients of their characteristic polynomials (omitting the coefficient $1$ of the highest degree monomial).
Let $f_\infty\in C^\infty_c(\bG(\R)^1)$. 
Suppose that $g\in \bG(\A)^1$ is in  the support of $f_\infty\cdot\One_{K_f}$ and that the semisimple part of $g$ is conjugate in $\bG(\A)$ to some $\sigma\in \bG(\Q)$. Then $\chi(\rho(g))=\chi(\rho(\sigma))\in\Q^N$, and further $\chi(\rho(g))\in p^{-N\nu_p}\Z_p^N$ for every prime $p$. Hence $\chi(\rho(g))\in M^{-N}\Z^N$. 
If we choose $f_\infty$ to be supported in a sufficiently small, bi-$\K_\infty$-invariant neighborhood of the identity of $\bG(\R)^1$ (the support of $f_\infty$ will possibly depend on $M$), we can arrange that $\chi(\rho(g))\in \Z^N$ so that the eigenvalues of $\rho(g)$ are all algebraic integers. Shrinking the support of $f_\infty$ even further if necessary, we can conclude the eigenvalues of $\rho(g)$ must all be roots of unity. Otherwise, the group generated by the matrix
$\rho(g)$ would not be contained in a compact set. 
By assumption, $K_f$ is neat so that these eigenvalues in fact need to be 
equal to $1$. The semisimple part of $\rho(g)$ therefore equals $I_N$ (the identity matrix in $\GL_N(\Q)$), that is, the semisimple part of $g$ equals the identity in $\bG(\A)^1$ so that $g$ is unipotent.

By the discussion above on the coarse geometric expansion, we can now find a bi-$\K_\infty$-invariant neighborhood $\Omega$ of the identity in $\bG(\R)^1$ such that whenever $f_\infty\in C^\infty_c(\bG(\R)^1)$ is supported in $\Omega$ we have $J_{\geo}(f)= J_{\text{unip}}(f)$ where $f=f_\infty\cdot\One_{K_f}\in C^\infty_c(G(\A)^1)$. Hence if we choose the support of $\beta$ in the definition of $\tilde\phi_t^\nu$ sufficiently small, we obtain the proposition.
\end{proof}

In light of \eqref{eq:geom:approx} and \eqref{eq:geom:unip} we therefore only need to study the asymptotic expansion of $J_{\text{unip}}(\tilde\phi_t^\nu)$ as $t\searrow0$ to prove Theorem~\ref{thm:asymptotic:geom}. This will be done in the following sections by using the fine geometric expansion of $J_{\text{unip}}$ involving weighted orbital integrals over the unipotent conjugacy classes in $\bG(\R)$.

\section{Preliminaries on unipotent conjugacy classes and integrals}
\label{sec:unipotent}
  Until \S \ref{sec:global} we will be concerned only with the real Lie group $\bG(\R)^1$ so that we write $G_\infty$ for $\bG(\R)^1$.
  
  \subsection{Notation} In abuse of our previous notation, we write $P_0=M_0U_0\subseteq G_\infty$ for the minimal parabolic $P_0(\R)\cap G_\infty= (M_0(\R)\cap G_\infty) (U_0(\R)\cap G_\infty)$ in $G_\infty$ until \S \ref{sec:global}.
  
  A parabolic subgroup $P$ of $G_\infty$ is called standard if it contains $P_0$, and semistandard if it contains $M_0$. A Levi subgroup $M$ in $G_\infty$ is called semistandard if it equals the Levi component containing $M_0$ of some semistandard parabolic subgroup. We write $\CmL$ for the set of semistandard Levi subgroups of $G_\infty$. If $M\in \CmL$, we write $\CmL(M)$ for the set of all $L\in\CmL$ with $M\subseteq L$ so that $\CmL=\CmL(M_0)$. 
  
  If $L\in\CmL$, then $P_0^L:=P_0\cap L = M_0 (U_0\cap L)$ is a minimal parabolic subgroup in $L$, and $P\mapsto P\cap L$ defines a surjective map from standard parabolic subgroups in $G_\infty$ to standard parabolic subgroups in $L$ (with respect to $P_0^L$). Similarly, we get surjective maps from semistandard parabolic and semistandard Levi subgroups in $G_\infty$ to such subgroups in $L$ (with respect to $M_0$).  If $M, L\in\CmL$, $M\subseteq L$, we write $\CmL^L(M)$ for the semistandard Levi subgroups in $L$ containing $M$. Further, we write $\CmF^L(M)$ for the set of all semistandard parabolic subgroups in $L$ containing $M$. If $L=G_\infty$, we write $\CmF(M)=\CmF^{G_\infty}(M)$. Though this clashes with our global notation previously used, we hope that it will not lead to any confusion.

 \subsection{Unipotent conjugacy classes}\label{sec:unip:cl}  
 We recall some basic facts on unipotent conjugacy classes, which can for example be found in \cite{Carter} or \cite{CoMc}.
 Let $M\in \CmL$ and let $\CmO\subseteq M$ be a unipotent $M$-conjugacy class in $M$. If $L\in\CmL(M)$, we write $\CmO^L$ for the unipotent conjugacy class induced from $\CmO$ in $M$ to $L$ along some semistandard parabolic subgroup in $L$ (the induced class is independent of that choice of parabolic).

Let $L\in\CmL(M)$.  Let $P^L\subseteq L$ be a Jacobson--Morozov parabolic associated with $\CmO^{L}$ in $L$ (see \cite[Remark 3.8.5]{CoMc}). We can choose $P^L$ to be standard, and we write $P^L= M^L U^L$ for its Levi decomposition with $M^L\supseteq M_0$.  
 
Let $\lf$ denote the Lie algebra of $L$, and let 
\[
 \lf=\bigoplus_{i\in \Z}\lf_{i}
\]
be the grading attached to the standard triple corresponding to our choice of Jacobson--Morozov parabolic $P^L$. Put $\uf^L_{j}:=\bigoplus_{i>j} \lf_{i}$. Let $X_0\in \lf_{2}$ such that $u_0:=e^{X_0}\in \CmO^{L}$.

 \subsection{Measures on unipotent classes} 
 We keep the notation from \S \ref{sec:unip:cl}. 
 To define distributions on the unipotent conjugacy classes, we need to fix measures. We fix once and for all some normalization of measures on $G_\infty$, on 
$\K_\infty$, on the semistandard Levi subgroups in $G_\infty$, and on the unipotent radicals of the semistandard parabolic subgroups. We choose those measures such that they are compatible with respect to Iwasawa decomposition. We also fix a normalization of the measures on the vector spaces $\lf_i$.

Let $L_{u_0}$ be the centralizer of $u_0$ in $L$. Then $L_{u_0}\backslash L$ is diffeomorphic to $\CmO^L$ and $L_{u_0}$ is unimodular being a unipotent group. The quotient measure on $L_{u_0}\backslash L$ (denoted by $d^*g$) defines an $L$-invariant measure on $\CmO^L$ which in fact is a Radon measure on $\CmO^L$ and has an explicit description: 
 \begin{prop}[\cite{Rao}]
There exists $c>0$ and a polynomial $\varphi: \lf_2\longrightarrow \C$ of degree $\dim\lf_1$ such that if $f\in C_c^\infty(L)$, then 
 \begin{equation}\label{eq:unipotent:orbital}
  \int_{L_{u_0}\backslash L} f(g^{-1} u_0 g)\, d^*g
  = c\int_{V_0}\int_{\uf^L_2} f_{K^L_\infty}(e^{X+Z}) |\varphi(X)|^{1/2}\, dZ\, dX
 \end{equation}
where $V_0\subseteq \lf_2$ is the orbit of $X_0$ under $M^L$ (a dense 
subset of $\lf_2$), and $f_{K^L_\infty}$ is defined by $f_{K^L_\infty}(g):
=\int_{K^L_\infty} f(k^{-1} g k)\, dk$, $g\in L$.
\end{prop}

If we want to emphasize the dependence on $L$, we write $V_0^L$, $\lf^L$, $\varphi^L$ etc. If $Q=LV$ is a semistandard parabolic subgroup, we might also write $V_0^Q:=V_0^{L}$, $\lf^Q:=\lf^L$, $\varphi^Q:=\varphi^L$ etc.

\subsection{Behavior under induction}
Let $\CmO\subseteq M$, $L\in \CmF(M)$ be as before. We can induce in stages, that is, $\CmO^{G_\infty} = \left(\CmO^L\right)^{G_\infty}$.
The invariant measure on $\CmO^{G_\infty}$ is then given by a constant multiple of 
\begin{equation}\label{eq:measure:induced}
 \int_{N_Q}\int_{V_0^L}\int_{\uf^L_2} f_{\K_\infty}(e^{X+Z} n) 
|\varphi^L(X)|^{1/2}\, dZ\, dX\, dn
\end{equation}
for any $f\in C_c^\infty(G_\infty)$ where $Q\in \CmF(L)$ is such that $L=L_Q$ is the Levi component of $Q$, $N_Q$ the unipotent radical of $Q$, and 
$f_{\K_\infty}(g)=\int_{\K_\infty}(k^{-1}gk)\, dk$.

\begin{remark}\label{rem:dimensions}
The dimension of a unipotent orbit can be computed in terms of the dimensions of the grading coming from the attached standard triple. 
More precisely, $\dim \CmO^L= 2 \dim \uf_1^L + \dim \lf_1$, see \cite[Lemma~4.1.3]{CoMc}. Taking into account that $\dim V_0^L=\dim \lf_2$, and 
$\deg \varphi^L = \dim \lf_1$  we get
\[
 \dim \CmO^L = 2\left(\dim \uf_2^L + \dim V_0^L \right) + \deg\varphi^L,
\]
in particular, the dimension is even.
Suppose that $Q\in\CmF(M)$ is any semistandard parabolic subgroup with Levi component $L$ and unipotent radical $N$, $Q=LN$.
Then $\dim \CmO^{G_\infty}= 2\dim N + \dim\CmO^L$ so that
\[
\dim\CmO^{G_\infty} = 2\left(\dim N + \dim \uf_2^L + \dim V_0^L\right)  + \deg\varphi^L.
\]
\end{remark}

\section{Weighted unipotent integrals}\label{sec:weighted}
\subsection{Introduction} Arthur's fine geometric expansion and his splitting formula (see \S\S\ref{sec:fine:exp}-\ref{sec:splitting}) describe $J_{\text{unip}}$ as a linear combination of certain products of real and $p$-adic weighted unipotent integrals. For our purposes we only need to be concerned with the archimedean case for which we follow \cite{Ar88}: Let $f\in C^\infty_c(G_\infty)$ and let $\CmO$ be a unipotent conjugacy class in $M$. The archimedean weighted orbital integrals $J_M^{G_\infty}(f,\CmO)$ can be defined as sum of integrals over $\CmO^{G_\infty}$ against certain non--invariant measure.
Those non--invariant measures can be described as follows: Using the proof of \cite[Corollary~6.2]{Ar88} we have
\begin{equation}\label{eq:unipotent:weighted}
 J_M^{G_\infty}(f, \CmO)
 = \sum_{Q\in\CmF(M)} c(Q,\CmO) \int_{N_Q} \int_{V_0^{Q}} \int_{\uf^{Q}_2} 
f_{\K_\infty}(e^{X+Z} n) w_{M,\CmO}^Q(e^{X+Z}) |\varphi^{Q}(X)|^{1/2}\, dZ\, dX \, dn\, dk,
\end{equation}
where the notation is as follows:
\begin{itemize}
\item $f_{\K_\infty}(g)=\int_{\K_\infty} f(k^{-1}gk)\, dk$   
\item $w_{M,\CmO}^Q$ is a certain weight function discussed in \cite{Ar88}. We will study this weight function in more detail below, 
\item $c(Q,\CmO)>0$ are suitable constants coming from the normalization of measures in \eqref{eq:unipotent:orbital} and \eqref{eq:measure:induced}, 
\item for $Q\in\CmF(M)$, $Q=L_QN_Q$  denotes its  Levi 
 decomposition with $L_Q$ its Levi component, $L_Q\supseteq M$, and $N_Q$ its 
unipotent radical.
\end{itemize}
Set $\uf_{Q,1}= \uf_1^{L_Q}\oplus \nf_Q$ with $\nf_Q$ the Lie algebra of $N_Q$. We  extend $\varphi^{L_Q}$ and $w_{M,\CmO}^Q(\exp(\cdot))$ to all of $\uf_{Q,1}$ by projecting to $\uf_1^{L_Q}$ along $\nf_Q$. Then we can write the integral above also as
\begin{equation}\label{eq:unipotent:combined}
 J_M^{G_\infty}(f, \CmO)
 = \sum_{Q\in\CmF(M)} c(Q,\CmO) \int_{\uf_{Q,1}} f_{\K_\infty}(e^Y) w_{M,\CmO}^Q(e^Y) |\varphi^{L_Q}(Y)|^{1/2} \, dY.
\end{equation}
Note that for each $Q$, the integral is over the same unipotent orbit $\CmO^{G_\infty}$ but with different weight functions.

\subsection{An asymptotic expansion of the weights}\label{sec:setup}
 We now study the functions $w_{M,\CmO}^Q(\cdot)$ from \eqref{eq:unipotent:weighted} in more detail. For convenience of notation, we only consider the case $Q=G_\infty$. We write $w_{M,\CmO}= w_{M,\CmO}^{G_\infty}$.

We fix an embedding $\iota : G_\infty\hookrightarrow \GL_n(\R)^1$ that satisfies certain properties. Write $H=\iota(G_\infty)$, $S=\iota(T(\R))$, $N_0= \iota(U(\R))$, where $T(\R)$ is a maximal split torus in $G_\infty$, and $N_0(\R)$ the unipotent radical of our fixed minimal parabolic subgroup of $G_\infty$. Then we assume $\iota$ to satisfy the following:
\begin{itemize}
\item $H$ is self-adjoint;
 \item $S$ is contained in the group of diagonal matrices $T_0\subseteq \GL_n(\R)^1$;
 \item $N_0$ is contained in the group of unipotent upper triangular matrices $U_0\subseteq \GL_n(\R)^1$;
 \item The restriction of the positive roots of $(T_0,U_0)$ to $S$ are positive roots of $(S,N_0)$; every root of $(S,N_0)$ is obtained this way.
\end{itemize}
The existence of such an embedding follows from \cite[Proposition 3.13]{PlRa}. In the following we will write $\kh$ for the Lie algebra of $H$, $G_n=\GL_n(\R)^1$ and $\kg_n$  for the Lie algebra of $G_n$.

Let $P=MN$ be a semi-standard parabolic subgroup in $H$ and let $\CmO\subseteq M$ be a unipotent conjugacy class. Let $P_1=MN_1$ be another semi-standard parabolic subgroup with the same Levi component $M$.
Let $\varpi$ be a weight on $\ka_M/\ka_H$ that is an extremal weight for an 
irreducible representation $\Lambda_\varpi$ of $H$ on a finite dimensional vector space $V_{\varpi}$, defined over $\R$, which 
is also $P$-dominant. Let $H_P:H\longrightarrow\ka_M$ be the Iwasawa 
projection.

Then for any $h\in H$,  Arthur defines  a weight function 
$v_P(\varpi, h)$ by $v_P(\varpi, h)= e^{-\varpi(H_P(h))}$ and as shown in 
\cite[(3.3)]{Ar88}, it satisfies
\[
v_P(\varpi, h)= e^{-\varpi(H_P(h))}=\|\Lambda_\varpi(h^{-1})\phi_\varpi\|,\quad h\in H,
\]
for $\phi_\varpi$ a unit vector in the representation space $V_\varpi$ of $\Lambda_\varpi$  with respect to a fixed norm $\|\cdot\|$ on $V_\varpi$.

Let $\pi=u \nu\in P_1$ with $u\in \CmO$ and $\nu\in N_1$. Let $a\in A_{M}$ be regular. Then there is a unique $n\in N_1$ such that
\begin{equation}\label{eq:conjugation}
 a\pi = n^{-1} a u n.
\end{equation}
Therefore,  $a\mapsto \Lambda_\varpi(n^{-1})\phi_\varpi$ is a rational function on a dense subset of, and hence on all of $A_M/A_H$.

By Arthur's construction \cite[p.\ 238--239]{Ar88} there exist unique integers $k_\beta\ge0$ such that 
\begin{equation}\label{eq:defW}
\lim_{a\rightarrow 1} \prod_{\beta\in \Sigma_P\cap \Sigma_{\overline{P_1}}} (a^{\beta}-a^{-\beta})^{k_\beta} v_P(\varpi, n) 
\end{equation}
exists and is non-zero on a dense subset of $\CmO N_1$. Here $\Sigma_P$ denotes the set of reduced roots of $A_M$ on $N$, and similarly, $\Sigma_{\overline{P_1}}$ is the set of reduced roots on the opposite parabolic. The limit is in fact of the form  $\|W_\varpi(1, \pi)\|$ with $W_\varpi(1,\pi)\in V_\varpi$ a polynomial on $\CmO N_1$.

The weight functions $w_M(\pi)$ appearing in the weighted unipotent integrals are then of the form
\begin{equation}\label{eq:weightw}
w_{M,\CmO}(\pi) = \sum_{\Omega} c_\Omega \prod_{\varpi\in \Omega} \log\|W_\varpi(1, \pi)\|
\end{equation}
where $\Omega$ runs over all finite subsets of extremal weights of $\ka_M$, and $c_\Omega\in\C$ are coefficients which vanish for all but finitely many of the $\Omega$s. Note that $w_{M,\CmO}(\pi)$ attains a finite value on a dense subset of $\pi$s. 

\subsection{Extending polynomials}\label{sec:extpolynomials}
We recall the notion of a Jacobson--Morozov parabolic subalgebra. Recall that $\CmO\subseteq M$ denotes a unipotent conjugacy class, and let $\CmN \subseteq \km$ be the corresponding nilpotent orbit.  By the Jacobson--Morozov Theorem 
\cite[3.3]{CoMc} we can find an $\mathfrak{sl}_2$-triple $(h_0, x_0, y_0)$ for $\CmN$ in $\km$ with $h_0$ semisimple and $x_0, y_0$ nilpotent. We decompose $\km$ into eigenspaces under $h_0$, that is $\km=\bigoplus_{k\in\Z} \km_k$ with $\km_k=\{X\in\km\mid [h_0, X] = k X\}$. 
Let $\tilde M$ denote the centralizer of $A_M$ in $G_n$ and let $\tilde \km\subseteq \kg$ be its Lie algebra. Then $\tilde M$ is a semi--standard Levi subgroup of $G_n$ with $M\subseteq \tilde M$, $\km\subseteq \tilde\km$.
 Then $h_0$ also defines a grading on $\tilde \km$, $\tilde\km=\bigoplus_{k\in\Z} \tilde\km_k$. 

Let $\kq^M=\bigoplus_{k\ge0} \km_k \subseteq \km$ be the corresponding Jacobson--Morozov parabolic subalgebra. Let $\kv^M=\bigoplus_{k>0}\km_k$ and $\kv_1^M=\bigoplus_{k>1}\km_k$.
We also define $\tilde\kv^M=\oplus_{k>0} \tilde\km_k$ and $\tilde\kq^M=\oplus_{k\ge0}\tilde\km_k$ so that $\kv^M$ is a sub--vectorspace of $\tilde\kv^M$. 

Then each $u\in \CmO$ can be written as $k^{-1} e^X k$ for $k\in K\cap M$ and a unique $X\in \kv_1^M$. Note that $W_\varpi(1, \pi)$ is invariant under the conjugation by elements of $K\cap M$ by \cite[(3.7)]{Ar88} so that we can assume that $\pi=e^{X + Y}$ with $X\in \kv_1^M$ and $Y\in \kn_1$ where $\kn_1$ is the Lie algebra of $N_1$. By \cite[p.\ 253]{Ar88} $W_\varpi(1,\pi)$ is a polynomial in $X+Y$. Similarly, $v_P(\varpi, e^Y)$ is a polynomial in $Y\in \kn_1$. 

There are semistandard Levi subgroups $\tilde P= \tilde M \tilde N$ and $\tilde P_1=\tilde M\tilde N_1$ of $G_n$ such that $N\subseteq \tilde N$ and $N_1\subseteq \tilde N_1$. Let $\tilde\kn_1$ be the Lie algebra of $N_1$ so that $\kn_1$ is a sub--vectorspace of $\tilde\kn_1$. We have a canonical isomorphism $\tilde\kn_1\simeq\R^{\dim\tilde\kn_1}$ via the coordinates given by the matrix entries. This also gives a canonical inner product on $\tilde\kn_1$ so that we can find the orthogonal complement of $\kn_1$ in $\tilde\kn_1$. We can therefore extend any polynomial on $\kn_1$ trivially to a polynomial on $\tilde\kn_1$ along that complement. 
A similar construction holds for polynomials on $\kv_1^M$ so that they can be extended trivially to polynomials on $\tilde\kv_1^M$ as well.

In particular, we can extend $X+Y\mapsto W_\varpi(1, e^{X+Y})$ to a polynomial on all of $\tilde\kv_1^M+ \tilde\kn_1$, and $Y\mapsto v_P(\varpi, e^Y)$ to a polynomial on all of $\tilde\kn_1$. Since $X+Y$ is nilpotent,  $\log(\Mid+X+Y)$ and 
$e^{X+Y}$ are finite series. Hence we can also consider $W_\varpi(1, \Mid + X+Y)
=W_\varpi(1, e^{\log(\Mid + X+Y)})$ and $v_P(\varpi, \Mid + Y)
=v_P(\varpi, e^{\log(\Mid + Y)})$ with 
$X\in\tilde\kv_1^{\tilde M}$ and $Y\in \tilde\kn_1$ which are again both 
polynomials. 

\subsection{In which the group is $GL(n)$}
In this section we first prove a slightly more general version of \cite[Lemma 7.2]{MM1}. We change the notation for this section slightly: We assume that $H=G_n$ so that in particular, $P=MN$ and $P_1=MN_1$ be semi--standard parabolic of $G_n$ with the same Levi component $M$. Let $a\in A_M$ be regular. Suppose that $\pi\in P_1$ is a unipotent element and write $\pi= u\nu$ with $u\in M$ and $\nu\in N_1$ unipotent. Then there is a unique $n\in N_1$ such that
\begin{equation}\label{eq:relation}
a\pi= n^{-1} au n,
\end{equation}
that is, we have a well-defined polynomial map $\mathcal{U}_MN_1 \ni\pi\mapsto n\in N_1$ depending on $a$ where $\mathcal{U}_M$ denotes the set of all unipotent elements in $M$. 

 Let $\Phi$ denote the set of all roots of $T_0$ on  
$\kg_n$. 
Let $\Phi_M\subseteq \Phi$ denote the subset of roots which are not trivial when restricted to $A_M$, and let $\Phi_1\subseteq \Phi_M$ be the subset of roots that are positive with respect to $N_1$. Let $\Phi^+$ denote a choice of positive roots in $\Phi$ such that $\Phi_1\subseteq \Phi^+$. Let $\kn'\subseteq \km$ denote the nilpotent subalgebra corresponding to $\Phi^+\smallsetminus \Phi_1$. Then $\kn'\oplus \kn_1$ is the nilpotent radical of the minimal parabolic subalgebra of  $\kg_n$ corresponding to $\Phi^+$.

If $Z\in \kg_n$ is any matrix and $\beta\in\Phi$ a root, we write $Z_\beta$ for the matrix entry of $Z$ corresponding to $\beta$.
Write $u=\Mid +X_0$, $\nu= \Mid + X$, and $n= \Mid + Y$ with $X_0, X, Y$ suitable nilpotent matrices. Up to conjugation with $K\cap M$ we can assume that $X_0\in \kn'$ which we will do from now on.

 We show the following:
\begin{prop}\label{prop:coordinates}
 Let $\beta\in \Phi_1$. 
Then for each subset $\underline{\alpha}\subseteq \Phi_1$ there is a polynomial $Q_{\beta,\underline{\alpha}}(Z; a^{\alpha}, \alpha\in\underline{\alpha})$, $Z:= \pi-\Mid= X_0+X +X_0X$ , such that:
\begin{itemize}
\item\noindent
\[
Y_\beta = \sum_{\underline{\alpha}\subseteq \Phi_1} \frac{Q_{\beta, \underline{\alpha}}(aZa^{-1}; a^\alpha, \alpha\in\underline{\alpha})}{\prod_{\alpha\in\underline{\alpha}} (a^{\alpha}-1)},
\]
where $\underline{\alpha}$ runs over all subsets of $\Phi_1$;
\item  as a function in the matrix entries of $Z$, $Q_{\beta, \underline{\alpha}}$ is a homogeneous polynomial of degree $\#\underline{\alpha}$; 
\item  if $Q_{\beta,\underline{\alpha}}(aZa^{-1}; a^{\alpha}, \alpha\in\underline{\alpha})$ does not vanish identically, then for $X$ in general position, its limit as $a\rightarrow 1$ is non-zero.
\end{itemize}
\end{prop}

\begin{proof}
We introduce a grading on the set $\Phi_1$: We say that $\beta\in\Phi_1$ has degree $k$, $k\ge1$, if the $\beta$-coordinate of $A^k$ is non-zero, but that of $A^{k+1}$ is zero for a general matrix $A\in \kn_1$.
We write $\Phi_1^{(k)}$ for the set of $\beta\in\Phi_1$ of degree $k$. Note that $\Phi_1^{(k)}=\emptyset$ when $k\ge n$.

We rearrange the relation \eqref{eq:relation} as follows:
\begin{multline*}
\Mid + X_0+ a (Z-X_0) a^{-1} 
= n^{-1}ana^{-1} + n^{-1} X_0 a n a^{-1}\\
= \Mid +\sum_{k=1}^{n-1} (-1)^{k-1}Y^{k-1} \left[a Y a^{-1} -Y\right]
+  X_0 + \sum_{k=1}^{n-1} (-1)^{k-1}Y^{k-1} \left[X_0a Y a^{-1} -YX_0\right]\\
\end{multline*}
 For $k\ge1$, the non-zero entries in the matrices $Y^{k-1} \left[a Y a^{-1} -Y\right]$ all correspond to $\beta$ in $\Phi_1^{(l)}$, $l\ge k$. Moreover, the matrix entry in
\begin{equation}\label{eq:conjugatedsum}
\sum_{k=1}^{n-1} (-1)^{k-1}Y^{k-1} \left[a Y a^{-1} -Y\right]
\end{equation}
 corresponding to a root $\beta$ of degree $k$ is of the form
\[
(a^\beta-1) Y_\beta + \sum_{\underline{\alpha}} (a^{\alpha_1}-1) C_{\underline{\alpha}} \prod_{\alpha\in\underline{\alpha}} Y_{\alpha}
\]
where the sum runs over all tuples $\underline\alpha=(\alpha_1,\ldots)$ of pairwise different elements in  $\bigcup_{l<k} \Phi_1^{(l)}$ such that $2\le\sum \deg\alpha_i\le k$, and $C_{\underline{\alpha}}\in\R$ are suitable coefficients. Note that this in particular means that the sum over the $\underline{\alpha}$ in \eqref{eq:conjugatedsum} contains only monomials of degree $\ge2$, and is in fact empty if $k=1$. 

The sum $\sum_{k=1}^{n-1} (-1)^{k-1}Y^{k-1} \left[X_0 a Y a^{-1} -YX_0\right]$ has a similar structure as \eqref{eq:conjugatedsum}, except that each monomial has exactly one linear factor consisting of a matrix entry of $X_0$.  In particular, as a polynomial in $Y$ and $X_0$, there is no linear factor, and the matrix entry corresponding to a root of degree $k$ has only factors consisting of $Y_\beta$ with $\deg \beta<k$ and $X_0$. Moreover, no matrix entry  of 
\[
\sum_{k=1}^{n-1} (-1)^{k-1}Y^{k-1} \left[X_0a Y a^{-1} -YX_0\right]
\]
is divisible by any of the factors $a^{\alpha}-1$ (unless it is identically $0$), since in $\left[X_0a Y a^{-1} -YX_0\right]$ the terms $a^\beta Y_\beta$ and $Y_\beta$ cannot occur with non--vanishing coefficient in the same matrix entry.

We can now argue inductively in the degree of $\beta$. If $\beta\in\Phi_1^{(1)}$, then 
\[
(a^{\beta}-1) Y_\beta = a^{\beta} Z_\beta.
\]
The assertion then follows by induction from the above description of the matrix entries. 
\end{proof}

\subsection{Back to $H$}\label{subsec:weights}
We return to  the notation of \S\ref{sec:setup} and \eqref{eq:conjugation}. Write $u=\Mid+X_0$, $\pi= \Mid +Z$ and $n=\Mid +Y$ with $X_0$, $Z$, and $Y$ nilpotent matrices. 

\begin{prop}

There exists $\nu$ such that 
\begin{equation}
\|W_\varpi(1, \Mid + sZ)\| 
= s^\nu  \|W_\varpi(1, \pi)\|
\end{equation}
for all  $s>0$.
\end{prop}
\begin{remark}
Note that $\Mid + sZ$ is not necessarily contained in $\CmO$ (it does not even have to be contained in $H$), but as we discussed above, we can extend $W_\varpi(1, \pi)$ to a polynomial on all of $\Mid+ \tilde\kv_1^M+ \tilde \kn_1$.
\end{remark}
\begin{proof}
We want to use Proposition \ref{prop:coordinates}. The element $a$ in \eqref{eq:conjugation} defines a semistandard Levi subgroup $\tilde M$ in $G_n$ by  taking its centralizer in $G_n$. Then $a$ is a regular element of the center of $\tilde M$ and $M\subseteq \tilde M$. Moreover, as before, there are semistandard parabolic subgroups $\tilde P= \tilde M \tilde N$ and $\tilde P_1= \tilde M \tilde N_1$ of $G_n$ such that $N\subseteq \tilde N$ and $N_1\subseteq \tilde N_1$.  As explained above, we can extend $v_P(\varpi, n)$ to a polynomial in $Y\in \tilde\kn_1$, $n=\Mid +Y$. Each coordinate $Y_\beta$ of $Y$ can be described according to Proposition \ref{prop:coordinates} so that under the change $Z\mapsto sZ$  the coordinate becomes
\[
Y_\beta = \sum_{\underline{\alpha}\subseteq \Phi_1} s^{\#\underline{\alpha}} \frac{Q_{\beta, \underline{\alpha}}(aZa^{-1}; a^\alpha, \alpha\in\underline{\alpha})}{\prod_{\alpha\in\underline{\alpha}} (a^{\alpha}-1)}.
\]
Hence by definition of $\|W_\varpi(1,\pi)\|$ in \eqref{eq:defW} we can find $\nu\ge0$ such that
\[
\|W_\varpi(1, \Mid + sZ)\|  
= s^{\nu}  \|W_\varpi(1, \Mid + Z)\|.
\]
\end{proof}

Together with \eqref{eq:weightw} this immediately implies the following:
 \begin{cor}\label{prop:weight:fct}
 There exist 
 \begin{itemize}
  \item constants $r,  q\ge0$,
  \item polynomials $p_1, \ldots,p_q:\tilde\kv^M\oplus\tilde\kn_1\longrightarrow \R$ which do not vanish on an open dense subset of $Z$ with $\Mid+Z\in\CmO N_1$, and
  \item  complex polynomials $Q_j$ in $q$-many variables, $j=0,\ldots, r$,
 \end{itemize}
  such that for all $Z\in \kv^M\oplus\kn_1$ and $s>0$ we have
\[
 w_{M,\CmO}(\Mid + sZ)
 = \sum_{i=0}^r (\log s)^i Q_i(\log |p_1(Z)|, \ldots, \log |p_q(Z)|).
\]
 \end{cor}

\section{Test functions}\label{sec:test:fcts}

\subsection{Linearizing the metric}
 Let 
\begin{equation}\label{dist-fct}
r(g)=d(g\K_\infty, \K_\infty),
\end{equation}
where $d(\cdot, \cdot)$ denotes the geodesic distance function on $\tilde X= G_\infty/\K_\infty$.  We continue to assume that $G_\infty$ is a real semisimple Lie group, and we fix an embedding of $G_\infty$ into $\GL_n(\R)^1$ for some $n\ge1$ as at the beginning of \S\ref{sec:setup}. From now on, we will identify $G_\infty$ and all its subgroups with their image in $\GL_n(\R)^1$ instead of writing $H$ etc., and the Lie algebra $\kg$ will be identified with its corresponding image in $\Liesl_n(\R)$. In addition to the properties satisfied by $G_\infty$ in \S\ref{sec:setup}, we can choose the embedding furthermore such that  
 \begin{itemize}
\item  $\K_\infty\subseteq \rO(n)$.
\item $\af_0^G$ is contained in the subalgebra of diagonal matrices in $\Liesl_n(\R)$. 
\item If $\kg = \kk\oplus\ks$ is the Cartan decomposition of $\kg$, then $\ks$ is contained in the symmetric matrices in $\Liesl_n(\R)$, and $\kk$ in the skew-symmetric matrices.
\item  For $X=(X_{ij})_{i,j=1,\ldots, n}\in\kg$, put $\|X\|^2=\sum_{i,j} |X_{ij}|^2$. Then $\|\cdot\|$ coincides with the norm on $\ks$ obtained from the Killing form on $\kg$. (More generally, we define $\|\cdot\|$ similarly on all of $\Liesl_n(\R)$.)
\item The unipotent radical $U_0$ of the minimal parabolic subgroup $P_0$ is contained in the group of  unipotent upper triangular matrices of $\GL_n(\R)^1$. 
 \end{itemize}
Let $d_n(\cdot, \cdot)$ denote the geodesic distance function on $\GL_n(\R)^1/\rO(n)$, and let $r_n(g)=d_n(g\rO(n), \rO(n))$. By our choice of embedding, the Cartan decomposition on $G_\infty$ with respect to $K_\infty$ and on  $\GL_n(\R)^1$ with respect to $\rO(n)$ are compatible, that is, if $g = k_1ak_2\in G_\infty=K_\infty A_0^G K_\infty$, then $k_1ak_2$ is also the Cartan decomposition of $g$ in $\GL_n(\R)^1$. Hence $r(g)=r(a)$ and $r_n(g)=r_n(a)$. By \cite[(4.6.25)]{GaVa} we have  $r(a)=\|\log a\|$ and $r_n(a)=\|\log a\|$ so that $r$ and $r_n$ coincide on $G_\infty\subseteq \GL_n(\R)^1$. This allows us to use results on the geodesic distance $r_n$ from \cite{MM1} also for $r$. Recall from \cite[Lemma 12.2]{MM1} that if $g=I_n+X\in\GL_n(\R)^1$ with $X$ a nilpotent upper triangular matrix, then
\begin{equation}\label{eq:distance:small:X}
 r_n(I_n+X) = \frac{1}{4} \|X\|^2 + O(\|X\|^3)
\end{equation}
as $X\rightarrow 0$. Here $I_n\in\GL_n(\R)^1$ denotes the identity matrix. Hence if $X$ varies over matrices such that $I_n+X\in G_\infty$, then the same is true for $r_n(I_n+X)$ by our above remark.

We also need to understand how $r(g)$ behaves as $g$ varies over unipotent matrices in $G_\infty$ that become unbounded:
\begin{lem}\label{lem:lower:bound:global}
 We have
 \[
  r(g)\ge \frac{1}{2} \log\left(1+\frac{1}{n}\|g-I_n\|^2\right)
 \]
for all unipotent $g\in G_\infty$. 
\end{lem}
\begin{proof}
Let $g\in G_\infty$ be unipotent and write $g=e^{Y_0}$ with $Y_0\in\kg$ nilpotent. There exists $k\in \K_\infty$ such that $Y:=\Ad(k) Y_0$ is an upper triangular nilpotent matrix. 
We can therefore write $e^Y=I_n+N$ for some nilpotent upper triangular $N\in \Liesl_n(\R)$. Let $X\in\af_0^G$ and $k_1, k_2\in \K_\infty$ such that $g=k_1 e^X k_2$. Then $\|X\|= r(g) = r(k e^{Y_0} k^{-1})= r(e^Y)$, where the first equality follows from \cite[(4.6.25)]{GaVa}.
Then
\[
 \tr e^{2X} = \tr(g^tg) = \tr e^{Y_0^t}e^{Y_0} =\tr e^{Y^t} e^Y 
 = \tr (I_n+N)^t(I_n+N)= n + \tr N^tN
 = n + \|N\|^2.
\]
Let $X_1,\ldots, X_n$ be the diagonal entries of $X$. Then 
\[
 \tr e^{2X}= \sum_{j=1}^n e^{2X_j}
 \le \sum_{j=1}^n e^{2|X_j|}
 \le n e^{2\|X\|}.
\]
Hence
\[
 e^{\|X\|} \ge (1+\frac{1}{n}\|N\|)^{1/2}
\]
so that 
\[
 r(g)=\|X\|\ge \frac{1}{2}\log(1+\frac{1}{n}\|N\|^2)
\]
as asserted.
\end{proof}

\begin{lem}
Let $\tilde\uf_0$ denote the vector space of all nilpotent upper triangular matrices in $\Liegl_n(\R)$. Then there exists $x_0>0$ such that for all $X\in\tilde\uf_0$ with $\|X\|\ge x_0$  we have
\[
\|e^X-I_n\|\ge \frac{1}{\sqrt{2}}\|X\|.
\]
\end{lem}
\begin{proof}
We consider a more abstract situation: Let $P:\R^N\longrightarrow\R$ be a non-negative polynomial satisfying the following two properties: 
\begin{itemize}
\item $P(x)\longrightarrow \infty$ as $\|x\|\longrightarrow\infty$, where $\|x\|$ denotes the usual Euclidean norm on $\R^N$, and 
\item $P(x)$ can be written as $P(x) = \|x\|^2 + \sum_{\alpha} a_\alpha x^\alpha$ with $\alpha$ running over all multiindices $\alpha=(\alpha_1,\ldots, \alpha_N)$ of degree $3\le |\alpha| = \sum_i\alpha_i\le \deg P$.
\end{itemize}
The function $\tilde\uf_0\ni X\mapsto\|e^X-I_n\|^2$ is a polynomial exactly of this form. It will therefore suffice to show that for such polynomials we have $P(x)\ge \|x\|^2/2$ for $\|x\|\ge x_0$ for some $x_0>0$. 

Fix $x\in\R^N$ with $\|x\|=1$ and set $Q_x(t) = P(tx)-t^2$ for $t\in\R$. If $Q_x$ vanishes identically in $t$, we are done. If $Q_x$ does not vanish identically, then $Q_x$ is a non-trivial polynomial of degree $\ge3$ in $t$. Since $P(tx)\longrightarrow +\infty$ as $t\rightarrow\infty$, we must also have $Q_x(t)\longrightarrow+\infty$ for $t\rightarrow\infty$. Hence there exists $t_x>0$ such that $Q_x(t)\ge0$ for all $t\ge t_x$, that is, $P(tx)\ge t^2$ for $t\ge t_x$. Since the set of all $x\in\R^N$ with $\|x\|=1$ is compact, we can find $t_0>0$ such that $P(x)\ge \|x\|^2/2$ for all $x\in\R^N$ with $\|x\|\ge t_0$. This finishes the  proof.
\end{proof}

\subsection{Asymptotic expansion of the heat kernel}\label{sec:asymp:heat}

We recall the following asymptotic expansion of $h_t^\nu$ which, for example, can be found in \cite[Corollary~10.4]{MM1}:
If $\psi\in C^\infty(\R)$ is non-negative, equals $1$ around $0$ and has sufficiently small support, 
then for any sufficiently large $N$ and any $0<t\le 1$ we have
\begin{equation}\label{eq:asympt:kernel}
 h_t^\nu(g)= (4\pi t)^{-d/2} \psi(r(g)) \exp\left(-r(g)^2/(4t)\right) \sum_{n=0}^N a_n^\nu(g) t^n + O\left(t^{N+1-d/2}\right)
\end{equation}
where $a_n^\nu\in C^\infty(G_\infty)$ are suitable functions and the implied constant depends on the function $\psi$ but not on $N$. 
We further note a uniform upper bound for $h_t^\nu$ (see, for example, \cite[Corollary~10.2]{MM1}): There exists $C>0$ such that
\begin{equation}\label{eq:uniform:upper:heat}
 \left|h_t^\nu(g)\right|
 \le C t^{-d/2} e^{-r(g)^2/(4t)}
\end{equation}
for all $g\in G_\infty$ and all $0<t\le 1$.

\section{Proof of Proposition \ref{prop:asymp:unipotent:distr}}\label{sec:asympt:weight:orb}
The unipotent distribution $J_{\text{unip}}(\tilde\phi_t^\nu)$ can be written as a linear combination of certain weighted unipotent integrals $J_M^G(\CmO, \tilde\phi_t^\nu)$, see \S\ref{sec:fine:exp}. It will turn out (see \S\ref{sec:splitting}) that for us only the integrals $J_M^G(\CmO, \psi_t^\nu)$ will be relevant. For these we have the following:

\begin{prop}\label{prop:asymp:unipotent:distr}
Let $P=MU\in\CmF$, and let $\CmO\subseteq M(\R)$ be a unipotent conjugacy class in $M(\R)$. 
Let $\CmO^G\subseteq \bG(\R)$ be the unipotent conjugacy class in $\bG(\R)$ induced from $M(\R)$ along $P(\R)$.
Let $d_\CmO^G =\dim\CmO^G$ and $r_M=\dim\ka_M^G$.
Then there exist constants $b_{ij}=b_{ij}(M,\CmO)\in\C$, $j\ge0$, $0\le i\le r_M$ such that for every $0<t<1$ 
\[
 J_M^G(\CmO, \psi_t^\nu)
 \sim t^{-d/2  + d_\CmO^G/4} \sum_{j=0}^\infty \sum_{i=0}^{r_M} b_{ij} t^{j/2} (\log t)^i.
\]
Moreover, the coefficients $b_{ij}$ are uniformly bounded.
\end{prop}

We prove a slightly more general version of Proposition~\ref{prop:asymp:unipotent:distr} which will make for a cleaner proof of Theorem~\ref{thm:asymptotic:geom} in the next section.
Let $M\in\CmL$ and let $P_1=M_1U_1\in\CmF(M)$. 
If $f\in C^\infty(G_\infty)$ define $f_{P_1}\in C^\infty(M_1)$ by
\[
 f_{P_1}(m)
 =\delta_{P_1}(m)^{1/2} \int_{\K_\infty} \int_{U_1} f(k^{-1} muk)\, du\, dk
\]
provided the right hand side is finite for any $m\in M_1$.
If $f$ has compact support, $f_{P_1}$ is compactly supported as well. 

If $\CmO$ is a unipotent conjugacy class in $M$, we can define $J_M^{M_1}(\CmO, f_{P_1})$ as before with $G_\infty$ replaced by $M_1$.

\begin{prop}\label{prop-asympt-exp}
Let $M\in \CmL$, $\CmO\subseteq M$ a unipotent conjugacy class in $M$, and $P_1=M_1U_1$ a semistandard parabolic subgroup of $G$ with $M\subseteq M_1$. Let $d_\CmO^{G_\infty}=\dim\CmO^{G_\infty}$ be the dimension of the unipotent orbit in $G_\infty$ induced from $M$, and let $r_M^{M_1}=\dim \ka_M^{M_1}$. Then there exist constants $b_{ij}=b_{ij}(M,\CmO)\in\C$, $j\ge0$, $0\le i\le r_M^{M_1}$ such that for every $0<t\le1$ 
\begin{equation}\label{asympt-exp7}
 J_M^{M_1}(\CmO, \left(\psi_t^\nu\right)_{P_1})
 \sim t^{-d/2  + d_\CmO^G/4} \sum_{j=0}^\infty \sum_{i=0}^{r_M^{M_1}} b_{ij} t^{j/2} (\log t)^i.
\end{equation}
Moreover, the coefficients $b_{ij}$ are uniformly bounded.
\end{prop}

For the proof of this proposition we follow \cite[\S 12]{MM1} taking into account our results above.
Suppose $\tilde Q\subseteq M_1$ is a semistandard parabolic which contains $M$. Let $\tilde Q= L_{\tilde Q} N_{\tilde Q}$ with $M\subseteq L_{\tilde Q}$, and let $Q= \tilde Q U_1$. Then $Q$ is a semistandard parabolic subgroup of $G$ with Iwasawa decomposition $Q=L_{\tilde Q}  N_Q$, $ N_Q= N_{\tilde Q} U_1$, that is, $Q$ and $\tilde Q$ have the same Levi component $L_Q=L_{\tilde Q}$ containing $M$. Hence $\varphi^Q=\varphi^{\tilde Q}$ and $w^Q_{\CmO, M}= w^{\tilde Q}_{\CmO, M}$.  Applying \eqref{eq:unipotent:combined} to $J_M^{M_1}(\CmO, \left(\psi_t^\nu\right)_{P_1})$ and unfolding the definition of $\left(\psi_t^\nu\right)_{P_1}$, we obtained
\begin{equation}\label{eq:weighted:levi}                                       
J_M^{M_1}(\CmO, \left(\psi_t^\nu\right)_{P_1})
   = \sum_{Q} \tilde c(Q, \CmO) \int_{\uf_{Q,1}} \psi^\nu_t(e^Y) w_{M,\CmO}^Q(e^Y) |\varphi^{L_Q}(Y)|^{1/2} \, dY
\end{equation}
where the sum runs over all $Q\in \CmF(M)$ in the image of the map $\tilde Q\mapsto Q$, and the constants depend on the normalization of measures. We therefore need to analyze integrals that are of the form as those on the right hand side of \eqref{eq:weighted:levi}.

 Fix $Q\in\CmF(M)$. We equip $\uf_{Q,1}$ with a Euclidean norm by fixing some isomorphism $\uf_{Q,1}\simeq \R^{\dim\uf_{Q,1}}$ which respects the direct sum decomposition $\uf_{Q,1}= \uf_1^{L_Q}\oplus\nf_Q$. 
 Let $\eps>0$ and let $B(\eps)\subseteq \uf_{Q,1}$ be the ball around $0$ of radius $\eps$. 
 Put $U(\eps)=\uf_{Q,1}\smallsetminus B(\eps)$. We split the integral on the right hand side of \eqref{eq:weighted:levi} into integrals over $B(\eps)$ and $U(\eps)$. 
 
 Recall the function $\beta$ from the definition \eqref{eq:def:psi} of $\psi_t^\nu$. We assume that $\beta$ has sufficiently small support. Let $\eps>0$ be such that  $\beta(r(\exp(\cdot)))$ restricted to $B(\eps)$ is identically equal to $1$.   Fix $\psi\in C^\infty(\R)$ as in \S\ref{sec:asymp:heat}
 (for the expansion \eqref{eq:asympt:kernel}) such that the support of the restriction of $\psi(r(\exp(\cdot)))$ to $\uf_{Q,1}$ is contained in  $B(\eps)$.

 We first show that the integral over $U(\eps)$ decays exponentially in $t^{-1}$ as $t\searrow 0$ and therefore does not contribute to the asymptotic expansion. Indeed, it follows from \S\ref{subsec:weights} and \eqref{eq:uniform:upper:heat} that
 \begin{equation}\label{eq:int:bigger:eps}
\left|  \int_{U(\eps)} h_t^\nu(e^Y) |\varphi^Q(Y)|^{1/2} w_{M,\CmO}^Q(e^Y)\, dY \right|
\end{equation}
can be bounded by a constant multiple of a finite sum of integrals of the form
\begin{equation}\label{eq:integration:u(eps)}
t^{-d/2}\int_{U(\eps)} e^{-r(e^Y)^2/(4 t)}  |\varphi^Q(Y)|^{1/2} \prod_{j=1,\ldots, J} \log \|p_j(e^Y-I)\|\, dY  
 \end{equation}
where $p_j$, $j=1,\ldots, J$ are suitable polynomials on the Lie algebra as the ones appearing in Corollary \ref{prop:weight:fct}. In particular, $p_j(e^Y-I)$ does not vanish identically in $Y\in U(\eps)$.

The polynomial $\varphi^Q$ has degree $\dim \lf_1^Q$ so that $|\varphi^Q(Y)|^{1/2}\ll_\eps  \|Y\|^{\dim\lf_1^Q/2}$. 
It follows from Lemma~\ref{lem:lower:bound:global}, that there are constants $c_\eps, c_1>0$ such that \eqref{eq:integration:u(eps)} is bounded by a constant multiple of
\[
 e^{-c_{\eps}/t}\int_{\uf_{Q,1}} \|Y\|^{\frac{\dim\lf_1^Q}{2}} e^{-c_1\log^2(1+\|Y\|)/t} \prod_{j=1,\ldots,J} \log\|p_j(e^Y-I)\|\, dY
\]
provided that $\eps$ is sufficiently small. Note that $p_j(e^Y-I)$ is again a polynomial in $Y$.
By \cite[Lemma~7.7]{MM1} this integral converges. (In fact, \cite[Lemma~7.7]{MM1} only treats the case without the power of $\|Y\|$ in the integral, but it is immediate from the proof that a polynomial in $\|Y\|$ does not change the validity of the assertion.)
Therefore, \eqref{eq:integration:u(eps)} is bounded by a constant multiple of $e^{-c_\eps/t}$, that is, it decays exponentially in $t^{-1}$ as $t\searrow0$.

For the integral over $B(\eps)$ we use \eqref{eq:asympt:kernel} and the expansion of $w_{M,\CmO}^Q$ from Corollary~\ref{prop:weight:fct}. 
Using \eqref{eq:asympt:kernel} we get
\begin{multline}\label{eq:int:over:ball}
   \int_{B(\eps)} h_t^\nu(e^Y) |\varphi^Q(Y)|^{1/2} w_{M,\CmO}^Q(e^Y) \, dY\\
      = (4\pi t)^{-d/2} \sum_{n=0}^N t^n \int_{B(\eps)}\exp\left(-r(e^Y)^2/(4t)\right) \psi(r(e^Y)) a_n^\nu(e^Y) |\varphi^Q(Y)|^{1/2} w_{M,\CmO}^Q(e^Y)\, dY\\ + O\left(t^{N+1-d/2}\right).
\end{multline}
We now first proceed as in \cite[\S12.2]{MM1} and use Taylor expansion of the functions  $\exp\left(-r(e^Y)^2/(4t)\right)$ and 
$f(Y):=\psi(r(e^Y)) a_n^\nu(e^Y) |\varphi^Q(Y)|^{1/2}$ in $N=e^Y-I$ around $0$. 
Note that except for $\varphi^Q(Y)$ all involved functions depend only on $r(e^Y)$, and can therefore be continued to smooth function on all of $\Liegl_n(\R)$. Since $\varphi^Q(Y)$ is a polynomial, we can extend it to a polynomial on all of $\Liegl_n(\R)$ as well.    
Then, using \eqref{eq:distance:small:X}, we can write for any $K\ge1$ and any $0<t\le 1$,
\[
\exp\left(-r(I+ t^{1/2}N)^2/(4t)\right)
= \exp\left(-\|N\|^2/4\right)\left(\sum_{k=0}^K t^{k/2} q_k(N) + R_K(t, N)\right)
\]
where $q_k$ are suitable polynomials of degree $\le 3nk$, $q_0(N)=1$,  and $R_K(t, N)$ is a remainder term satisfying 
\[
 |R_K(t, N)| \le c_1 t^{(K-1)/2} (1+ \|N\|)^{3 n K}
\]
for every $0<t\le 1$ with $c_1>0$ some suitable constant.

Similarly, 
\[
 f(\log( I+ t^{1/2} N)) 
 = \sum_{l\le L} b_l(N) t^{l/2}  + Q_L(t, N)
\]
where $b_l$ is a polynomial in $N$ of degree $\le l$,
and $Q_L(t, N)$ is a remainder term satisfying 
\[
|Q_L(t, N)|\le  c_2 t^{(L+1)/2} (1+\|N\|)^{L+1}
\]
for all $0<t\le 1$ and $N$ with $Y= \log (I+N) \in B(\eps)$ with $c_2>0$ some absolute constant.  
Note that $b_l(N)=0$ whenever $l< \dim\lf_1^Q/2$ since $|\varphi^Q|^{1/2}$ is homogeneous in $Y$ of degree $\dim\lf_1^Q/2$.

Hence the left hand side of \eqref{eq:int:over:ball} equals after a change of variables
\[
 \sum_{k\le K}\sum_{l\le L} t^{ (k+l)/2}t^{\dim\uf_{Q,1}/2} \int_{t^{-1/2} \mathcal{B}(\eps)} \exp\left(-\|N\|^2/4\right) q_k(N) b_l(N) w_M^Q(I+ t^{1/2}N)\, dN
 + \Phi_{K, L}(t)
\]
where $\mathcal{B}(\eps)$ is the image of $B(\eps)$ in $\Liegl_n(\R)$ under $Y\mapsto N=e^Y-I$, and
with the remainder $\Phi_{K, L}(t)$ satisfying
\[
 \left|\Phi_{K, L}(t)\right|
 \le c_3 t^{(L+K+1)/2} .
\]
Note that the Jacobian of the change from $Y$ to $N$ is  a polynomial in $N$ with non-vanishing constant term that we absorb into the asymptotic expansion.
Using the asymptotic expansion for the weight $w_M^Q(I+ t^{1/2}N)$, we can find coefficients $c_{m, j,\Xi}^\nu$ such that the left hand side of \eqref{eq:int:over:ball} equals 
\[
t^{-d/2} t^{\dim\uf_{Q,1}/2}\sum_{\dim\lf_1^Q\le m\le N} \sum_{j=0}^r t^{m/2} (\log t)^j \sum_{\Xi} \int_{t^{-1/2}\mathcal{B}(\eps)} c_{m, j, \Xi}^\nu(Y) \prod_{k\in\Xi} \log\|p_k(Y)\|\, d Y
+ \tilde\Phi_N(t)
\]
with the polynomials $p_k$, $k=1,\ldots, q$,  as in Corollary \ref{prop:weight:fct}, $\Xi$ running over (multi-)subsets of $\{1,\ldots, q\}$ whose size is bounded by the polynomials $Q_j$ appearing in Corollary \ref{prop:weight:fct}, and  $\tilde\Phi_N(t)$ satisfying
\[
 \left|\tilde\Phi_N(t)\right|
 \le c_4 t^{(N-d+\dim\lf_1^Q+\dim\uf_{Q,1}+ 1)/2}.
\]
Now for $0<t\le1$ we have
\[
 \sum_{\Xi} \int_{t^{-1/2}\mathcal{B}(\eps)} c_{m, j, \Xi}^\nu(Y) \prod_{k\in\Xi} \log\|p_k(Y)\|\, d Y
 = C_{m, j} + O\left(e^{-c_4/t}\right)
\]
which follows as in \cite[\S 12]{MM1}.
Hence  \eqref{eq:int:over:ball} equals
\[
t^{(-d+\dim\lf_1^Q +\dim\uf_{Q,1})/2} \sum_{m\le N} \sum_{j=0}^r \tilde C_{m,j} t^{m/2} (\log t)^j + O\left(t^{(N-d+\dim\lf_1^Q+1)/2}\right)
\]
for $0<t\le1$ where $\tilde C_{m,j}=C_{m+\dim\lf_1^Q,j}$. Since $d_\CmO^{G_\infty}= 2(\dim\uf_{Q,1}+\dim\lf_1^Q)$, the assertion of the proposition follows.

\section{Proof of Theorem \ref{thm:asymptotic:geom}}\label{sec:global}
The proof of Theorem~\ref{thm:asymptotic:geom} is global so that we return to our global notation. In particular, $\bG$ denotes again a reductive algebraic group defined over $\Q$. We fix a minimal parabolic subgroup $P_0$ of $\bG$ (as an algebraic group over $\Q$), and write $P_0=M_0U_0$ with $U_0$ the unipotent radical of $P_0$ and $M_0$ the Levi subgroup of $P_0$. We call a parabolic subgroup of $\bG$ standard if it contains $P_0$, and semistandard if it contains $M_0$. Let $\CmL$ denote the set of all semistandard Levi subgroups of $\bG$, that is, all $M\subseteq \bG$ which are Levi components of semistandard parabolic subgroups.

\subsection{Arthur's fine geometric expansion}\label{sec:fine:exp}
\setcounter{equation}{0}
 Let $S$ be a finite set of places of $\Q$, which includes the archimedean place, such that $K_v=\cpt_v$ for $v\not\in S$. Let $G(\Q_S)^1=G(\Q_S)\cap G(\A)^1$. 

Let $M\in\CmL$. Following Arthur, we introduce an equivalence relation on the set of unipotent elements in $M(\Q)$  that depends on the set $S$: Two unipotent elements $u,v\in M(\Q)$ are equivalent if and only if $u$ and $v$ are conjugate in $M(\Q_S)$. We denote the equivalence class of $u$ by $[u]_S\subseteq M(\Q)$ and let $\CmU_S^M$ denote the set of all such equivalence classes. 

Note that two equivalent unipotent elements define the same unipotent conjugacy class in $M(\Q_S)$, so we can view $\CmU^M_S$ also as the set of unipotent conjugacy classes in $M(\Q_S)$ which have at least one $\Q$-rational representative, and we denote the corresponding conjugacy class by $[u]_S$ as well. 
This differs from our notation for unipotent conjugacy classes in $\bG(\R)^1$ from the previous sections, but we now need to keep track of the dependence on~$S$.

\begin{remark}
 \begin{enumerate}[label=(\roman{*})]
  \item If $T\subseteq S$, then we get a well-defined map $\CmU^M_S\ni[u]_S\mapsto [u]_T\in\CmU^M_T$. 

  \item If $G=\GL(n)$, the equivalence relation is independent of $S$ and is the same as conjugation in $M(\Q)$.
 \end{enumerate}
\end{remark}

For $[u]_S\in \CmU^M_S$ and $f_S\in C^\infty_c(\bG(\Q_S)^1)$, Arthur associates a weighted orbital integral $J_M^G([u]_S, f_S)$ \cite{Ar88} which is a distribution supported on  the $\bG(\Q_S)$-conjugacy class induced from $[u]_S\subseteq M(\Q_S)$. 
If $S=\{\infty\}$, these distributions were discussed in \S \ref{sec:weighted}. 
 Let $\One_{\cpt^S}\in C_c^\infty(\bG(\A^S))$ be the characteristic function of $\cpt^S$, if $f_S\in C_c^\infty(\bG(\Q_S)^1)$, then we write $f=f_S\One_{\cpt^S}\in C^\infty_c(\bG(\A)^1)$. 

\begin{prop}[\cite{Ar85}, Corollary 8.3]\label{prop:fine:expansion}
 There exist unique constants $a^M([u]_S, S)\in\C$, $[u]_S\in \CmU^M_S$, such that for all $f_S\in C_c^{\infty}(\bG(\Q_S)^1)$ we have
 \begin{equation}\label{fine-geom5}
  J_{\text{unip}}(f)
  = \sum_{M\in \CmL}\sum_{[u]_S\in\CmU^M_S} a^M([u]_S, S) J_M^{\bG}([u]_S, f_S).
 \end{equation}
\end{prop}

\subsection{The splitting formula}\label{sec:splitting}
To understand the behavior of the distributions $J_M^{\bG}([u]_S,f_S)$ for our test functions, we want to apply our asymptotic expansion for the archimedean weighted integral. To that end we need to separate $\infty$ from the other places in $S$ which we will do by using Arthur's splitting formula \cite[(18.7)]{Ar05}: 
Suppose that $S=S_1\cup S_2$ with $S_1$, $S_2$ non-empty and disjoint, and that $f_S$ is the restriction of a product $f_{S_1}f_{S_2}$ to $G(\Q_S)^1$ with $f_{S_j}\in C^\infty(G(\Q_{S_j}))$, $j=1,2$. Then
\begin{equation}\label{eq:splitting}
 J_M^{\bG}([u]_S, f_S) 
 = \sum_{L_1,\,L_2\in\CmL(M)} d_M^{\bG}(L_1,L_2) J_{M}^{L_1}([u]_{S_1}, f_{S_1,Q_1}) J_M^{L_2}([u]_{S_2}, f_{S_2,Q_2}),
\end{equation}
where the notation is as follows: The $d_M^{\bG}(L_1,L_2)\in\R$ are certain constants which depend only on $M, L_1, L_2, \bG$ but not on $S$. In fact, $d_M^{\bG}(L_1, L_2)$ is non-zero only if the natural map $\af_M^{L_1}\oplus \af_M^{L_2}\longrightarrow\af_M^{\bG}$ is an isomorphism. The $Q_j$ are arbitrary elements in $\CmP(L_j)$, and  
\begin{equation}\label{auxil-fct}
f_{S_j, Q_j}(m)= \delta_{Q_j}(m)^{1/2}\int_{\cpt_{S_j}}\int_{N_j(\Q_{S_j})} f_{S_j}(k^{-1} m n k)\, dn \, dk
\end{equation}
where $N_j$ is the unipotent radical of $Q_j$. Finally, $J_M^{L_j}([u]_{S_j}, \cdot)$ denotes the $S_j$-adic distribution supported on the $L_j(\Q_{S_j})$-conjugacy class induced from $[u]_{S_j}\subseteq M(\Q_{S_j})$ and defined as in \cite{Ar88}.

\subsection{Completion of the proof of Theorem \ref{thm:asymptotic:geom}}
Let $S$ be as in \S\ref{sec:fine:exp} and write $S=\{\infty\}\sqcup S_0$. Then $K_f= K_{S_0} \cpt^S$. Recall the definition of $\tilde\phi^\nu_t$ and $\psi^\nu_t$ from \eqref{eq:def:phi:tilde} and \eqref{eq:def:psi}, respectively, so that 
\[
\tilde\phi^{\nu}_t
= \psi_t^\nu \cdot \One_{K_{S_0}}\cdot \One_{\cpt^S} .
\]
By Proposition~\ref{prop:fine:expansion} we have
\[
 J_{\text{unip}}(\tilde\phi^\nu_t)
  = \sum_{M\in \CmL}\sum_{[u]_S\in\CmU^M_S} a^M([u]_S, S) J_M^G([u]_S, \psi_t^\nu \cdot \One_{K_{S_0}}).
\]
This is a finite sum and the number of summands is independent of $t$ because the support of $ \psi_t^\nu \cdot \One_{K_{S_0}}$ is independent of $t$. 
To prove Theorem~\ref{thm:asymptotic:geom} it therefore suffices to establish an asymptotic expansion of $J_M^{\bG}([u]_S, \psi_t^\nu \cdot \One_{K_{S_0}})$ as $t\searrow 0$.
We first apply the splitting formula \eqref{eq:splitting} to this integral with $S_1=\{\infty\}$ and $S_2=S_0$. We get
\[
 J_M^{\bG}([u]_S, \psi_t^\nu \cdot \One_{K_{S_0}})
 = \sum_{L_1,\,L_2\in\CmL(M)} d_M^{\bG}(L_1,L_2) J_{M}^{L_1}([u]_{\infty}, \psi^\nu_{t,Q_1}) J_M^{L_2}([u]_{S_0}, \One_{K_{S_0},Q_2}).
\]
Again, this is a finite sum with number of summands independent of $t$, and $d_M^{\bG}(L_1,L_2)$ and $J_M^{L_2}([u]_{S_0}, \One_{K_{S_0},Q_2})$ constant in $t$. In combination with the asymptotic expansion of $J_{M}^{L_1}([u]_{\infty}, \psi^\nu_{t,Q_1})$ as $t\searrow0$ from Proposition~\ref{prop:asymp:unipotent:distr} we obtain Theorem~\ref{thm:asymptotic:geom}

\begin{remark}
In the rank one case, the short time asymptotic expansion of the 
regularized trace of the heat operators can also be obtained by using methods of microlocal analysis
\cite[Theorem~A.1]{AR}. In fact, this method works for the more general class
of manifolds with cusps. It is a
challenging problem to see if in the higher rank case the asymptotic expansion
\eqref{asex} can also be derived by methods of microlocal analysis.
\end{remark}

\section{The spectral side of the non-invariant trace formula}\label{sec-trform}
\setcounter{equation}{0}

For the convenience of the reader we summarize in this section some basic facts about 
Arthur's non-invariant trace formula. The trace formula is the equality 
\begin{equation}\label{tracef1}
J_{\geo}(f)=J_{\spec}(f),\quad f\in C_c^\infty(G(\A)^1),
\end{equation}
of the geometric side $J_{\geo}(f)$ and the spectral side $J_{\spec}(f)$ of the
trace formula. The geome\-tric side has been described in the previous section.
In this section we recall the definition of the spectral side, and in particular
the refinement of the spectral expansion obtained in \cite{FLM1}. 

The main ingredient of the spectral side  are logarithmic derivatives of 
intertwining operators. We briefly recall the structure of the intertwining 
operators.

Let $P=MU_P\in\cP(M)$. 
Recall that we denote  by $\rts_P\subset\af_P^*$ the set of reduced roots of 
$A_M$ of the Lie algebra $\mathfrak{u}_P$ of $U_P$.
Let $\srts_P$ be the subset of simple roots of $P$, which is a basis for 
$(\af_P^G)^*$.
Write $\af_{P,+}^*$ for the closure of the Weyl chamber of $P$, i.e.
\[
\aaa_{P,+}^*=\{\lambda\in\aaa_M^*:\sprod{\lambda}{\alpha^\vee}\ge0
\text{ for all }\alpha\in\rts_P\}
=\{\lambda\in\aaa_M^*:\sprod{\lambda}{\alpha^\vee}\ge0\text{ for all }
\alpha\in\srts_P\}.
\]
Denote by $\modulus_P$ the modulus function of $P(\A)$.
Let $\bar\AF_2(P)$ be the Hilbert space completion of
\[
\{\phi\in C^\infty(M(\Q)U_P(\A)\bs G(\A)):\modulus_P^{-\frac12}\phi(\cdot x)\in
L^2_{\disc}(\Ai M(\Q)\bs M(\A)),\ \forall x\in G(\A)\}
\]
with respect to the inner product
\[
(\phi_1,\phi_2)=\int_{\Ai M(\Q)\bU_P(\A)\bs \bG(\A)}\phi_1(g)
\overline{\phi_2(g)}\ dg.
\]
Let $\alpha\in\rts_M$.
We say that two parabolic subgroups $P,Q\in\cP(M)$ are \emph{adjacent} along 
$\alpha$, and write $P|^\alpha Q$, if $\rts_P\cap-\rts_Q=\{\alpha\}$.
Alternatively, $P$ and $Q$ are adjacent if the group $\langle P,Q\rangle$
generated by $P$ and $Q$ belongs to $\cF_1(M)$ (see \eqref{f1} for its
definition).
Any $R\in\cF_1(\M)$ is of the form $\langle P,Q\rangle$, where $P,Q$ are
the elements of $\cP(M)$ contained in $R$. We have $P|^\alpha Q$ with 
$\alpha^\vee\in\rts_P^\vee \cap\af^R_M$.
Interchanging $P$ and $Q$ changes $\alpha$ to $-\alpha$.

For any $P\in\cP(M)$ let $\Ht_P\colon G(\A)\rightarrow\af_P$ be the 
extension of $\Ht_M$ to a left $U_P(\A)$-and right $\K$-invariant map.
Denote by $\cA^2(P)$ the dense subspace of $\bar\cA^2(P)$ consisting of its 
$\K$- and $\zzz$-finite vectors,
where $\zzz$ is the center of the universal enveloping algebra of 
$\mathfrak{g} \otimes \C$.
That is, $\cA^2(P)$ is the space of automorphic forms $\phi$ on 
$U_P(\A)M(\Q)\bs G(\A)$ such that
$\modulus_P^{-\frac12}\phi(\cdot k)$ is a square-integrable automorphic form on
$\Ai M(\Q)\bs M(\A)$ for all $k\in\K$.
Let $\rho(P,\lambda)$, $\lambda\in\af_{M,\C}^*$, be the induced
representation of $G(\A)$ on $\bar\cA^2(P)$ given by
\[
(\rho(P,\lambda,y)\phi)(x)=\phi(xy)e^{\sprod{\lambda}{\Ht_P(xy)-\Ht_P(x)}}.
\]
It is isomorphic to the induced representation 
\[
\Ind_{P(\A)}^{G(\A)}\left(L^2_{\disc}(\Ai M(\Q)\bs M(\A))
\otimes e^{\sprod{\lambda}{\Ht_M(\cdot)}}\right).
\]
For $P,Q\in\cP(M)$ let
\begin{equation}\label{intertw0}
M_{Q|P}(\lambda):\cA^2(P)\to\cA^2(Q),\quad\lambda\in\af_{M,\C}^*,
\end{equation}
be the standard \emph{intertwining operator} \cite[\S 1]{Ar9}, which is the 
meromorphic continuation in $\lambda$ of the integral
\[
[M_{Q|P}(\lambda)\phi](x)=\int_{U_Q(\A)\cap U_P(\A)\bs U_Q(\A)}\phi(nx)
e^{\sprod{\lambda}{\Ht_P(nx)-\Ht_Q(x)}}\ dn, \quad \phi\in\cA^2(P), \ x\in G(\A).
\]
Given $\pi\in\Pi_{\di}(M(\A))$, let $\cA^2_\pi(P)$ be the space of all 
$\phi\in\cA^2(P)$ for which the function 
$M(\A)\ni x\mapsto \delta_P^{-\frac{1}{2}}\phi(xg)$,
$g\in G(\A)$, belongs to the $\pi$-isotypic subspace of the space
$L^2(\Ai M(\Q)\bs M(\A))$.
For any $P\in\cP(\M)$ we have a canonical isomorphism of 
$G(\A_f)\times(\LieG_{\C},\K_\infty)$-modules
\[
j_P:\Hom(\pi,L^2(\Ai M(\Q)\bs M(\A)))\otimes 
\Ind_{P(\A)}^{G(\A)}(\pi)\rightarrow\cA^2_\pi(P).
\]
If we fix a unitary structure on $\pi$ and endow 
$\Hom(\pi,L^2(\Ai M(\Q)\bs M(\A)))$ with the inner product 
$(A,B)=B^\ast A$
(which is a scalar operator on the space of $\pi$), the isomorphism $j_P$ 
becomes an isometry. 

Suppose that $P|^\alpha Q$.
The operator $M_{Q|P}(\pi,s):=M_{Q|P}(s\varpi)|_{\cA^2_\pi(P)}$, where $\varpi\in
\af^\ast_M$ is such that $\langle\varpi,\alpha^\vee\rangle=1$, admits a 
normalization by a global factor
$n_\alpha(\pi,s)$ which is a meromorphic function in $s$. We may write
\begin{equation} \label{normalization}
M_{Q|P}(\pi,s)\circ j_P=n_\alpha(\pi,s)\cdot j_Q\circ(\Id\otimes R_{Q|P}(\pi,s))
\end{equation}
where $R_{Q|P}(\pi,s)=\otimes_v R_{Q|P}(\pi_v,s)$ is the product
of the locally defined normalized intertwining operators and 
$\pi=\otimes_v\pi_v$
\cite[\S 6]{Ar9}, (cf.~\cite[(2.17)]{Mu2}). In many cases, the 
normalizing factors can be expressed in terms automorphic $L$-functions 
\cite{Sha1}, \cite{Sha2}. 

We now turn to the spectral side. Let $L \supset M$ be Levi subgroups in 
$\levis$, $P \in \PPP (M)$, and
let $m=\dim\aaa_L^G$ be the co-rank of $L$ in $G$.
Denote by $\bases_{P,L}$ the set of $m$-tuples $\bss=(\beta_1^\vee,\dots,
\beta_m^\vee)$
of elements of $\rts_P^\vee$ whose projections to $\af_L$ form a basis for 
$\af_L^G$.
For any $\bss=(\beta_1^\vee,\dots,\beta_m^\vee)\in\bases_{P,L}$ let
$\vol(\bss)$ be the co-volume in $\af_L^G$ of the lattice spanned by $\bss$ 
and let
\begin{align*}
\Xi_L(\bss)&=\{(Q_1,\dots,Q_m)\in\cF_1(M)^m: \ \ \beta_i^{\vee}\in\af_M^{Q_i}, 
\, i = 1, \dots, m\}\\&=
\{\langle P_1,P_1'\rangle,\dots,\langle P_m,P_m'\rangle): \ \ 
P_i|^{\beta_i}P_i', \, 
i = 1, \dots, m\}.
\end{align*}

For any smooth function $f$ on $\af_M^*$ and $\mu\in\af_M^*$ denote by 
$D_\mu f$ the directional derivative of $f$ along $\mu\in\af_M^*$.
For a pair $P_1|^\alpha P_2$ of adjacent parabolic subgroups in $\cP(M)$ write
\begin{equation}\label{intertw2}
\delta_{P_1|P_2}(\lambda)=M_{P_2|P_1}(\lambda)D_\varpi M_{P_1|P_2}(\lambda):
\cA^2(P_2)\rightarrow\cA^2(P_2),
\end{equation}
where $\varpi\in\af_M^*$ is such that $\sprod{\varpi}{\alpha^\vee}=1$.
\footnote{Note that this definition differs slightly from the definition of
$\delta_{P_1|P_2}$ in \cite{FLM1}.} Equivalently, writing
$M_{P_1|P_2}(\lambda)=\Phi(\sprod{\lambda}{\alpha^\vee})$ for a
meromorphic function $\Phi$ of a single complex variable, we have
\[
\delta_{P_1|P_2}(\lambda)=\Phi(\sprod{\lambda}{\alpha^\vee})^{-1}
\Phi'(\sprod{\lambda}{\alpha^\vee}).
\]
For any $m$-tuple $\dtup=(Q_1,\dots,Q_m)\in\Xi_L(\bss)$
with $Q_i=\langle P_i,P_i'\rangle$, $P_i|^{\beta_i}P_i'$, denote by 
$\Delta_{\dtup}(P,\lambda)$
the expression
\begin{equation}\label{intertw3}
\frac{\vol(\bss)}{m!}M_{P_1'|P}(\lambda)^{-1}\delta_{P_1|P_1'}(\lambda)M_{P_1'|P_2'}(\lambda) \cdots
\delta_{P_{m-1}|P_{m-1}'}(\lambda)M_{P_{m-1}'|P_m'}(\lambda)\delta_{P_m|P_m'}(\lambda)M_{P_m'|P}(\lambda).
\end{equation}
Recall the (purely combinatorial) map $\dtup_L: \bases_{P,L} \to \FFF_1 (M)^m$ with the property that
$\dtup_L(\bss) \in \Xi_L (\bss)$ for all $\bss \in \bases_{P,L}$ as defined in \cite[pp.\ 179--180]{FLM1}.\footnote{The map $\dtup_L$ depends in fact on the additional choice of
a vector $\underline{\mu} \in (\aaa^*_M)^m$ which does not lie in an explicit finite
set of hyperplanes. For our purposes, the precise definition of $\dtup_L$ is immaterial.}

For any $s\in W(M)$ let $L_s$ be the smallest Levi subgroup in $\levis(M)$
containing $w_s$. We recall that $\aaa_{L_s}=\{H\in\aaa_M\mid sH=H\}$.
Set
\[
\iota_s=\abs{\det(s-1)_{\aaa^{L_s}_M}}^{-1}.
\]
For $P\in\FFF(M_0)$ and $s\in W(M_P)$ let
$M(P,s):\AF^2(P)\to\AF^2(P)$ be as in \cite[p.~1309]{Ar3}.
$M(P,s)$ is a unitary operator which commutes with the operators $\rho(P,\lambda,h)$ for $\lambda\in\iii\aaa_{L_s}^*$.
Finally, we can state the refined spectral expansion.

\begin{theo}[\cite{FLM1}] \label{thm-specexpand}
For any $h\in C_c^\infty(G(\A)^1)$ the spectral side of Arthur's trace formula is given by
\begin{equation}\label{specside1}
J_{\spec}(h) = \sum_{[M]} J_{\spec,M} (h),
\end{equation}
$M$ ranging over the conjugacy classes of Levi subgroups of $G$ (represented by members of $\mathcal{L}$),
where
\begin{equation}\label{specside2}
J_{\spec,M} (h) =
\frac1{\card{W(M)}}\sum_{s\in W(M)}\iota_s
\sum_{\bss\in\bases_{P,L_s}}\int_{\iii(\aaa^G_{L_s})^*}
\Tr(\Delta_{\dtup_{L_s}(\bss)}(P,\lambda)M(P,s)\rho(P,\lambda,h))\ d\lambda
\end{equation}
with $P \in \PPP(M)$ arbitrary.
The operators are of trace class and the integrals are absolutely convergent
with respect to the trace norm and define distributions on $\Co(G(\A)^1)$.
\end{theo}

\section{Logarithmic derivatives of local intertwining operators}
\label{sec-locint}
\setcounter{equation}{0}

In this section we prove some auxiliary results about local intertwining 
operators. To begin with we recall some facts concerning local intertwining
operators and normalizing factors. Let $M\in\cL$ and $P,Q\in\cP(M)$. Let
$v$ be a finite place of $\Q$ and $K_v$ an open compact subgroup of $G(\Q_v)$. 
Let $\pi_v\in\Pi(M(\Q_v))$. Given $\lambda\in
\af^\ast_{M,\C}$, let $(I_P^G(\pi_v,\lambda),\H_P(\pi_v))$ denote the induced
representation. Let $\H_P^0(\pi_v)\subset\H_P(\pi_v)$ be the subspace of 
$K_v$-finite functions. Let 
\[
J_{Q|P}(\pi_v,\lambda)\colon \H_P^0(\pi_v)\to\H_Q^0(\pi_v)
\]
be the local intertwining operator between the induced representations 
$I_P^G(\pi_v,\lambda)$ and $I_Q^G(\pi_v,\lambda)$ \cite{Sha1}. It is proved 
in \cite{Ar4}, \cite[Lecture 15]{CLL} that there exist scalar valued 
meromorphic functions $r_{Q|P}(\pi_v,\lambda)$ of $\lambda\in\af_{P,\C}^\ast$ 
such that the normalized intertwining operators
\begin{equation}\label{norm-inter}
R_{Q|P}(\pi_v,\lambda)=r_{Q|P}(\pi_v,\lambda)^{-1}J_{Q|P}(\pi_v,\lambda)
\end{equation}
satisfy the conditions $(R_1)-(R_8)$ of Theorem~2.1 of
\cite{Ar4}. We recall some facts about the local normalizing factors. First
assume that $v$ is a finite valuation of $\Q$ with corresponding prime number $q_v\in\N$. Furthermore assume that 
$\dim(\af_M/\af_G)=1$ and $\pi_v$ is square integrable. Let $P\in\cP(M)$ and
let $\alpha$ be the unique simple root of $(P,A_M)$. Then Langlands
\cite[Lecture 15]{CLL} has shown that there exists a rational function 
$V_P(\pi_v,z)$ of one variable such that 
\begin{equation}\label{rat-funct}
r_{\ov P|P}(\pi_v,\lambda)=V_P(\pi_v,q_v^{-\lambda(\widetilde\alpha)}),
\end{equation}
where $\widetilde\alpha\in\af_M$ is uniquely determined by $\alpha$. 
For the construction of $V_P$ see also \cite[Sect. 3]{Mu2}. We need the 
following lemma.
\begin{lem}\label{lem-loc-int}
Let $M\in\cL$ be such that $\dim(\af_M/\af_G)=1$.
There exists $C>0$ such that for all $P\in\cP(M)$ and all 
$\pi\in\Pi(M(\Q_v))$ the number of zeros of the rational function 
$V_P(\pi,z)$ is less than or equal to $C$.
\end{lem}
\begin{proof}
First assume that $\pi$ is square integrable. Then the
corresponding statement for the number of poles was proved in 
\cite[Lemma~3.1]{Mu2}. However, by \cite[(3.6)]{Mu2} the number of zeros of 
$V_P(\pi,z)$ agrees with the number of poles $V_P(\pi,z)$. Now let $\pi$ be 
tempered. It is known that $\pi$ is an irreducible constituent of an induced 
representation $I_R^M(\sigma)$ where $M_R$ is an admissible Levi subgroup of 
$M$ and $\sigma\in\Pi(M_R(\Q_v))$ is square integrable modulo $A_R$. Then by
\cite[(2.2)]{Ar4} we are reduced to the square integrable case. In general 
$\pi$ is a Langlands quotient of of an induced representation 
$I_R^M(\sigma,\mu)$, where $M_R$ is an admissible Levi subgroup of $M$, 
$\sigma\in\Pi_{\temp}(M_R(\Q_v))$, and $\mu$ is a point in the chamber
of $\af_R^\ast/\af_M^\ast$ attached to $R$. Now we use \cite[(2.3)]{Ar4}
to reduce the proof to the tempered case.
\end{proof}

The main goal of this section is to estimate the logarithmic derivatives
of the normalized intertwining operators $R_{Q|P}(\pi,\lambda)$. For $\bG=
\GL(n)$ such estimates were derived in \cite[Proposition~0.2]{MS}. The proof
depends on a weak version of the 
Ramanujan conjecture, which is not available in general. Therefore we will
establish only an integrated version of it, which however, is sufficient for
our purpose. For $\pi\in\Pi_{\di}(M(\A))$ denote by $\H_P(\pi)$ the Hilbert
space of the induced representation $I_P^G(\pi,\lambda)$. Furthermore, for 
an open compact subgroup $K_f\subset G(\A_f)$ and $\nu\in\Pi(\K_\infty)$, denote
by $\H_P(\pi)^{K_f}$ the subspace of vectors, which are invariant under $K_f$
and let $\H_P(\pi)^{K_f,\nu}$ denote the $\nu$-isotypical subspace of 
$\H_P(\pi)^{K_f}$. Let $P,Q\in\cP(M)$ be adjacent parabolic subgroups. Then 
$R_{Q|P}(\pi,\lambda)$ depends on a single variable $s\in\C$ and we will
write 
\[
R^\prime_{Q|P}(\pi,s_0):=\frac{d}{ds}R_{Q|P}(\pi,s)\big|_{s=s_0}
\]
for any regular $s_0\in\C$.
\begin{prop}\label{log-der}
Let $M\in\cL$, and let $P,Q\in\cP(M)$ be adjacent parabolic subgroups. Let
$K_f\subset \bG(\A_f)$ be an open compact subgroup and let 
$\nu\in\Pi(\K_\infty)$.
Then there exists $C>0$ such that
\begin{equation}\label{int-log-der}
\int_{\R}\Big\|R_{Q|P}(\pi,it)^{-1}R^\prime_{Q|P}(\pi,it)\big|_{\H_P(\pi)^{K_f,\nu}}
\Big\|(1+|t|)^{-1}dt\le C
\end{equation}
for all $\pi\in\Pi_{\di}(M(\A))$ with $\H_P(\pi)^{K_f,\nu}\neq 0$. 
\end{prop}
\begin{proof}
We may assume that $K_f$ is factorisable, i.e., $K_f=\prod_v K_v$. Let $S$ be
the finite set of finite places such that $K_v$ is not hyperspecial. Since
$P$ and $Q$ are adjacent, by standard properties of normalized intertwining
operators \cite[Theorem~2.1]{Ar4} we may assume that $P$ is a maximal
parabolic subgroup and $Q=\overline P$, the opposite parabolic subgroup to
$P$. By
\cite[Theorem 2.1, (R8)]{Ar4}, $R_{Q|P}(\pi_v,s)^{K_v}$ is independent of $s$ if
$v$ is finite and $v\notin S$. Thus we have

\begin{equation}\label{log-der-inter1}
\begin{split}
R_{\oP|P}(\pi,s)^{-1}R^\prime_{\oP|P}(\pi,s)\big|_{H_P(\pi)^{K_f,\nu}}&=
R_{\oP|P}(\pi_\infty,s)^{-1}R^\prime_{\oP|P}(\pi_\infty,s)\big|_{\H_P(\pi_\infty)^\nu}\\
&+\sum_{v\in S}R_{\oP|P}(\pi_v,s)^{-1}R^\prime_{\oP|P}(\pi_v,s)\big|_{\H_P(\pi_v)^{K_v}}
\end{split}
\end{equation}
This reduces our problem to the operators at the local places. We distinguish
between the archimedean and the non-archimedean case.

{\bf Case 1:} $v<\infty$. Define $A_v\colon \C\to \End(\cH_P(\pi_v)^{K_v})$ by
\[
A_v(q_v^{-s}):=R_{\oP|P}(\pi_v,s)\big|_{\H_P(\pi_v)^{K_v}}.
\]
This is a meromorphic function with values in the space of endomorphisms of a finite
dimensional vector space. It has the following properties. 
By the unitarity of $R_{\oP|P}(\pi_v,it)$, $t\in\R$, it follows that $A_v(z)$ is
holomorphic for $z\in S^1$ and satisfies $\|A_v(z)\|\le 1$, $|z|=1$. By
\cite[Theorem~2.1]{Ar4}, the matrix coefficients of $A_v(z)$ are rational 
functions. As in
\cite[(14)]{FiLaMu} we get
\begin{equation}
\int_{\R} \Big\|R_{\oP|P}(\pi_v,it)^{-1}R^\prime_{\oP|P}(\pi_v,it)
\big|_{\H_P(\pi_v)^{K_v}}
\Big\|(1+|t|)^{-1}dt\le\left(2+\frac{1}{2}\log q_v\right)\int_{S^1}\left\| 
A_v^\prime(z)\right\|\;|dz|.
\end{equation}
As explained above, $A_v$ satisfies the assumptions of 
\cite[Corollary~5.18]{FiLaMu}. Denote by $z_1,...,z_m\in \C\setminus S^1$ be the poles
of $A_v(z)$. Then $(z-z_1)\cdots(z-z_m)A_v(z)$ is a polynomial of degree $n$
with coefficients in $\End(\cH_{\pi_v}^{K_v})$ and by 
\cite[Corollary~5.18]{FiLaMu} we get
\[
\|A_v^\prime(z)\|\le\max\left(\max(n-m,0)+\sum_{j\colon |z_j|>1}
\frac{|z_j|^2-1}{|z_j-z|^2},\; \sum_{j\colon |z_j|<1}\frac{1-|z_j|^2}{|z_j-z|^2}
\right),\quad z\in S^1.
\]
To estimate the right hand side, we need the following lemma.
\begin{lem}\label{lem:poisson-kern}
There exists $C>0$ such that for all $z_0\in\C\setminus S^1$
\begin{equation}
\int_{S^1}\frac{|z_0|^2-1}{|z-z_0|^2}|dz|\le C.
\end{equation}
\end{lem}
\begin{proof}
First consider the case $|z_0|<1$. If we change variables by $z=e^{i\theta}$, 
then up to a constant, the integral equals the integral of the Poisson kernel
for the unit disc over the unit circle. From the theory of the Poisson kernel it
is well known that, as a function of $z_0$, this integral is a continuous 
function on the closed unit disc which is equal to 1 on the unit circle. Thus 
the lemma holds for $|z_0|<1$. For $|z_0|>1$ we use that
\[
\int_0^{2\pi}\frac{|z_0|^2-1}{|e^{i\theta}-z_0|^2}d\theta= 
\int_0^{2\pi}\frac{|z_0^{-1}|^2-1}{|e^{i\theta}-z_0^{-1}|^2}d\theta,
\]
which reduces the problem to the previous case.
\end{proof}

Next we estimate $m$ and $n$. First consider $m$. Let $J_{\oP|P}(\pi_v,s)$ be
the usual intertwining  operator so that
\[
R_{\oP|P}(\pi_v,s)=r_{\oP|P}(\pi_v,s)^{-1}J_{\oP|P}(\pi_v,s),
\]
where $r_{\oP|P}(\pi_v,s)$ is the normalizing factor \cite{Ar4}. By 
\cite[Theorem~2.2.2]{Sha1} there exists a polynomial $p(z)$ with $p(0)=1$ 
whose degree is bounded independently of $\pi_v$, such that 
$p(q_v^{-s})J_{\oP|P}(\pi_v,s)$ is holomorphic on $\C$. To deal with the normalizing factor we use \eqref{rat-funct} together with Lemma~\ref{lem-loc-int} to
count the
number of poles of $r_{\oP|P}(\pi_v,s)^{-1}$. This leads to a bound for $m$
which depends only on $\bG$. To estimate $n$ we fix an open compact subgroup
$K_v$ of $\bG(\Q_v)$. Our goal is now to estimate the order at $\infty$ of any
matrix coefficient of $R_{\oP|P}(\pi_v,s)$ regarded as a function of
$z=q_v^{-s}$.
Write $\pi_v$ as Langlands quotient $\pi_v=J_R^M(\delta_v,\mu)$ where $R$ is a
parabolic subgroup of $M$, $\delta_v$ a square integrable representation of
$M_R(\Q_v)$ and $\mu\in(\af^\ast_R/\af^\ast_M)_\C$ with $\Re(\mu)$ in the chamber
attached to $R$. Then by \cite[p. 30]{Ar4} we have
\[
 R_{\oP|P}(\pi_v,s)=R_{\oP(R)|P(R)}(\delta_v,s+\mu)
\]
with respect to the identifications described in \cite[p. 30]{Ar4}. Here $s$ is
identified with a point in $(\af_R^\ast/\af_G^\ast)_\C$ with respect to the 
canonical embedding $\af_M^\ast\subset\af_G^\ast$. Using again the factorization
of normalized intertwining operators we reduce the problem to the case of a 
square-integrable representation $\delta_v$. Moreover $\delta_v$ has to 
satisfy $[I_P^G(\delta_v,s)\big|_{K_v}\colon \one]\ge 1$. By \cite[Lemma~1]{Si1}
we have 
\begin{equation}\label{multipl}
[I_P^G(\delta_v,s)\big|_{K_v}\colon \one]\ge 1\Leftrightarrow
[\delta_v\big|_{K_v\cap M(\Q_v)}\colon\one]\ge 1
\end{equation}
Let $\Pi_2(M(\Q_v))$ be the space of square-integrable representations of
$M(\Q_v)$. This space has a manifold structure \cite{HC2}, \cite{Si1}.
By \cite[Theorem~10]{HC2} the set of square-integrable representations 
$\Pi_2(M(\Q_v),K_v)$ of $M(\Q_v)$ with $[\delta_v\big|_{K_v\cap M(\Q_v)}
\colon\one]\ge 1$ is a compact subset of $\Pi_2(M(\Q_v))$. Under the canonical 
action of $i\af_M$, the set $\Pi_2(M(\Q_v),K_v)$ decomposes into a finite number
of orbits. For $\mu\in i\af_M$ and $\delta_v\in\Pi_2(M(\Q_v),K_v)$, let 
$(\delta_v)_\mu\in \Pi_2(M(\Q_v),K_v)$ be the result of the canonical action.
Then it follows that
\[
R_{\oP|P}((\delta_v)_\mu,\lambda)=R_{\oP|P}(\delta_v,\lambda+\mu).
\]
In this way our problem is finally reduced to the consideration of
the matrix coefficients of $R_{\oP|P}(\pi_v,s)\big|_{K_v}$ for a finite number of
representations $\pi_v$. This implies that $n$ is bounded by a constant which is
independent of $\pi_v$. Together with Lemma \ref{lem:poisson-kern} it follows 
that for each finite place
$v$ of $\Q$ and each open compact subgroup $K_v$ of $G(\Q_v)$ there exists 
$C_v>0$ such that
\begin{equation}\label{loc-int-est1}
\int_{\R} \Big\|R_{\oP|P}(\pi_v,it)^{-1}R^\prime_{\oP|P}(\pi_v,it)\big|_{\H_P(\pi_v)^{K_v}}
\Big\|(1+|t|)^{-1}dt\le C_v
\end{equation}
for all $\pi_v\in\Pi(M(\Q_v))$ with $I_P^G(\pi_v)\big|_{\H_P(\pi_v)^{K_v}}\neq0$.

{\bf Case 2:} $v=\infty$. As above let $M\in\cL$ with $\dim(\af_M/\af_G)=1$
and $P\in\cP(M)$. Let $\pi_\infty\in\Pi(M(\R))$ and $\nu\in\Pi(\K_\infty)$.
As explained in \cite[Appendix]{MS}, there exist $w_1,...,w_r\in\C$ and 
$m\in\N$ such that the poles of $R_{\oP|P}(\pi_\infty,s)\big|_{\H_P(\pi_\infty)^\nu}$
are contained in $\cup_{j=1}^r\{w_j-k\colon k=1,...,m\}$. Moreover, by
\cite[Proposition~A.2]{MS} there exists $c>0$ which depends only on $G$, such 
that
\begin{equation}\label{est-const}
r\le c,\quad m\le c(1+\|\nu\|).
\end{equation}
Let $A\colon \C\to \H_{\pi_\infty}(\nu)$ be defined by 
\[
A(z):=R_{\oP|P}(\pi_\infty,z)\big|_{\H_P(\pi_\infty)^\nu}
\]
and let $b(z)=\prod_{j=1}^r\prod_{k=1}^m(z-w_j+k)$. Then it follows from ($R_6$) of
\cite[Theorem~2.1]{Ar4} that  $b(z)A(z)$ is a polynomial function. Moreover,
by unitarity of $R_{\oP|P}(\pi_\infty,it)$, $t\in\R$, we have $\|A(it)\|\le 1$. 
Thus $A(z)$ satisfies the assumptions of \cite[Lemma~5.19]{FiLaMu}. Thus by
\cite[Lemma~5.19]{FiLaMu} and \eqref{est-const} we get
\begin{equation}\label{log-esti-archim}
\begin{split}
\int_\R\Big\|R_{\oP|P}(\pi_\infty,it)^{-1}R_{\oP|P}^\prime(\pi_\infty,it)
\big|_{\H_P(\pi_\infty)^\nu}&\Big\|(1+|t|^2)^{-1} dt=\int_\R\| A^\prime(it)\|(1+|t|^2)^{-1} 
dt\\
&\ll\sum_{j=1}^r\sum_{k=1}^m
\frac{|\Re(w_j)-k|+1}{(|\Re(w_j)-k|+1)^2+(\Im(w_j))^2}\\
&\ll 1+\|\nu\|.
\end{split}
\end{equation} 
Combining \eqref{log-der-inter1}, \eqref{loc-int-est1} and \eqref{log-esti-archim},
the proposition follows.
\end{proof}

\section{The analytic torsion}\label{sec-analtor}
\setcounter{equation}{0}

As before we consider a reductive quasi-split algebraic group $\bG$ over $\Q$,
an open compact subgroup $K_f$ of $\bG(\A_f)$ and the adelic quotient 
\[
X:=X(K_f),
\]
defined by \eqref{adelic-quot}.
For simplicity we assume that $\bG$ is semisimple and simply connected. Let
$\K_\infty\subset G(\R)$ be a maximal compact subgroup and let 
$\widetilde X=\bG(\R)/\K_\infty$. Then by strong approximation we have
\begin{equation}\label{ad-qu1}
X=\Gamma\bs\widetilde X,
\end{equation}
where $\Gamma=(\bG(\R)\times K_f)\cap \bG(\Q)$. In general, there are finitely
many arithmetic subgroups $\Gamma_j\subset \bG(\R)^1$, $j=1,...,l$, such that
$X(K_f)$ is the disjoint union of $\Gamma_j\bs \widetilde X$, $j=1,...,l$. 

\subsection{The Hodge-Laplace operator and heat kernels}
Let
$\tau$ be an irreducible finite-dimensional complex representation of $\bG(\R)$
on $V_\tau$. Let $E_\tau$ be the flat vector bundle over $X$ associated to the 
restriction of $\tau$ to $\Gamma$. Let $\widetilde E^\tau$ be the homogeneous
vector bundle associated to $\tau|_{\K_\infty}$ and let 
$E^\tau:=\Gamma\bs \widetilde E^\tau$. There is a canonical isomorphism
\begin{equation}\label{iso-vb}
E^\tau\cong E_\tau
\end{equation}
\cite[Proposition~3.1]{MM}. By \cite[Lemma~3.1]{MM}, there exists an 
inner product $\left<\cdot,\cdot\right>$ on $V_{\tau}$,
which is unique up to scaling, which satisfies
\begin{enumerate}
\item $\left<\tau(Y)u,v\right>=-\left<u,\tau(Y)v\right>$ for all 
$Y\in\mathfrak{k}$, $u,v\in V_{\tau}$
\item $\left<\tau(Y)u,v\right>=\left<u,\tau(Y)v\right>$ for all 
$Y\in\mathfrak{p}$, $u,v\in V_{\tau}$.
\end{enumerate}
Such an inner product is called admissible. We fix an 
admissible inner product. Since $\tau|_{\K_\infty}$ is unitary with respect to 
this inner product, it induces a metric on $E^{\tau}$, and by \eqref{iso-vb} 
on $E_\tau$, which we also call 
admissible. Let $\Lambda^{p}(E_{\tau})=\Lambda^pT^*(X)\otimes E_\tau$. Let
\begin{align}\label{repr4}
\nu_{p}(\tau):=\Lambda^{p}\Ad^{*}\otimes\tau:\:\K_\infty\rightarrow\GL
(\Lambda^{p}\mathfrak{p}^{*}\otimes V_{\tau}).
\end{align}
Then by \eqref{iso-vb} there is a canonical isomorphism
\begin{align}\label{pforms}
\Lambda^{p}(E_{\tau})\cong\Gamma\backslash(\bG(\R)\times_{\nu_{p}(\tau)}
(\Lambda^{p}\mathfrak{p}^{*}\otimes V_{\tau})).
\end{align}
of locally homogeneous vector bundles. 
Let  $\Lambda^{p}(X,E_{\tau})$ be the space the smooth $E_{\tau}$-valued 
$p$-forms on $X$. Let
\begin{align}\label{globsect}
\begin{split}
C^{\infty}(\bG(\R),\nu_p(\tau)):=\{f:\bG(\R)\rightarrow \Lambda^{p}\mathfrak{p}^{*}\otimes V_{\tau}\colon f\in C^\infty,\:
f(g&k)=\nu_p(\tau)(k^{-1})f(g),\\
&\forall g\in \bG(\R), \,\forall k\in \K_\infty\},
\end{split}
\end{align}
and
\begin{align}\label{globsect1}
C^{\infty}(\Gamma\backslash \bG(\R),\nu_p(\tau)):=
\left\{f\in C^{\infty}(\bG(\R),\nu_p(\tau))\colon 
f(\gamma g)=f(g),\;\forall g\in \bG(\R),\; \forall \gamma\in\Gamma\right\}.
\end{align}
The isomorphism \eqref{pforms} induces an isomorphism
\begin{align}\label{isoschnitte}
\Lambda^{p}(X,E_{\tau})\cong C^{\infty}(\Gamma\backslash \bG(\R),\nu_{p}(\tau)).
\end{align}
A corresponding isomorphism also holds for the spaces of $L^{2}$-sections.
Let $\Delta_{p}(\tau)$ be the 
Hodge-Laplacian on $\Lambda^{p}(X,E_{\tau})$ with respect to the admissible 
metric in $E_\tau$. Let $R_\Gamma$ denote the right regular representation 
of $\bG(\R)$ in $L^2(\Gamma\bs \bG(\R))$. Let $\Omega\in\cZ(\gf_\C)$ be the
Casimir element. By \cite[(6.9)]{MM} it follows that 
with respect to the isomorphism \eqref{isoschnitte} one has
\begin{equation}\label{laplace1}
\Delta_{p}(\tau)=-R_\Gamma(\Omega)+\tau(\Omega)\Id.
\end{equation}
Let $\widetilde E_\tau\to \widetilde X$ be the lift of $E_\tau$ to 
 $\widetilde X$.
There is a canonical isomorphism 
\begin{equation}\label{iso-vbcov}
\Lambda^p(\widetilde X,\widetilde E_\tau)\cong C^\infty(\bG(\R),\nu_p(\tau)).
\end{equation}
Let $\widetilde\Delta_p(\tau)$ be the lift of $\Delta_p(\tau)$ to 
$\widetilde X$.
Then again it follows  from \cite[(6.9)]{MM} that with respect to the
isomorphism \eqref{iso-vbcov} we have
\begin{equation}\label{kuga}
\widetilde \Delta_p(\tau)=-R(\Omega)+\tau(\Omega)\Id.
\end{equation}
Let $e^{-t\widetilde\Delta_p(\tau)}$ be the 
corresponding heat semigroup. Regarded as an operator in the Hilbert space
$L^2(\bG(\R),\nu_p(\tau))$, it is a convolution operator with kernel
\begin{align}\label{DefHH}
H^{\tau,p}_t\colon \bG(\R)\to\End(\Lambda^p\mathfrak p^*\otimes
V_\tau)
\end{align} 
which belongs to $C^\infty\cap L^2$ and  satisfies the covariance property
\begin{equation}\label{covar}
H^{\tau,p}_t(k^{-1}gk')=\nu_p(\tau)(k)^{-1} H^{\tau,p}_t(g)\nu_p(\tau)(k')
\end{equation}
with respect to the representation \eqref{repr4}. Moreover, for all $q>0$ we 
have 
\begin{equation}\label{schwartz1}
H^{\tau,p}_t \in (\mathcal{C}^q(\bG(\R))\otimes
\End(\Lambda^p\pf^*\otimes V_\tau))^{\K_\infty\times \K_\infty}, 
\end{equation}
where $\mathcal{C}^q(G(\R))$ denotes Harish-Chandra's $L^q$-Schwartz space
(see \cite[Sect. 4]{MP2}). 
Let $h^{\tau,p}_t\in C^\infty(\bG(\R))$ be defined by
\begin{equation}\label{tr-kern}
h^{\tau,p}_t(g)=\Tr H^{\tau,p}_t(g),\quad g\in \bG(\R).
\end{equation}
Let $\chi_{K_f}$ be defined by \eqref{char-funct}. We define
$\phi^{\tau,p}_t\in C^\infty(\bG(\A))$ by
\begin{equation}\label{extension5}
\phi^{\tau,p}_t(g_\infty g_f):=h^{\tau,p}_t(g_\infty)\chi_{K_f}(g_f)
\end{equation}
for $g_\infty\in \bG(\R)$ and $g_f\in \bG(\A_f)$. 
Following the definition \ref{reg-trace}, we define the regularized trace of
$e^{-t\Delta_p(\tau)}$ by
\begin{equation}\label{regtr-heat}
\Tr_{\reg}\left(e^{-t\Delta_p(\tau)}\right):=J_{\geo}(\phi^{\tau,p}_t).
\end{equation}

\subsection{The asymptotic behavior of the regularized trace}
Our next goal is to determine the asymptotic behavior of 
$\Tr_{\reg}\left(e^{-t\Delta_p(\tau)}\right)$ as $t\to0$ and $t\to\infty$.
For $\bG=\GL(n)$ or $\bG=\SL(n)$ this has been carried out in \cite{MM1}. 
Concerning the asymptotic behavior as $t\to0$, we have
\begin{lem}\label{lem:asymp-exp} 
There exist $a_j,b_{ij}\in\C$, $j\in\N_0$, $i=0,...,r_j$, such that 
as $t\to0$, there is an asymptotic expansion
\[
\Tr_{\reg}\left(e^{-t\Delta_p(\tau)}\right)\sim t^{-d/2}\sum_{j=0}^\infty a_j t^{j}+
t^{-(d-1)/2}\sum_{j=0}^\infty\sum_{i=0}^{r_j} b_{ij}t^{j/2} (\log t)^i.
\]
\end{lem}
\begin{proof}
Let $\Omega_{\K_\infty}\in\cZ(\kf_\C)$ be the Casimir element of $\K_\infty$. For
$p=0,...,n$ put
\[
E_p(\tau):=\nu_p(\tau)(\Omega_{\K_\infty}),
\]
which we regard as an endomorphism of $\Lambda^p\pg^*\otimes V_\tau$.
It defines
an endomorphism of $\Lambda^pT^\ast(X)\otimes E_\tau$. By 
\cite[Proposition~1.1]{Mia} and \eqref{kuga} we have
\[
\widetilde \Delta_p(\tau)=\widetilde\Delta_{\nu_p(\tau)}+\tau(\Omega)\Id-E_p(\tau).
\]
Let $\nu_p(\tau)=\oplus_{\nu\in\Pi(\K_\infty)} m(\nu)\nu$ be the decomposition
of $\nu_p(\tau)$ into irreducible representations. This induces a
corresponding decomposition of the homogeneous vector bundle
\begin{equation}\label{vb-decomp}
\widetilde E_{\nu_p(\tau)}=\bigoplus_{\nu\in\Pi(\K_\infty)}m(\nu)\widetilde E_\nu.
\end{equation}
With respect to this decomposition we have 
\begin{equation}\label{decomp1}
E_p(\tau)=\bigoplus_{\nu\in\Pi(\K_\infty)}m(\nu)
\sigma(\Omega_{\K_\infty})\Id_{V_\nu},
\end{equation}
where $\nu(\Omega_{\K_\infty})$ is the Casimir eigenvalue of $\nu$ and 
$V_\nu$ the corresponding  representation space. Let $\widetilde\Delta_\nu$
be the Bochner-Laplace operator associated to $\nu$. By \eqref{vb-decomp}
we get a corresponding decomposition of 
$C^\infty(\widetilde X,\widetilde E_{\nu_p(\tau)})$ and with respect to this
decomposition we have
\begin{equation}\label{decomp2}
\widetilde\Delta_{\nu_p(\tau)}=\bigoplus_{\nu\in\Pi(\K_\infty)} m(\nu)
\widetilde\Delta_\nu.
\end{equation}
This shows that 
$\widetilde\Delta_{\nu_p(\tau)}$ commutes with $E_p(\tau)$, and therefore
we have
\begin{equation}\label{equ-kernel}
H^{\tau,p}_t=e^{-t(\tau(\Omega)-E_p(\tau))}\circ H^{\nu_p(\tau)}_t.
\end{equation}
Let $H_t^\nu$ be the kernel of $e^{-t\widetilde\Delta_\nu}$ and $h_t^\nu:=\Tr\circ
H_t^\nu$. Then it follows from \eqref{decomp1} and \eqref{decomp2} that
\begin{equation}\label{equ-kernel2}
h^{\tau,p}_t=\sum_{\nu\in\Pi(\K_\infty)} m(\nu)
e^{-t(\tau(\Omega)-\nu(\Omega_{\K_\infty}))}h^{\nu}_t,
\end{equation}
Using the definition of $\phi^{\tau,p}_t$ and $\phi^\nu_t$, respectively,  we get
\[
\begin{split}
\Tr_{\reg}\left(e^{-t\Delta_p(\tau)}\right)=J_{\geo}(\phi^{\tau,p}_t)&=
\sum_{\nu\in\Pi(\K_\infty)}m(\nu)e^{-t(\tau(\Omega)-\nu(\Omega_{\K_\infty}))}J_{\geo}(\phi^\nu_t)\\
&=\sum_{\nu\in\Pi(\K_\infty)}m(\nu)e^{-t(\tau(\Omega)-\nu(\Omega_{\K_\infty}))}
\Tr_{\reg}\left(e^{-t\Delta_\nu}\right).
\end{split}
\]
Applying Theorem~\ref{thm:asymptotic:geom} concludes the proof.
\end{proof}

To study the asymptotic behavior as $t\to\infty$ we use the Arthur trace 
formula \eqref{tracef1}. By \cite[Corollary~1]{FLM1}, $J_{\spec}$  is a 
distribution on 
$\Co(\bG(\A);K_f)$ (see section \ref{sec-prelim} for its definition) and by \cite[Theorem~7.1]{FL1}, $J_{\geo}$ is continuous on
$\Co(\bG(\A);K_f)$. This implies that \eqref{tracef1} holds for 
$\phi_t^{\tau,p}$ and we have
\begin{equation}\label{regtrace1a}
\Tr_{\reg}\left(e^{-t\Delta_p(\tau)}\right)=J_{\spec}(\phi_t^{\tau,p}).
\end{equation}
Now we apply Theorem~\ref{thm-specexpand} to study the asymptotic 
behavior as $t\to\infty$ of the right hand side.  First we have 
\begin{equation}
J_{\spec}(\phi_t^{\tau,p})=\sum_{[M]} J_{\spec,M} (\phi_t^{\tau,p}),
\end{equation}
where the sum ranges over the conjugacy classes of Levi subgroups of $G$ and 
$J_{\spec,M} (\phi_t^{\tau,p})$ is given by \eqref{specside2}.

To analyze these terms, we proceed as in \cite[Section 13]{MM1}.
Let $M\in\cL$ and $P\in\cP(M)$. We use the notation 
introduced in section \ref{sec-trform}. Recall that the discrete subspace
$L^2_{\di}(\Ai M(\Q)\bs M(\A))$ splits as the completed direct sum of its
$\pi$-isotypic components for $\pi\in\Pi_{\di}(M(\A))$. We have a corresponding
decomposition of $\bar{\cA}^2(P)$ as a direct sum of Hilbert spaces
$\hat\oplus_{\pi\in\Pi_{\di}(M(\A))}\bar{\cA}^2_\pi(P)$. Similarly, we have the
algebraic direct sum decomposition
\[
\cA^2(P)=\bigoplus_{\pi\in\Pi_{\di}(M(\A))}\cA^2_\pi(P),
\]
where $\cA^2_\pi(P)$ is the ${\bf K}$-finite part of $\bar{\cA}^2_\pi(P)$. 
Let $\cA^2_\pi(P)^{K_f}$ be the subspace of $K_f$-invariant functions in
$\cA^2_\pi(P)$, and for any $\sigma\in\Pi(\K_\infty)$ let 
$\cA^2_\pi(P)^{K_f,\sigma}$ be the $\sigma$-isotypic subspace of 
$\cA^2_\pi(P)^{K_f}$. Recall that $\cA^2_\pi(P)^{K_f,\sigma}$ is finite dimensional
\cite[Prop. 3.5]{Mu1}.

For $P,Q\in\cP(M)$ let $M_{Q|P}(\lambda)$ be the intertwining operator 
\eqref{intertw0}. Denote by
$M_{Q|P}(\pi,\lambda)$ the restriction of $M_{Q|P}(\lambda)$ to
$\cA^2_\pi(P)$. Recall that the operator $\Delta_{\mathcal{X}}(P,\lambda)$, which 
appears in the formula \eqref{specside2}, is defined by \eqref{intertw3}.
Its definition involves the intertwining operators $M_{Q|P}(\lambda)$. If we 
replace $M_{Q|P}(\lambda)$ by its restriction $M_{Q|P}(\pi,\lambda)$ to
$\cA^2_\pi(P)$, we obtain the restriction $\Delta_{\mathcal{X}}(P,\pi,\lambda)$ of
$\Delta_{\mathcal{X}}(P,\lambda)$ to $\cA^2_\pi(P)$. Similarly, let $\rho_\pi(P,\lambda)$
be the induced representation in $\bar{\cA}^2_\pi(P)$. 
Fix $\beta\in\bases_{P,L_s}$ and $s\in W(M)$.  Then for the integral 
on the right of \eqref{specside2} with $h=\phi_t^{\tau,p}$ we get
\begin{equation}\label{specside3}
\sum_{\pi\in\Pi_{\di}(M(\A))}\int_{i(\af^G_{L_s})^*}\Tr\left(
\Delta_{\dtup_{L_s}(\bss)}(P,\pi,\lambda)M(P,\pi,s)\rho_\pi(P,\lambda,\phi^{\tau,p}_t)
\right)\;d\lambda.
\end{equation}
In order to deal with the integrand, we need the following result. Let $\pi$
be a unitary  admissible representation of $\bG(\R)$. Let $A\colon \H_\pi\to
\H_\pi$ be a bounded operator which is an intertwining operator for 
$\pi|_{\K_\infty}$. Then $A\circ\pi(h_t^\nu)$ is a finite rank operator. Define
an operator $\tilde A$ on $\H_\pi\otimes V_\nu$ by $\tilde A:=A\otimes\Id$.
Then by \cite[(9.13)]{MM1} we have
\begin{equation}\label{TrFT}
\Tr(A\circ\pi(h_t^\nu))=e^{t(\pi(\Omega)-\nu(\Omega_{\K_\infty}))}
\Tr\left(\tilde A|_{(\H_\pi\otimes V_\nu)^{\K_\infty}}\right).
\end{equation}
We will apply this to the induced representation $\rho_\pi(P,\lambda)$. 
Let $P,Q\in\cP(M)$ and $\nu\in\Pi(\K_\infty)$.  Assume that 
$(\cA^2_\pi(P)^{K_f}\otimes V_{\nu})^{\K_\infty}\neq0$. Denote by 
$\widetilde M_{Q|P}(\pi,\nu,\lambda)$ the restriction of 
\[
M_{Q|P}(\pi,\lambda)\otimes\Id\colon \cA^2_\pi(P)\otimes V_{\nu}\to 
\cA^2_\pi(P)\otimes V_{\nu}
\]
to $(\cA^2_\pi(P)^{K_f}\otimes V_{\nu})^{\K_\infty}$. Denote by 
$\widetilde\Delta_{\dtup_{L_s}(\bss)}(P,\pi,\nu,\lambda)$ and 
$\widetilde M(P,\pi,\nu,s)$ the corresponding restrictions. 
Let $m(\pi)$ denote the multiplicity with which $\pi$ occurs in the regular
representation of $M(\A)$ in $L^2_{\di}(\Ai M(\Q)\bs M(\A))$. Then
\begin{equation}\label{iso-ind}
\rho_\pi(P,\lambda)\cong \oplus_{i=1}^{m(\pi)}\Ind_{P(\A)}^{G(\A)}(\pi,\lambda).
\end{equation}
Fix positive restricted roots of $\af_P$ and let $\rho_{\af_P}$ denote the 
corresponding half-sum of these roots. For $\xi\in \Pi(M(\R))$ and
$\lambda\in\af^\ast_P$ let 
\[
\pi_{\xi,\lambda}:=\Ind_{P(\R)}^{G(\R)}(\xi\otimes e^{i\lambda})
\]
be the unitary induced representation. Let $\xi(\Omega_M)$ be the Casimir
eigenvalue of $\xi$. Define a constant $c(\xi)$ by
\begin{equation}\label{casimir4}
c(\xi):=-\langle\rho_{\af_P},\rho_{\af_P}\rangle+\xi(\Omega_M).
\end{equation}
Then for $\lambda\in\af^\ast_P$ one has
\begin{equation}\label{casimir5}
\pi_{\xi,\lambda}(\Omega)=-\|\lambda\|^2+c(\xi)
\end{equation}
(see \cite[Theorem~8.22]{Kn}). Let 
\begin{equation}\label{def-F}
\cT:=\{\nu\in\Pi(\K_\infty)\colon [\nu_p(\tau)\colon\nu]\neq 0\}.
\end{equation}
Using \eqref{equ-kernel2}, \eqref{iso-ind} and \eqref{TrFT}, it follows that 
\eqref{specside3} is equal to
\begin{equation}\label{specside4}
  \begin{split}
    \sum_{\pi\in\Pi_{\di}(M(\A))}\sum_{\nu\in\cT}m(\pi)&
  e^{-t(\tau(\Omega)-c(\pi_\infty))}\\
&\cdot\int_{i(\af^G_{L_s})^*}e^{-t\|\lambda\|^2}\Tr\left(
\widetilde\Delta_{\dtup_{L_s}(\bss)}(P,\pi,\nu,\lambda)
\widetilde M(P,\pi,\nu,s)\right)\;d\lambda.
\end{split}
\end{equation}
Since $M(P,\pi,s)$ is unitary, \eqref{specside4} can be estimated by
\begin{equation}\label{specside5}
\begin{split}
  \sum_{\pi\in\Pi_{\di}(M(\A))}\sum_{\nu\in\cT}&m(\pi)
  \dim\left(\cA^2_\pi(P)^{K_f,\nu}\right)\\
&\cdot e^{-t(\tau(\Omega)-c(\pi_\infty))}
\int_{i(\af^G_{L_s})^*}e^{-t\|\lambda\|^2}\|
\widetilde\Delta_{\dtup_{L_s}(\bss)}(P,\pi,\nu,\lambda)\|\;d\lambda.
\end{split}
\end{equation}
For $\pi\in\Pi(M(\A))$ denote by $\lambda_{\pi_\infty}$ the Casimir eigenvalue
of the restriction of $\pi_\infty$ to $M(\R)^1$. Given $\lambda>0$, let 
\[
\Pi_{\di}(M(\A);\lambda):=\{\pi\in\Pi(M(\A))\colon |\lambda_{\pi_\infty}|\le
\lambda\}.
\]
Let $d=\dim M(\R)^1/\K_\infty^M$. If we use \cite[Theorem~0.1]{Mu3}
and argue in the same way as in the proof of \cite[Proposition~3.5]{Mu1} 
it follows that for every 
$\nu\in\Pi(\K_\infty)$ there exists $C>0$ such that
\begin{equation}\label{estim10}
\sum_{\pi\in\Pi_{\di}(M(\A);\lambda)}m(\pi)\dim\cA^2_\pi(P)^{K_f,\nu}\le C(1+\lambda^{2d})
\end{equation}
for all $\lambda\ge 0$. 
Next we estimate the integral in \eqref{specside5}. We use the notation of
section \ref{sec-trform}.
Let $\bss=(\beta_1^\vee,\dots,\beta_m^\vee)$ and $\dtup_{L_s}(\bss)=
(Q_1,\dots,Q_m)\in\Xi_{L_s}(\bss)$ with with $Q_i=\langle P_i,P_i'\rangle$, 
$P_i|^{\beta_i}P_i'$, $i=1,\dots,m$.
 Using the definition \eqref{intertw3} of 
$\Delta_{\dtup_{L_s}(\bss)}(P,\pi,\nu,\lambda)$, it follows that we can bound the
integral by a constant multiple of
\begin{equation}\label{est-integral}
\dim(\nu)\int_{i(\af^G_{L_s})^*}e^{-t\|\lambda\|^2}\prod_{i=1}^m\left\| 
\delta_{P_i|P_i^\prime}(\lambda)\Big|_{\cA^2_\pi(P_i^\prime)^{K_f,\nu}}\right\|
\;d\lambda,
\end{equation}
where $\delta_{P_i|P_i^\prime}(\lambda)$ is defined by \eqref{intertw2}.
We introduce new coordinates $s_i:=\langle\lambda,\beta_i^\vee\rangle$,
$i=1,\dots,m$, on $(\af^G_{L_s,\C})^\ast$. Using \eqref{normalization} and
\eqref{intertw2}, we
can write
\begin{equation}\label{delta}
\delta_{P_i|P_i^\prime}(\lambda)=\frac{n^\prime_{\beta_i}(\pi,s_i)}{n_{\beta_i}(\pi,s_i)}
+j_{P_i^\prime}\circ(\Id\otimes R_{P_i|P_i^\prime}(\pi,s_i)^{-1}
R^\prime_{P_i|P_i^\prime}(\pi,s_i))\circ j_{P_i^\prime}^{-1}.
\end{equation}
Put
\[
\cA^2_\pi(P)^{K_f,\cT}=\bigoplus_{\nu\in\cT}\cA^2_\pi(P)^{K_f,\nu},
\]
where $\cT$ is defined by \eqref{def-F}. 
It follows from \cite[Theorem~5.3]{Mu2} that there
exist $N,k\in\N$ and $C>0$ such that 
\begin{equation}\label{log-deriv}
\int_{i\R}\left|\frac{n^\prime_{\beta_i}(\pi,s)}{n_{\beta_i}(\pi,s)}\right|
(1+|s|^2)^{-k}\;ds\le C(1+\lambda_{\pi_\infty}^2)^N,\;i=1,\dots,m,
\end{equation}
for all $\pi\in \Pi_{\di}(M(\A))$ with $\cA^2_\pi(P)^{K_f,\cT}\neq 0$. 
Combining \eqref{delta}, \eqref{log-deriv} and Proposition~\ref{log-der},
it follows that  we have
\[
\int_{i(\af^G_{L_s})^*}e^{-t\|\lambda\|^2}\prod_{i=1}^m\left\| 
\delta_{P_i|P_i^\prime}(\lambda)\Big|_{\cA^2_\pi(P_i^\prime)^{K_f,\nu}}\right\|
\;d\lambda\ll (1+\lambda_{\pi_\infty}^2)^{mN}
\]
for all $t\ge 1$, and $\pi\in \Pi_{\di}(M(\A))$ with $\cA^2_\pi(P)^{K_f,\cT}\neq 0$.
Thus \eqref{specside5} can  estimated by a constant multiple of
\begin{equation}\label{est-specside1}
\sum_{\pi\in\Pi_{\di}(M(\A))}\sum_{\nu\in\cT}m(\pi)\dim\left(\cA^2_\pi(P)^{K_f,\nu}\right)
(1+\lambda_{\pi_\infty}^2)^N e^{-t(\tau(\Omega)-c(\pi_\infty))}.
\end{equation}
Now we can proceed as in \cite{MM1}. For the convenience of the reader we 
recall the arguments. By \eqref{casimir4} we have
\begin{equation}\label{eigenv}
\tau(\Omega)-c(\pi_\infty)=\tau(\Omega)+\|\rho_\af\|^2-\lambda_{\pi_\infty}.
\end{equation}
Together with \cite[Lemma~13.2]{MM1}, it follows that there 
exists $\lambda_0> 0$ such that
\[
\tau(\Omega)-c(\pi_\infty)\ge |\lambda_{\pi_\infty}|/2
\]
for all $\pi\in \Pi_{\di}(M(\A))$ with $\cA^2_\pi(P)^{K_f,\cT}\neq 0$ and
$|\lambda_{\pi_\infty}|\ge \lambda_0$. We decompose the sum over $\pi$ in 
\eqref{est-specside1} in two summands $\Sigma_1(t)$ and $\Sigma_2(t)$,
where in $\Sigma_1(t)$ the summation runs over all $\pi$ with 
$|\lambda_{\pi_\infty}|\le \lambda_0$. Using \eqref{estim10}, it follows that
for $t\ge 1$
\[
\Sigma_2(t)\ll e^{-t|\lambda_0|/2}.
\]
Since $\Sigma_1(t)$ is a finite sum, both in $\pi$ and $\nu$, it follows from
\cite[Lemma~13.1]{MM1} that there exists $c>0$ such that
\[
\Sigma_1(t)\ll e^{-ct}
\]
for $t\ge 1$. Recall that \cite[Lemma~13.1]{MM1} requires that $P=MAN$ is a
proper parabolic subgroup of $G$. Putting everything together we obtain the 
following result.
\begin{lem}\label{lem-trspec}
Let $\tau\in \Rep(\bG(\R))$. Assume that $\tau\not\cong\tau_\theta$. Let $M$ be a
proper Levi subgroup of $\bG$. There exists $c>0$ such that
\[
J_{\spec,M}(\phi_t^{\tau,p})=O(e^{-ct})
\]
for $t\ge 1$.
\end{lem}
It remains to deal with the case $M=\bG$. This has been done already in 
\cite{MM1} for an arbitrary reductive group $\bG$. Using \cite[(13.34)]{MM1}
and the considerations following this equality, we get
\begin{equation}
J_{\spec,G}(\phi^{\tau.p}_t)=O(e^{-ct})
\end{equation}
for some $c>0$. 
Combined with Lemma~\ref{lem-trspec} we obtain 
\begin{prop}\label{asympinf}
There exists $c>0$ such that 
\[
J_{\spec}(\phi_t^{\tau,p})=O\left(e^{-ct}\right)
\]
for all $t\ge 1$ and $p=0,\dots,n$.
\end{prop}
Applying the trace formula \eqref{tracef1}, we get
\[
\Tr_{\reg}\left(e^{-t\Delta_p(\tau)}\right)=O(e^{-ct}),\quad\text{as}\;
 t\to\infty,
\]
which is the proof of Theorem~\ref{prop-lt}. The asymptotic expansion as
$t\to +0$ is provided by Lemma \ref{lem:asymp-exp}.

Thus the corresponding zeta function $\zeta_p(s;\tau)$, defined by the Mellin 
transform 
\begin{equation}\label{zetafct}
\zeta_p(s;\tau):=\frac{1}{\Gamma(s)}\int_0^\infty 
\Tr_{\reg}\left(e^{-t\Delta_p(\tau)}\right) t^{s-1}\; dt.
\end{equation}
is holomorphic in the half-plane $\Re(s)>d/2$ and admits a meromorphic
extension to the whole complex plane. It may have a pole at $s=0$.  Let $f(s)$ 
be a meromorphic function on $\C$. For $s_0\in\C$ let 
\[
f(s)=\sum_{k\ge k_0}a_k(s-s_0)^k
\]
be the Laurent expansion of $f$ at $s_0$. Put $\FP_{s=s_0}:=a_0$. Now we define
the analytic torsion $T_X(\tau)\in\R^+$ by
\begin{equation}\label{analtor3}
\log T_X(\tau)=\frac{1}{2}\sum_{p=0}^d (-1)^p p 
\left(\FP_{s=0}\frac{\zeta_p(s;\tau)}{s}\right).
\end{equation}
\begin{remark}
In the case of $\bG=\GL(3)$ we have determined the coefficients of the 
log-terms in \cite{MM1}. The calculation shows that the 
zeta functions definitely have a pole at $s=0$. However, the combination
$\sum_{p=1}^5 (-1)^p p \zeta_p(s;\tau)$ turns out to be holomorphic at $s=0$ and
we can instead define the logarithm of the analytic torsion by
\[
\log T_{X(K_f)}(\tau)=\frac{d}{ds}
\left(\frac{1}{2}\sum_{p=1}^5 (-1)^p p \zeta_p(s;\tau)\right)\bigg|_{s=0}.
\]
\end{remark}
Put
\begin{equation}\label{altheat}
K(t,\tau):=\sum_{p=1}^d (-1)^p p \Tr_{\reg}\left(e^{-t\Delta_p(\tau)}\right).
\end{equation}
Then $K(t,\tau)=O(e^{-ct})$ as $t\to\infty$ and the Mellin transform
\[
\int_0^\infty K(t,\tau)t^{s-1} dt
\]
converges absolutely and uniformly on compact subsets of $\Re(s)>d/2$ and 
admits a meromorphic extension to $\C$. Moreover, by \eqref{analtor3} we
have
\begin{equation}\label{analtor1}
\log T_X(\tau)=\FP_{s=0}\left(\frac{1}{\Gamma(s)}\int_0^\infty
K(t,\tau)t^{s-1} dt\right).
\end{equation}
Let $\delta(\widetilde X):=\rk_\C(G(\R)^1)-\rk_\C({\bf K}_\infty)$. Let 
$\Gamma\subset G(\R)^1$ be a torsion free, co-compact lattice. Let $X=
\Gamma\bs \widetilde X$. If $\dim X$ is even or $\delta(\widetilde X)\ge 2$,
 then $T_X(\tau)=1$ for every $\tau\in\Rep(G(\R)^1)$ \cite[Prop. 4.2]{MP2}.
We note that the proof of \cite[Prop. 4.2]{MP2} for the case 
$\delta(\widetilde X)\ge 2$ contains a mistake, which is the claim that the
Grothendieck group of admissible representations is generated by induced
representations. This is only true for $\GL(n,\R)$.
However, the proof can be fixed by proceeding as in the proof
of \cite[Corollary 2.2]{MoS}. 
As the example in \cite{MP1} shows,  in the non-compact case the analytic
torsion need not be trivial if $\dim X$ is even. Hence one may guess that
this is also true if $\delta(\widetilde X)\ge 2$.
We consider here the contribution of the discrete spectrum to the analytic 
torsion and study when it vanishes. Let
\[
\phi_t^\tau:=\sum_{p=1}^d (-1)^p p \phi_t^{\tau,p} \quad\text{and}\quad 
k_t^\tau:=\sum_{p=1}^d (-1)^p p h_t^{\tau,p}.
\]
Then by \eqref{regtr-heat} we have
\begin{equation}\label{alt-reg-tr}
K(t,\tau)=J_{\spec}(\phi_t^\tau).
\end{equation}
For $\pi\in\Pi(G(\R))$ let $\Theta_\pi$ be the global character. Then the
contribution of the discrete spectrum is given by
\begin{equation}\label{alt-reg-tr1}
J_{\spec,G}(\phi^\tau_t)=\sum_{\pi\in\Pi_{\di}(G(\A))} m(\pi)\dim\left(\cH_{\pi_f}^{K_f}
\right)\Theta_{\pi_\infty}(k_t^\tau).
\end{equation}
If $\dim X$ is even, we can follow the proof of \cite[Prop. 4.2]{MP2} to
show that $J_{\spec,G}(\phi^\tau_t)=0$.  
We conjecture that $J_{\spec,G}(\phi^\tau_t)=0$, if $\delta(\widetilde X)\ge 2$. 
In \cite[Lemma 13.5]{MM1} it was shown that this holds for $G=\GL(n)$, i.e.,
$J_{\spec,\GL(n)}(\phi^\tau_t)=0$, if $n\ge 5$.

\section{Approximation of $L^2$-analytic torsion}\label{sec-towers}
\setcounter{equation}{0}

In this section we prove the Theorem \ref{theo-approx}. To begin with we
introduce some notation.

Let $\bG$ be a reductive algebraic group over $\Q$. Throughout this section
we assume that 
$\bG$ satisfies two conditons, called (TWN) and (BD), which were introduced in  
\cite[Definition 5.2, Definition 5.9]{FiLaMu}. These conditions imply
appropriate estimations for the logarithmic derivatives of the intertwining
operators, which occur on the spectral side of the trace formula. Property
(TWN) is a global condition concerning the scalar-valued normalizing factors of 
the intertwining operators, while (BD) is a condition for the local 
intertwining operators. For (BD) we need additional assumptions introduced in
\cite[Definition 2]{FL4}. Let $S$ be a set of primes. We say that $G$ satisfies
property (BD) for $S$, if the local groups $G(\Q_p)$, $p\in S$, satisfy (BD)
with a uniform value of $C$ (see \cite{FL4} for details).

For families of open compact subgroups of $\bG(\A_f)$ we will need a certain
non-degeneracy condition introduced in \cite[Definition 1.3]{FL3} which we 
recall next.
For any reductive groups $H$ let $H(\A)^+$ be the 
image of the map $H^{\sico}\to H(\A)$, where $H^{\sico}$ is the simply connected
cover of the derived group of $H$. Define $H(\A_f)^+$ analogously. 
A family ${\mathcal K}$ of compact open subgroups of $\bG(\A_f)$ is called
{\it non-degenerate}, if for any $\Q$-simple normal subgroup $H$ of $\bG$ we 
have $\vol_{H(\A_f)^+}(K\cap H(\A_f)^+)\to 0$, $K\in{\mathcal K}$. 

For an open compact subgroup $K_M\subset M(\A_f)$ let 
$\mu^M_{K_M}$ be the measue on $\Pi(M(\R)^1)$ defined by
\begin{equation}\label{measure2}
\begin{split}
\mu^M_{K_M}=&\frac{\vol(K_{M})}{\vol(M(\Q)\bs M(\A)^1)}\\
&\hskip20pt\cdot\sum_{\pi\in\Pi(M(\A)^1)}\dim\Hom_{M(\A)^1}(\pi,L^2(M(\Q)\bs M(\A)^1))
\dim \pi_f^{K_{M}}\delta_{\pi_\infty}.
\end{split}
\end{equation}
In the notation of \cite{FiLaMu} this is the measure $\mu^{M,S_\infty}_{K_M}$,
where $S_\infty=\{\infty\}$. We will consider collection of measures on
$\Pi(M(\R)^1)$ which are {\it polynomially bounded}. For the definition we
refer to \cite[Sect. 4.1]{FL3}.

Let $K\subset \bG(\A_f)$ be a compact open subgroup and let $X(K)$ be the
adelic quotient \eqref{adelic-quot}. Let $\tau\in\Rep(G(\R)^1)$ be irreducible
and assume that $\tau$ satisfies $\tau\not\cong\tau\circ\theta$, where $\theta$
is the Cartan involution. Let $E_{\tau,K}\to X(K)$ be the flat vector bundle
over $X(K)$ associated to $\tau$. Let $\Delta_{p,K}(\tau)$ be the Laplace
operator acting in the space of $E_{\tau;K}$-valued $p$-forms on $X(K)$. Let
$T_{X(K)}(\tau)$ be the regularized analytic torsion. 

Recall the definition of $\cT$ from \ref{def-F}. Let $\Co(G(\A)^1)_\cT$ be the subspace of $\Co(G(\A)^1)$ consisting of all
$f\in\Co(G(\A)^1)$ whose $\bK_\infty\times\bK_\infty$-types are contained in
$\cT\times\cT$.

To establish Theorem \ref{theo-approx}, we follow the approach of \cite{MM2}.
Recall that the definition of  $T_{X(K)}(\tau)$ by \eqref{analtor1} and
\eqref{altheat}. This reduces the problem to the consideration of the Mellin
transform of the regularized trace of the heat operators. We write
\begin{equation}\label{splitt-at}
\begin{split}
\int_0^\infty \Tr_{\reg}\left(e^{-t\Delta_{p,K}(\tau)}\right)t^{s-1}dt 
&=\int_0^T \Tr_{\reg}\left(e^{-t\Delta_{p,K}(\tau)}\right)t^{s-1} dt\\
&+\int_T^\infty \Tr_{\reg}\left(e^{-t\Delta_{p,K}(\tau)}\right)t^{s-1}dt.
\end{split}
\end{equation}
To begin we estimate the second integral on the right hand side.
The first auxiliary result is the following proposition.
\begin{prop}\label{prop-lt1}
Let $\tau\in\Rep(G(\R)^1)$ be irreducible.
Assume that $\tau\not\cong\tau_\theta$. Let ${\bf K}_0\subset\bG(\A_f)$ be a 
compact open subgroup. Then there exist $C,c>0$ such that
\begin{equation}
\frac{1}{\vol(X(K))}\big|\Tr_{\reg}\left(e^{-t\Delta_{p,K}(\tau)}\right)\big|\le
C e^{-ct}
\end{equation}
for all $t\ge 1$, $p=0,...,d$, and open subgroups $K\subset {\bf K}_0$.
\end{prop}
\begin{proof}
Let $h_t^{\tau,p}$ be defined by \eqref{tr-kern}. Then by 
\eqref{regtrace1a}, the definition of $\phi_t^{\tau,p}$ by \eqref{extension5}
and Theorem \ref{thm-specexpand} we have
\begin{equation}\label{reg-trace-heat}
\Tr_{\reg}\left(e^{-t\Delta_{p,K}(\tau)}\right)=\sum_{[M]}
J_{\spec,M}(h_t^{\tau,p}\otimes\chi_K),
\end{equation}
where $[M]$ runs over the conjugacy classes of Levi subgroups of $\bG$. 
Let $M$ be a proper standard Levi 
subgroup, $P\in\cP(M)$, 
and let $\level(K,G_M^+)$ be defined by \cite[(3.1)]{FL3}. Let $l\in\N$ and $\varepsilon>0$.
By \cite[Lemma 6.3]{MM2} there exist $C,c>0$ such that
\begin{equation}\label{est-spec-M}
|J_{\spec,M}(h_t^{\tau,p}\otimes\chi_K)|\le C e^{-ct} \sum_{\pi\in\Pi_{\di}(M(\A))}
\sum_{\nu\in\cT}\dim\left(\cA^2_\pi(P)^{K,\nu}\right)(1+|\lambda_{\pi_\infty}|)^{-l}
\level(K,G_M^+)^\varepsilon
\end{equation}
for all $t\ge 1$, $p=0,...,d$, and open compact subgroups $K\subset \bG(\A_f)$. 
To estimate the double series we proceed as in the proof of Propositon 4.4
in \cite{FL3}. First we need to introduce some additional notation.
Let  $\cT_M\subset \Pi(\K_{M,\infty})$ be the finite set of all 
irreducible components of the restriction of elements of $\cT$ to 
$\bK_{M,\infty}$. By
Frobenius reciprocity only those $\pi\in\Pi_{\di}(M(\A))$ such that $\pi_\infty$
contains a $\bK_{M,\infty}$-type in $\cT_M$ can contribute to the right-hand
side of \eqref{est-spec-M}. Denote by $\Pi_{\di}(M(\A))^{\cT_M}$ the corresponding
subset of $\Pi_{\di}(M(\A))$. 
Now consider the dimension of the spaces of automorphic forms
appearing on the right hand side of \eqref{est-spec-M}. We have
\[
\begin{split}
\dim\cA^2_\pi(P)^{K,\nu}&=m_\pi\dim\Ind_{P(\A)}^{G(\A)}(\pi)^{K,\nu}\\
&=m_\pi\dim\left(\Ind_{P(\R)}^{\bG(\R)}(\pi_\infty)\right)^\nu 
\dim\left(\Ind_{P(\A_f)}^{\bG(\A_f)}(\pi_f)\right)^K,
\end{split}
\]
where
\[
m_\pi=\dim\Hom(\pi, L^2_{\di}(A_M M(\Q)\bs M(\A))).
\]
Now observe that $\dim\left(\Ind_{P(\R)}^{\bG(\R)}(\pi_\infty)\right)^\nu \le
(\dim \nu)^2$. Denote by $\proj_M\colon P(\A_f)\to M(\A_f)$ the canonical 
projection. For 
$\gamma\in\bG(\A_f)$ and an open compact subgroup $K\subset\bG(\A_f)$ let
\[
K_M^\gamma:=\proj_M(P(\A_f)\cap\gamma\ K\gamma^{-1}).
\]
Then we have
\[
 \dim\left(\Ind_{P(\A_f)}^{\bG(\A_f)}(\pi_f)\right)^K
=\sum_{\gamma\in P(\A_f)\bs G(\A_f)/K} \dim \pi_f^{K_M^\gamma}.
\]
Summarizing we get
\begin{equation}\label{est-spec-M1}
\begin{split}
J_{\spec,M}(h_t^{\tau,p}\otimes\chi_K)\ll_{\cT,\varepsilon} &e^{-ct}
\level(K,G_M^+)^\varepsilon\\
&\cdot\sum_{\gamma\in P(\A_f)\bs G(\A_f)/K}\sum_{\pi\in\Pi_{\di}(M(\A))^{\cT_M}}
(1+|\lambda_{\pi_\infty}|)^{-l} m_\pi\dim\pi_f^{K_M^\gamma}.
\end{split}
\end{equation}
According to the proof of \cite[Corollary 4.9]{FL3}
the set of measures 
$\{\mu_{K^\gamma_M}^M\}$, where $K$ ranges
over the open subgroups of ${\bf K}_0$ and $\gamma\in\bG(\A_f)$, is 
polynomially bounded. Using \cite[(4.1)]{FL3}, it follows there exists 
$N\in\N$, depending only on $\cT_M$, such that
\begin{equation}\label{est-spec-M2}
\vol_M(K^\gamma_M)\sum_{\pi\in\Pi_{\di}(M(\A))^{\cT_M}} (1+|\lambda_{\pi_\infty}|)^{-N}
m_\pi \dim \pi_f^{K_M^\gamma}\ll_{\cT_M}1
\end{equation}
for all $K\subset \bK_0$ and $\gamma\in\bG(\A_f)$. For any $f\in C_c(\bG(\A_f))$
define ${\mathcal O}{\mathcal I}_P(f)$ by
\begin{equation}\label{def-const-P}
{\mathcal O}{\mathcal I}_P(f):=\int_{P(\A_f)\bs\bG(\A_f)}\int_{N(\A_f)}
f(g^{-1}ug)du dg=\int_{\bK_f}\int_{N(\A_f)} f(k^{-1}uk)du dk.
\end{equation}
Put
\begin{equation}\label{def-const-P1}
{\mathcal O}{\mathcal I}_{P,K}={\mathcal O}{\mathcal I}_P(\1_K).
\end{equation}
Using \cite[Lemma 4.3]{FL3}, it follows from \eqref{est-spec-M2} that
\begin{equation}\label{est-spec-M3}
\vol(K)\sum_{\gamma\in P(\A_f)\bs\bG(\A_f)/K}\sum_{\pi\in\Pi_{\di}(M(\A))^{\cT_M}} 
(1+|\lambda_{\pi_\infty}|)^{-N} m_\pi \dim \pi_f^{K_M^\gamma}
\ll_{\cT_M}{\mathcal O}{\mathcal I}_{P,K}.
\end{equation}
Now observe that 
\[
\vol(X(K))= \frac{\vol(G(\Q)\bs\G(\A)^1)}{\vol(K)}.
\]
Note that $\varepsilon>0$ in \eqref{est-spec-M1} can be chosen arbitrarily
small. Then
by \cite[Lemma 4.6]{FL3} and \eqref{est-spec-M1} it follows that for every
proper standard Levi subgroup $M$ there exist $C,c,\delta>0$ such that
\begin{equation}\label{est-spec-M4}
\frac{1}{\vol(X(K))} |J_{\spec,M}(h_t^{\tau,p}\otimes\chi_K)|\le C e^{-ct}
\level(K,G_M^+)^{-\delta}
\end{equation}
for all open compact subgroups $K\subset \bK_0$, $t\ge 1$, and $p=0,...,d$.
The remaining case is $M=\bG$. Then 
\begin{equation}
\begin{split}
J_{\spec,\bG}(h_t^{\tau,p}\otimes\chi_K)&=\sum_{\pi\in\Pi_{\di}(G(\A)^1)} m_\pi
\Tr\pi(h_t^{\tau,p}\otimes\chi_K)\\
&=\sum_{\pi\in\Pi_{\di}(G(\A)^1)} m_\pi 
\dim(\pi_f^K)\Tr\pi_\infty(h_t^{\tau,p}).
\end{split}
\end{equation}
Now observe that by \cite[(4.18), (4.19)]{MP2} we have
\[
\Tr\pi_\infty(h_t^{\tau,p})=e^{t(\pi_\infty(\Omega)-\tau(\Omega))}
\dim(\cH_{\pi_\infty}\otimes\Lambda^p\pf^\ast\otimes V_\tau)^{K_\infty},
\]
where $\Omega$ is the Casimir operator of $\bG(\R)^1)$. Furthermore, for
$\nu\in\Pi(K_\infty)$ we have $[\pi_\infty|_{K_\infty}\colon\nu]\le \dim(\nu)$
\cite[Theorem 8.1]{Kn}. Thus there exists $C>0$ such that
\[
\frac{1}{\vol(X(K))}|J_{\spec,\bG}(h_t^{\tau,p}\otimes\chi_K)|\le C\vol(K)
\sum_{\pi\in\Pi_{\di}(G(\A)^1)^{\cT}} m_\pi \dim(\pi_f^K)
e^{t(\pi_\infty(\Omega)-\tau(\Omega))}
\]
for all $t>0$ and open subgroups $K\subset \bK_0$. Note that in our
notation  $\lambda_{\pi_\infty}=\pi_\infty(\Omega)$. Arguing as in the proof
of Lemma 6.3 in \cite{MM2}, it follows that there exists $c>0$ and for all
$l\in\N$ there exists $C_l>0$ such that
\[
\begin{split}
\frac{1}{\vol(X(K))}|J_{\spec,\bG}(h_t^{\tau,p}\otimes\chi_K)|\le &C_l \vol(K) 
e^{-ct}\\
&\times \sum_{\pi\in\Pi_{\di}(G(\A)^1)} m_\pi(1+|\lambda_{\pi_\infty}|)^{-l}
\dim(\pi_f^K)
\end{split}
\]
for all $t\ge 1$ and open subgroups $K\subset\bK_0$. Now observe that 
\eqref{est-spec-M2} holds also for $M=\bG$, which implies that there exist
$C,c>0$ such that
\begin{equation}\label{est-spec-G}
\frac{1}{\vol(X(K))}|J_{\spec,\bG}(h_t^{\tau,p}\otimes\chi_K)|\le C e^{-ct}
\end{equation}
for all $t\ge 1$, $p=0,...,d$, and open subgroups $K\subset\bK_0$.
Combining \eqref{reg-trace-heat}, \eqref{est-spec-M4} and \eqref{est-spec-G},
the proof is completed.
\end{proof} 

Using Proposition \ref{prop-lt1} it follows that there exist $C,c>0$ such
that
\begin{equation}\label{est-larget3}
\frac{1}{\vol(X(K))}\left|\int_T^\infty 
\Tr_{\reg}\left(e^{-t\Delta_{p,K}(\tau)}\right)t^{-1}dt\right|\le C e^{-cT}
\end{equation}
for all $T\ge 1$, $p=0,\dots,d$, and open subgroups $K\subset\bK_0$.

Now we turn to the estimation of the first integral on the right hand side of
\eqref{splitt-at}. In order to study the short time behavior of the 
regularized trace of the heat operator with the help of the trace formula,
we first need to show that we can replace $h_t^{\tau,p}$ by an appropriate 
compactly
supported test function without changing the asymptotic behavior as $t\to 0$.
Let $r(g)$ be the function on $G_\infty$, which is defined by \eqref{dist-fct}. 
For $R>0$ let 
\[
B_R:=\{ g\in G_\infty\colon r(g)< R\}.
\]
We need the following auxiliary lemma.
\begin{lem}\label{lem-est5}
There exist $C,c>0$ such that
\[
\int_{G_\infty} e^{-r^2(g)/t}dg\le C e^{ct}
\]
for all $t>0$. 
\end{lem}
\begin{proof}
For $G=\GL_n$ this was proved in \cite[Lemma 7.1]{MM2}. Let $d_n(\cdot,\cdot)$
denote the geodesic distance function on $GL_n(\R)^1/\rO(n)$, and let $r_n(g)=
d_n(g\rO(n),\rO(n))$. At the beginning of section \ref{sec:test:fcts} we have 
shown
that with respect to our choice of the embedding $G_\infty\subset \GL_n(\R)^1$,
$r$ and $r_n$ coincide on $G_\infty$. Thus the lemma follows from the case of
$\GL_n$.
\end{proof} 
Let 
$f\in C^\infty(\R)$ such that $f(u)=1$, if $u\le 1/2$, and $f(u)=0$, if
$u\ge 1$. Let $\varphi_R\in C^\infty_c(G_\infty)$ be defined by
\begin{equation}\label{phiR}
\varphi_R(g):= f\left(\frac{r(g)}{R}\right).
\end{equation}
Then we have $\supp\varphi_R\subset B_R$. Extend $\varphi_R$ to $G(\R)$ by
\[
\varphi_R(g_\infty z)=\varphi_R(g_\infty),\quad g_\infty\in G_\infty,\; z\in 
A_G(\R)^0.
\]
Define $\widetilde h^{\tau,p}_{t,R}\in C^\infty(G(\R))$ by
\begin{equation}\label{test-funct1}
\widetilde h^{\tau,p}_{t,R}(g_\infty):= \varphi_R(g_\infty) h^{\tau,p}_t(g_\infty),
\quad g_\infty\in G(\R).
\end{equation}
Then the restriction of 
$\widetilde h_{t,R}^{\tau,p}\otimes\chi_{K}$ to $G(\A)^1$ belongs to 
$C^\infty_c(G(\A)^1)$. 
\begin{prop}\label{prop-cptsupp}
There exist constants $C_1,C_2,C_3>0$ such that
\[
\frac{1}{\vol(X(K))}\big|J_{\spec}(h^{\tau,p}_t\otimes\chi_{K})- 
J_{\spec}(\widetilde h^{\tau,p}_{t,R}\otimes\chi_{K})\big|\le C_1 e^{-C_2R^2/t+C_3t}
\]
for all open subgroups $K\subset\bK_0$, $p=0,\dots,d$, $t>0$ and $R\ge 1$.
\end{prop}
Proposition \ref{prop-cptsupp} allows us to replace $h^{\tau,p}_t$ by a compactly
supported function. 
\begin{proof}
Let
$\psi_R:=1-\varphi_R$.  Then 
\[
J_{\spec}(h^{\tau,p}_t\otimes \chi_{K})- 
J_{\spec}(\widetilde h^{\tau,p}_{t,R}\otimes\chi_{K})=
J_{\spec}(\psi_R h^{\tau,p}_t\otimes\chi_{K}).
\]
Now we use the refined spectral expansion \eqref{specside2}. Let $M\in\cL$.
and let $J_{\spec,M}$ be the distribution on the right hand side of 
\eqref{specside2}, which corresponds to $M$. Assume that $M\neq\bG$. 
We note that
$\psi_R h_t^{\tau,p}\otimes\chi_{K}$ belongs to $\Co(G(\A)^1)_\cT$ and 
\cite[Prop. 4.4]{FL3} extends to $\Co(G(\A)^1)_\cT$. 
We recall that by
\cite[Proposition]{FL3} the set of measures $\{\mu^M_{K_M^\gamma}\}$, where $K$
ranges over open subgroups of $\bK_0$ and $\gamma\in\bG(\A_f)$, is polynomial
bounded. Using \cite[Proposition 4.4]{FL3} and \cite[Lemma 4.6]{FL3} it follows
that there exist $k\in\N$ and $\delta>0$ such that
\begin{equation}\label{estim-JM}
J_{\spec,M}(\psi_R h_t^{\tau,p}\otimes\chi_K)\ll_\cT \vol(K)\|\psi_R h_t^{\tau,p}\|_k
\level(K;G_M^+)^{-\delta}
\end{equation}
for all $t>0$, $p=0,...,d$, and open subgroups $K\subset\bK_0$. The norm
$\|\cdot\|_k$ is defined in \cite[Sect. 3.1]{FL3}. Let $\nabla$ be the
canonical connection on $\bG(\R)^1$. There exists $C>0$ such that
the norm
\[
\|h\|_k\le C\sum_{l=0}^k\|\nabla^l h\|_{L^1(G(\R)^1)}, \quad h\in \Co(G(\A)^1).
\]
Using Lemma \ref{lem-est5} we can proceed exactly as in the
proof of Proposition 7.2 in \cite{MM2} and estimate 
$\|\psi_R h^{\tau,p}_t)\|_k$. Combined with 
\eqref{estim-JM} it follows that there exist $C_1,C_2,C_3>0$ such that
\begin{equation}\label{estim-JM1}
\frac{1}{\vol(X(K))}|J_{\spec,M}(\psi_R h^{\tau,p}_t\otimes\chi_{K})|\le 
C_1 e^{-C_2R^2/t+C_3t}\level(K;G_M^+)^{-\delta}
\end{equation}
for all open subgroups $K\subset \bK_0$, $p=0,\dots,d$, $t>0$. and $R\ge 1$. 

It remains to deal with the case $M=G$. In this case we have 
\[
\begin{split}
J_{\spec,G}(\psi_R h_t^{\tau,p}\otimes\chi_{K})&=\sum_{\pi\in\Pi_{\di}(G(\A)^1)} m_\pi
\Tr\pi(\psi_R h_t^{\tau,p}\otimes\chi_{K})\\
&=\sum_{\pi\in\Pi_{\di}(G(\A)^1)} m_\pi\dim(\pi_f^{K})
\Tr\pi_\infty(\psi_R h_t^{\tau,p}).
\end{split}
\]
If we argue as in the proof of Propostion 7.2 in \cite[pp. 335]{MM2}, we get
\[
\begin{split}
\frac{1}{\vol(X(K))}&|J_{\spec,G}(\psi_R h_t^{\tau.p}\otimes\chi_{K})|\\
&\le
C_k e^{-C_2R^2/t+C_3t}\vol(K)\sum_{\pi\in\Pi_{\di}(G(\A)^1)^\cT} m_\pi\dim(\pi_f^{K})
(1+|\lambda_{\pi_\infty}|)^{-k}.
\end{split}
\]
By \cite[Proposition 4.7]{FL3}, the collection of measures 
$\{\mu_{K}^G\}_{N\in\N}$, where $K$ ranges over open subgroups of $\bK_0$,
 is polynomially bounded hence we have \eqref{est-spec-M2}, which also 
holds for $M=\bG$. Hence there exist $C_1,C_2,C_3>0$ such 
that
\[
\frac{1}{\vol(X(K))}|J_{\spec,G}(\psi_R h^{\tau,p}_{t}\otimes\chi_{K})|\le 
C_1 e^{-C_2R^2/t+C_3t}
\]
for all open subgroups $K\subset\bK_0$, $t>0$, $p=0,\dots,d$, and $R\ge 1$. 
This completes the proof of the proposition.
\end{proof}
Now recall that
\[
\Tr_{\reg}\left(e^{-t\Delta_{p,K}(\tau)}\right)=J_{\spec}(h_t^{\tau,p}\otimes\chi_{K}).
\]
For $R>0$ let $\varphi_R\in C^\infty_c(G(\R)^1)$ be the function defined by
\eqref{phiR}.
By
Proposition \ref{prop-cptsupp} we have
\begin{equation}\label{cptsupp2}
\Tr_{\reg}\left(e^{-t\Delta_{p,K}(\tau)}\right)=
J_{\spec}(\varphi_R h_t^{\tau,p}\otimes\chi_{K})+r_R(t),
\end{equation}
where $r_R(t)$ is a function of $t\in [0,T]$ which satisfies
\begin{equation}\label{r-estim}
\frac{1}{\vol(X(K))}|r_R(t)|\le C_1 e^{-C_2R^2/t+C_3t}
\end{equation}
for $0\le t\le T$. This implies that $\int_0^T r_R(t) t^{s-1}dt$ is holomorphic 
in $s\in\C$ and 
\[
FP_{s=0}\left(\frac{1}{s\Gamma(s)} \int_0^Tr_R(t) t^{s-1}dt\right)= 
\int_0^Tr_R(t) t^{-1}dt.
\]
Moreover
\begin{equation}\label{est-remain}
\begin{split}
\frac{1}{\vol(X(K))}\left|\int_0^T r_R(t)t^{-1}dt\right|&\le 
C_1\int_0^T e^{-C_2R^2/t+C_3t} t^{-1} dt\\
&\le C_1 e^{-C_4R^2/T+C_3T} \int_0^{T/R^2} e^{-C_4/t}t^{-1} dt.
\end{split}
\end{equation}
Now put $R=T^2$ and let 
\begin{equation}\label{cpt-supp4}
h^{\tau,p}_{t,T}:=\varphi_{T^2} h_t^{\tau,p}.
\end{equation}
Then it follows from \eqref{cptsupp2} and \eqref{est-remain}
that there exist $C,c>0$ such that
\begin{equation}\label{cpt-supp3}
\begin{split}
\frac{1}{\vol(X(K))}\biggl|\FP_{s=0}&\left(\frac{1}{s\Gamma(s)}\int_0^T
\Tr_{\reg}\left(e^{-t\Delta_{p,K}(\tau)}\right)t^{s-1}dt\right)\\
&-\FP_{s=0}\left(\frac{1}{s\Gamma(s)}
\int_0^TJ_{\spec}(h_{t,T}^{\tau,p}\otimes\chi_{K})t^{s-1}dt\right)\biggr|
\le C e^{-cT}
\end{split}
\end{equation}
for $T\ge 1$, $p=0,\dots,d$, and open subgroups $K\subset \bK_0$. 
Using the trace formula, we are reduced to deal with
\[
\FP_{s=0}\left(\frac{1}{s\Gamma(s)}
\int_0^TJ_{\geo}(h_{t,T}^{\tau,p}\otimes\chi_{K})t^{s-1}dt\right).
\]
Now we fix $T>1$. 
Let $\varphi\in C_c^\infty(G(\R)^1)$ be such that $\varphi(g)=1$ in a 
neighborhood of $1\in G(\R)^1$ and $\supp\varphi\subset B_{T^2}$. Put 
\[
\widetilde h_t^{\tau,p}=\varphi h_t^{\tau,p}.
\]
We consider test functions with $\widetilde h_t^{\tau,p}$ at the infinite place
and $\chi_{K_j}$ at the finite places. By Proposition 
\eqref{prop:replace:geom:unip} we have
\begin{equation}\label{unip4}
J_{\geo}(\widetilde h_t^{\tau,p}\otimes\chi_{K_j})
=J_{\unip}(\widetilde h_t^{\tau,p}\otimes\chi_{K_j})
\end{equation}
for all $K_j$ provided we choose the support of $\widetilde h_t^{\tau,p}$ 
sufficiently small (this choice depends on $\bK_0$ only).

In order to deal with the unipotent contribution we need to restrict the 
sequences of subgroups as described in the introduction. We fix a finite 
set $S_1$ of primes.
Let $\{K_j\}_{j\in\N}$ be sequence of open subgroups of $\bK_0$ satisfying 
$K_j\underset{S_1}{\rightarrow} 1$. Recall that this means that
 $K_j=K_{S_1,j}\times K^{S_1}$ with $K^{S_1}=\prod_{p\not\in S_1}K_p$ a fixed open 
compact subgroup of $\bG(\A^{S_1})$ and
$K_{S_1,j}=\prod_{p\in S_1}K_{p,j}$ with $K_{S_1,j}
\underset{j\to\infty}{\rightarrow} 1$. Let $S_0\subset S_{\fin}\setminus S_1$ 
such that $K_p=\bK_p$ for $p\in S_{\fin}\setminus(S_0\cup S_1)$. Thus
\[
K^{S_1}=\prod_{p\in S_0}K_p\times \prod_{p\in S_{\fin}\setminus(S_0\cup S_1)}\bK_p.
\]
We note that $\{K_j\}_{j\in\N}$ is a non-degenerate sequence in the
sense defined above. Put
\[
 S:=\{\infty\}\cup S_0\cup S_1.
\]
Note that $K_j=\prod_v K_v$ with $K_v=\K_v$ for $v\not\in S$. 
Hence by the fine geometric expansion \eqref{fine-geom5} we have
\begin{equation}\label{fine-exp6}
\begin{split}
J_{\unip}(\widetilde h_t^{\tau,p}\otimes\chi_{K_j})&=
\sum_{M\in \CmL}\sum_{[u]_S\in\CmU^M_S} a^M([u]_S, S) J_M^{\bG}([u]_S, \widetilde h_t^{\tau,p}\otimes\chi_{K_j})\\
&=\vol(G(\Q)\bs G(\A)^1/K_j)\widetilde h_t^{\tau,p}(1)\\
&\quad+\sum_{(M,[u]_S)\neq (G,\{1\})} a^M([u]_S, S) J_M^{\bG}([u]_S, 
\widetilde h_t^{\tau,p}\otimes\chi_{K_j}).
\end{split}
\end{equation}
Concerning the volume factor in the first summand on the right hand side,
 we used that $\chi_{K_j}=\1_{K_j}/\vol(K_j)$. To deal with the first 
term on the right hand side, we note that $\widetilde h_t^{\tau,p}(1)=
h_t^{\tau,p}(1)$. Furthermore, by
\cite[(5.11)]{MP2} there is an asymptotic expansion
\begin{equation}\label{asymp-exp}
h_t^{\tau,p}(1)\sim \sum_{j=0}^\infty a_j t^{-d/2+j}
\end{equation}
as $t\to 0$, and by \cite[(5.16)]{MP2} there exists $c>0$ such that
\begin{equation}\label{larget3}
h_t^{\tau,p}(1)=O(e^{-ct})
\end{equation}
as $t\to\infty$. From \eqref{asymp-exp} and \eqref{larget3} follows that the 
integral
\begin{equation}
\int_0^\infty h_t^{\tau,p}(1) t^{s-1}dt
\end{equation}
converges in the half-plane $\Re(s)>d/2$ and admits a meromorphic extension
to $\C$ which is holomorphic at $s=0$. The same is true for the integral 
over $[0,T]$ and we get
\begin{equation}\label{identcontr}
\FP_{s=0}\left(\frac{1}{s\Gamma(s)}\int_0^T h_t^{\tau,p}(1) t^{s-1}\right)=
\frac{d}{ds}\left(\frac{1}{\Gamma(s)}\int_0^\infty h_t^{\tau,p}(1) t^{s-1}\right)
\bigg|_{s=0}+O(e^{-cT}).
\end{equation}
Now recall the definition of the $L^{2}$-analytic torsion \cite{Lo}, \cite{MV}.
For $t>0$ let 
\[
K^{(2)}(t,\tau):=\sum_{p=1}^d (-1)^p p h_t^{\tau,p}(1).
\]
Put
\[
t^{(2)}_{\widetilde X}(\tau):=\frac{1}{2}\frac{d}{ds}\left(\frac{1}{\Gamma(s)}
\int_0^\infty K^{(2)}(t,\tau) t^{s-1} dt\right)\bigg|_{s=0}.
\]
To summarize, we get
\begin{equation}\label{l2-tor1}
\frac{1}{2}\sum_{p=1}^d (-1)^p p \FP_{s=0}\left(\frac{1}{s\Gamma(s)}
\int_0^T h_t^{\tau,p}(1) t^{s-1} dt\right)=t^{(2)}_{\widetilde X}(\tau)+ O(e^{-cT})
\end{equation}
for $T\ge 1$. We observe that by \cite[(5.20)]{MP2}, the $L^2$-analytic torsion 
$T^{(2)}_{X(K_j)}(\tau)\in\R^+$ is given by
\begin{equation}\label{l2-torsion}
\log T^{(2)}_{X(K_j)}(\tau)=\vol(X(K_j))\cdot t^{(2)}_{\widetilde X}(\tau).
\end{equation} 
Next we consider the weighted orbital integrals on the right hand side of
\eqref{fine-exp6}. Note that by definition of $\chi_{K_j}$ we have
\[
J_M^G([u]_S,\widetilde h^{\tau,p}_t\otimes \chi_{K_j})=\frac{1}{\vol(K_j)}
J_M^G([u]_S,\widetilde h^{\tau,p}_t\otimes \1_{K_j}).
\]
To deal with the integral on the right hand side, we use the splitting
formula \eqref{eq:splitting}. Put $S_{f}:=S\setminus\{\infty\}=S_0\cup S_1$.
Then we get
\begin{equation}\label{splitting5}
J_M^{\bG}([u]_S, \widetilde h_t^{\tau,p}\otimes\chi_{K_j})=
\sum_{L_1,\,L_2\in\CmL(M)} d_M^{\bG}(L_1,L_2) J_{M}^{L_1}([u]_{\infty}, \widetilde 
h^{\tau,p}_{t,Q_1}) J_M^{L_2}([u]_{S_{f}}, (\One_{K_j})_{S_{f},Q_2}).
\end{equation}
Using \eqref{equ-kernel2} and Proposition \eqref{prop-asympt-exp}, it follows 
that as $t\to 0$, $J_{M}^{L_1}([u]_{\infty}, \widetilde h^{\tau,p}_{t,Q_1})$ has an 
asymptotic expansion  of the form \eqref{asympt-exp7}. This implies that the
integral
\begin{equation}
\int_0^T J_{M}^{L_1}([u]_{\infty}, \widetilde h^{\tau,p}_{t,Q_1})t^{s-1} dt
\end{equation}
converges absolutely and uniformly on compact subsets of the half plane
$\Re(s)>d/2$ and admits a meromorphic extension to $s\in\C$. Put
\begin{equation}
A^{L_1}_M([u]_\infty,T):=\FP_{s=0}\left(\frac{1}{s\Gamma(s)}
\int_0^T J_{M}^{L_1}([u]_{\infty}, \widetilde h^{\tau,p}_{t,Q_1})t^{s-1} dt\right).
\end{equation}
By \eqref{splitting5} it follows that the Mellin transform of
$J_M^{\bG}([u]_S, \widetilde h_t^{\tau,p}\otimes\chi_{K_j})$ as a function of $t$
is a meromorphic function on $\C$, and we get
\begin{equation}\label{splitting6}
\begin{split}
\FP_{s=0}&\left(\frac{1}{s\Gamma(s)}\int_0^T J_M^{\bG}([u]_S, 
\widetilde h_t^{\tau,p}\otimes\chi_{K_j)})t^{s-1} dt\right)\\
&=\sum_{L_1,\,L_2\in\CmL(M)} d_M^{\bG}(L_1,L_2)A^{L_1}_M([u]_\infty,T)
J_M^{L_2}([u]_{S_{f}}, (\One_{K_j})_{S_{f},Q_2}),
\end{split}
\end{equation}
where $(\One_{K_j})_{S_{f},Q_2}$ is defined by \eqref{auxil-fct}. Next we deal with
the orbital integral on the right hand side. Let $L\in\cL$ and let $Q=LV$ be 
a semistandard parabolic subgroup. We first consider the case that the $K_j$ are principal congruence subgroups $K(N_j)$ with $N_j\longrightarrow\infty$ as $j\rightarrow\infty$.  
Using the definition \eqref{auxil-fct} and the fact that $K(N_j)$ is a normal 
subgroup of $\K_f$, we have
\[
\One_{K(N_j)_{S_{f}},Q}(m)=\delta_Q(m)^{1/2}\vol(\K_{S_\fin})\int_{V(\Q_{S_\fin})}\One_{K(N_j)_{S_\fin}}
(mv)dv
\]
for any $m\in L(\A_f)$. Hence $\One_{K(N_j)_{S_{f}},Q}(m)=0$ unless $m\in
K^L(N_j)_{S_\fin}=K(N_j)_{S_\fin}\cap L(\A_f)$. Now if $m\in K^L(N_j)_{S_\fin}$, we have 
$mv\in K(N_j)_{S_\fin}$ if and only if $v\in K(N_j)_{S_\fin}$. Hence
\[
\One_{K(N_j)_{S_{f}},Q}(m)=\vol(\K_{S_\fin})\One_{K^L(N_j)_{S_\fin}}(m)
\]
which converges to $0$ for any $m\neq 1$. 

We now go back to the general case of a sequence of open compact subgroups $K_j$ of $\K_f$ that converges to $1$ at $S_\fin$. Let $N_S$ denote the product of     all primes in $S_\fin$. By definition, for each $m\ge0$, there exists $J_m$ such that $K_j\subseteq K(N_S^m)$ for all $j\ge J_m$. Using the normality of $K(N_S^m)$ again, we also have $\{k^{-1}xk\mid k\in \K,\, x\in K_j\}\subseteq K(N_S^m)$ for all $j\ge J_m$. Using the computation for the principal congruence subgroups, we can therefore conclude that $\One_{K_{j,S_{f}},Q}$ also converges pointwise to $0$ away from $1$.
 Now we argue as in \cite[Lemma 7]{Cl} to conclude that for $[u]_{S_\fin}\neq 1$,
\begin{equation}
\lim_{j\to\infty}J^L_M([u]_{S_{f}}, \One_{K{j,S_{f}}})=0.
\end{equation}
Denote by $J_{\unip-\{1\}}(\widetilde h_t^{\tau,p}\otimes \chi_{K_j})$ the sum
on the right hand side of \eqref{fine-exp6}. Combining our results, we have proved

\begin{lem}\label{lem:unip-conv}
For every $T\ge 1$ we have
\[
\frac{1}{\vol(X(K_j))}\FP_{s=0}\left(\frac{1}{s\Gamma(s)}\int_0^T
J_{\unip-\{1\}}(\widetilde h_t^{\tau,p}\otimes \chi_{K_j})t^{s-1}dt\right)
\to 0
\]
as $j\to\infty$.
\end{lem}
Now we can complete the proof of Theorem \ref{theo-approx}. Let
\[
  K_j(t,\tau):=\frac{1}{2}\sum_{k=1}^d (-1)^kk \Tr_{\reg}
  \left(e^{-t\Delta_{k,j}(\tau)}\right)
\]
where the $\Delta_{k,j}(\tau)$ denotes the the twisted Laplace operator on $k$ forms with values in $T_\tau$ over $X(K_j)$.  
Let $T>0$. By \eqref{analtor} and \eqref{zetafct} we have
\begin{equation}\label{int-decomp}
\log T_{X(K_j)}(\tau)=\FP_{s=0}\left(\frac{1}{s\Gamma(s)}\int_0^T K_j(t,\tau) t^{s-1}dt\right)+\int_T^\infty K_j(t,\tau)t^{-1} dt.
\end{equation}
By \eqref{est-larget3} there exist $C,c>0$ such that
\begin{equation}\label{est-larget7}
\frac{1}{\vol(X(K_j)}\left|\int_T^\infty K_j(t,\tau)t^{-1} dt\right|\le C e^{-cT}
\end{equation}
for all $T\ge 1$ and $n\in\N$. 

Now we turn to the first term on the right hand side. Let $h_{t,T}^{\tau,p}\in 
C_c^\infty(G(\R)^1)$ be defined by \eqref{cpt-supp4}. Put
\[
K_j(t,\tau;T):=\frac{1}{2}\sum_{k=1}^d (-1)^k k 
J_{\geo}(h_{t,T}^{\tau,k}\otimes \chi_{K_j}).
\]
Using \eqref{cpt-supp3} and the trace formula, it follows that, up to a term of 
order $O(e^{-cT})$, we  can replace $K_j(t,\tau)$ by $K_j(t,\tau,T)$ in the
first term on the right hand side of \eqref{int-decomp}.  Let
\begin{equation}\label{cpt-supp5}
K_{\unip-\{1\},j}(t,\tau;T):=\frac{1}{2}\sum_{p=1}^d (-1)^p p 
J_{\unip-\{1\}}(h_{t,T}^{\tau,p}\otimes \chi_{K_j}).
\end{equation}
By \cite[Lemma 5]{Cl}  and \eqref{fine-exp6} it follows that for every 
$T\ge 1$ there exists $J_0(T)\in\N$ such that 
\[
K_j(t,\tau;T)=\frac{\vol(X(K_j))}{2}\sum_{k=1}^d (-1)^k k h_{t,T}^{\tau,k}(1)
+K_{\unip-\{1\},j}(t,\tau;T)
\]
for $j\ge J_0(T)$. Using \eqref{l2-tor1} it follows that for every $T\ge 1$ 
there exists $J_0(T)\in\N$ such that
\begin{equation}\label{eq:l2-tor}
\begin{split}
\frac{1}{\vol(X(K_j))}&\FP_{s=0}\biggl(\frac{1}{s\Gamma(s)}
\int_0^T K_j(t,\tau) t^{s-1}dt\biggr)\\
&=t^{(2)}_{\widetilde X}(\tau)+\frac{1}{\vol(X(K_j))}\FP_{s=0}
\biggl(\frac{1}{s\Gamma(s)}\int_0^T K_{\unip-\{1\},j}(t,\tau) t^{s-1}dt\biggr)\\
&\hskip10pt+O(e^{-cT})
\end{split}
\end{equation}
for $j\ge J_0(T)$. 
For the last step we have to specialize to sequences of groups $\{K_j\}_{j\in\N}$ which converge to $1$ at $S_1$.
Applying Lemma \ref{lem:unip-conv} we get that for
every $T\ge 1$ and $\varepsilon>0$ there exist constants $C_1(T), C_2, c>0$ 
and $J_0(T)\in\N$ 
such that
\begin{equation}\label{est-l2-tor2}
\begin{split}
\biggl|\frac{1}{\vol(X(K_j))}\FP_{s=0}\biggl(\frac{1}{s\Gamma(s)}
\int_0^T K_j(t,\tau)t^{s-1}dt\biggr)-t^{(2)}_{\widetilde X}(\tau)\biggr|
\le C_1(T) \varepsilon+C_2 e^{-cT}
\end{split}
\end{equation}
for $j\ge J_0(T)$. Combined with \eqref{int-decomp}, \eqref{est-larget7} and
\eqref{est-larget3}, Theorem \ref{theo-approx} follows.

\subsection{Proof of Corollary \ref{cor:sl3}}

Let $G=\SL(3)$. We proceed as in the proof of Theorem \ref{theo-approx}.
 The starting point is the decompositon
\eqref{splitt-at} of the Mellin transform of the regularized trace of the
heat operator. Concerning the second integral, the estimation 
\eqref{est-larget3} holds in general. We only have to deal with the first
integral. As above, using \eqref{cpt-supp3} and \eqref{unip4}, this can be 
reduced to the consideration of the unipotent contribution. 

To this end we use the description of the unipotent contribution to the 
geometric side of the trace formula given in \cite[Appendix B]{HW}. To begin
 with we need to introduce some notation.

For $z\in Z(\Q)=\{\pm I_3\}$ let $\of_z$ be the $\cO$-equivalence class in
$G(\Q)$, and for $x\in G(\Q)$ denote by $\{x\}_G$ the $G(\Q)$-conjugacy class 
of $x$. The set $\of_z$ is divided into three classes
\begin{equation}
\of_z=O_{\tri,z}\cup O_{\min,z}\cup O_{\reg,z}, 
\end{equation}
where $O_{\tri,z}=\{z\}$,
\begin{equation}\label{unip-conj-classes}
O_{\min,z}=\{z u(0,1,0)\}_G,\quad\text{and}\quad 
O_{\reg,z}=\bigcup_{\alpha\in\Q^\times/(\Q^\times)^3}\{z u(1,0,\alpha)\}_G
\end{equation}
(cf.\ \cite[\S 3]{Hales}).
For $T_0\in\af_0^+$ and $f\in C^\infty_c(G(\A)^1)$ let $J^{T_0}_{O_{\tri,z}}(f)$,
$J^{T_0}_{O_{\min,z}}(f)$, and $J^{T_0}_{O_{\reg,z}}(f)$ be the integrals defined in
\cite[Appendix A.3]{HW}. Then the unipotent contribution $J^{T_0}_{\of_z}(f)$ is
given by
\begin{equation}\label{unip-sl3}
J^{T_0}_{\of_z}(f)=J^{T_0}_{O_{\tri,z}}(f)+J^{T_0}_{O_{\min,z}}(f)+J^{T_0}_{O_{\reg,z}}(f).
\end{equation}
Let $S$ be a finite set of places of $\Q$
containing the Archimedean place. For a character $\chi=\prod_v \chi_v$  of
$\Q^\times\bs\A^1$, where $\chi_v$ is a character
of $\Q^\times_v$, which is unramified outside of $S$ let 
\begin{equation}
L^S(s,\chi):=\prod_{p\not\in S} L_p(s,\chi_p)
\end{equation} 
be the partial Hecke $L$-function outside of $S$. It is well known that
$L^S(s,\chi)$ is absolutely convergent and holomorphic for $\Re(s)>1$, and
admits a meromorphic extension to the entire complex plane, which in fact is entire when $\chi$ is not equal to the trivial character $\chi_0$ on $\Q^\times\backslash \A^1$.
If $S=\{\infty\}$, we write $L(s,\chi)$ for $L^S(s,\chi)$. For $\chi=\chi_0$ we further write  $\zeta^S(s):=
L^S(s,\chi_0)$ so that $\zeta(s)=\zeta^{\{\infty\}}(s)$ coincides with the Riemann zeta function.

Now we are ready to describe the global coefficients.
Let  
\[
L^S(s,\chi)
= \alpha_{-1}^S(\chi) (s-1)^{-1} + \alpha_0^S(\chi) + \alpha_1^S(\chi) (s-1) + \ldots
\]
be the Laurent expansion of $L^S(s,\chi)$ around $s=1$. We write $\alpha_k^S= \alpha_k^S(\chi_0)$ if $\chi=\chi_0$ is the trivial character. Note that $\alpha_{-1}^S(\chi)=0$ if $\chi\neq \chi_0$. For a semistandard Levi subgroup $H$ of $\SL(3)$ denote by $\vol_H$ the volume of $H(\Q)\bs H(\A)^1$. 

Let
\[
u(x_{12},x_{13},x_{23}):=\begin{pmatrix} 1&x_{12}&x_{13}\\0&1&x_{23}\\0&0&1
\end{pmatrix},\quad x_{12},x_{13},x_{23}\in\R.
\]
For abbreviation put
\[
\nu:=u(0,1,0)\quad\text{and}\quad \mu_x:=u(1,0,x).
\]
Let $M_0$ be the standard minimal Levi subgroup and $M$ the $(2,1)$ standard
Levi subgroup of $\SL(3)$.
The corresponding global coefficients then are according to \cite[Appendix B]{HW} given by 
\[
a^{\SL(3)}(S, \nu)= \vol_M\zeta^S(2)^{-1} (\zeta^S)'(2)
\]
and
\[
 a^{\SL(3)}(S, \mu_x) =  \frac{\vol_{M_0}}{3\alpha_1^2}\left( \alpha_1^S \alpha_{-1}^S + \sum_{\chi\neq \chi_0} \chi(x) \alpha_0^S(\chi) \alpha_0^S(\chi^{-1}) \right)
\]
where the sum runs over all cubic characters $\chi:\Q^{\times}\backslash\A^1\longrightarrow\C^\times$ that are unramified outside of $S$. Next we describe the integrals
on the right hand side of \eqref{unip-sl3}. For the first one we have 
$J^{T_0}_{O_{\tri,z}}(f)=\vol_G f(z)$. Let 
\begin{equation}\label{weight-orb-int}
J_{M}^{T_0}(z,f),\; J_G(z u(0,1,0),f),\;J^{T_0}_{M_0}(z,f),\;\text{and}\;
J^{T_0}_M(z u(1,0,0),f)
\end{equation}
be the weighted orbital integrals defined in \cite[Appendix A.3]{HW} which coincide with Arthur's weighted orbital integrals. Let 
$f\in C_c^\infty(G(\Q_S))$.  By \cite[Theorem B.1]{HW} the integral 
$J^{T_0}_{O_{\min,z}}(f)$ converges absolutely and we have
\begin{equation}\label{orb-int-sl3}
J^{T_0}_{O_{\min,z}}(f)=\vol_{M} J^{T_0}_M(z,f)+ a^{\SL(3)}(S, \nu) J_G(z\nu,f),
\end{equation}
For any $x\in\Q_S$ let $J_G(zu(1,0,x),f)$ be the integral defined
in \cite[p. 85]{HW}. 
By \cite[Theorem B.2]{HW} the integral $J^{T_0}_{O_{\reg,z}}(f)$ converges
absolutely and we have
\begin{equation}\label{orb-int-sl3-1}
\begin{split}
J^{T_0}_{O_{\reg,z}}(f)=&\frac{\vol_{M_0}}{6} J^{T_0}_{M_0}(z,f)+
\frac{\vol_{M_0}}{2} \alpha_0^S J^{T_0}_M(z\mu_0.f)\\
&+\frac{\vol_{M_0}}{3}\sum_{x\in\Q^\times/(\Q^\times\cap(\Q^\times_S)^3)}
a^{\SL(3)}(S,\mu_x) J_G(z\mu_x,f).
\end{split}
\end{equation}
The global coefficients are estimated by the following lemma. 
\begin{lem}\label{lem:coeff}
There exists a constant $c>0$, and for each $\epsilon>0$ a constant $c_\epsilon>0$, both $c$ and $c_\epsilon$ independent of $S$ and $x$, such that 
\[
\left| a^{\SL(3)}(S, \nu)\right| \le c
\]
and
\[
 a^{\SL(3)}(S, \mu_x)
\le c_\epsilon N_S^\epsilon \left(\log N_S\right)^2
\]
where $N_S=\prod_{p\in S_\fin} p$.
\end{lem}
\begin{proof}
The first assertion immediately follows from the fact that $\zeta^S(2)^{-1} (\zeta^S)'(2)$ is bounded from above independently of $S$. 

For the second assertion using standard estimates for the derivatives of Hecke $L$-functions at $1$ (see, e.g., \cite[Theorem 8.18]{Tenenbaum}) one can easily establish that for every Hecke character $\chi$ of $\Q$ that is unramified outside of $S$, and every $k\in\{-1,0,1\}$, 
\[
\alpha_k^S(\chi) \le c (\log N_S)^{k+1}
\]
where $c>0$ is a suitable constant that does not depend on $S$ or $\chi$ and $N_S:=\prod_{p\in S} p$. Finally, the number of cubic Hecke characters of $\Q$ that are unramified outside of $S$ is bounded from above by $9^{|S_\fin|}$ since $|\Q_p^\times/(\Q_p^\times)^3|\le 9$.
\end{proof}
Now we have to deal with the weighted orbital integrals with test function
$f:=\widetilde h_t^{\tau,p}\otimes \chi_{K_j}$. Except for
$J_G(z\mu_x,\widetilde h_t^{\tau,p}\otimes \chi_{K_j})$, the orbital integrals are
the same as the corresponding integrals for $\GL(3)$. These integrals can be 
treated as in \cite{MM2}. Using the splitting formula \eqref{eq:splitting}, 
we are reduced to the study of weighted orbital integrals of the form
$J_M^L(u,{\bf 1}_{K_p})$, for Levi subgroups $M\subseteq L\subseteq G$, $u\in M(\Q)$ and $p$ any prime. The remaining integral 
$J_G(z\mu_x,\widetilde h_t^{\tau,p}\otimes \chi_{K_j})$ does not contain any 
weight factor and can be dealt with directly. The integrals 
$J_M^L(u,{\bf 1}_{K_p})$ are estimated by the following lemma.
\begin{lem}\label{lem:weighted:orb}
There exist $c, R>0$ such that
for all Levi subgroups $M\subseteq L\subseteq G$ and all unipotent elements $u\in M(\Q)$ with $(u,M)\neq (1,L)$, we have
\[
J_M^L(u,1_{K_p})
\le c p^{-R}
\] 
for every prime $p$
\end{lem}
\begin{proof}
This can easily be seen from the explicit description of the weighted orbital integrals in \cite[Appendix A and B]{HW}. In fact, most weighted orbital integrals are the same as in the $\GL(3)$ case so that the estimate for those follows from \cite{MM2}. The only remaining integrals are those described just before  \cite[Theorem B.2]{HW} and do not contain any logarithmic weight functions. It is easily seen from their explicit description in \cite{HW} that the asserted bound also holds for them. 
\end{proof}

Combining Lemma \ref{lem:coeff}, Lemma \ref{lem:weighted:orb}, and the fact that $|\Q_p^{\times}/(\Q_p^\times)^3|\le 9$, Corollary \ref{cor:sl3} follows as \cite[Theorem 1.1]{MM2} in \cite[\S 10]{MM2}.


\begin{thebibliography}{-------}

\bibitem[AR]{AR} P.\ Albin, F.\ Rochon, {\it Some index formulae on the moduli 
space of stable parabolic vector bundles}. J. Aust. Math. Soc. {\bf 94} (2013), 
no.\ 1, 1--37. 

\bibitem[Ar78]{Ar1} J.\ Arthur, {\it A trace formula for reductive groups. 
I. Terms associated to classes in $G(\Q)$},  Duke Math. J. {\bf 45}  (1978), no.\ 4, 911--952.

\bibitem[Ar81]{Ar3} J.\ Arthur, {\it The trace formula in invariant form}, 
Ann. of Math. (2) {\bf 114} (1981), no.\ 1, 1--74.

\bibitem[Ar82]{Ar9} J.\ Arthur,
{\it On a family of distributions obtained from {E}isenstein series. {II}:  {E}xplicit formulas.},
   Amer. J. Math. {\bf 104} (1982), no.\ 6, 1289--1336.

\bibitem[Ar85]{Ar85} J.\ Arthur, {\it A measure on the unipotent variety}, Can. J. Math., Vol. XXXVII,  (1985), no.\ 6, 1237--1274.

\bibitem[Ar88]{Ar88} J.\ Arthur, {\it The local behavior of weighted orbital 
integrals}, Duke Math. J. \textbf{56} (1988), no.\ 2, 223--293.


\bibitem[Ar89]{Ar4} J.\ Arthur, {\it Intertwining operators and residues, I, 
Weighted characters}, J. Funct. Analysis {\bf 84} (1989), 19--84.


\bibitem[Ar05]{Ar05} J.\ Arthur, {\it An Introduction to the Trace Formula}, Clay Math. Proc. (2005)

\bibitem[BLS]{BLS} H. Bass, M. Lazard, J.-P. Serre, {\it Sous-groupes d'indice 
fini dans} $\SL(n,\Z)$,  Bull. Amer. Math. Soc. {\bf 70} (1964), 385--392. 

\bibitem[BV]{BV} N.\ Bergeron, A.\ Venkatesh, {\it The asymptotic growth of 
torsion homology for arithmetic groups}, J. Inst. Math. Jussieu {\bf 12}
 (2013), no.\ 2, 391--447.


\bibitem[BG]{BG} A.\ Borel and H.\ Garland, {\it Laplacian and the discrete spectrum of an arithmetic group},
 Amer. J. Math. {\bf 105} (1983), no.\ 2, 309--335.



\bibitem[CV]{CV} F.\ Calegari, A.\ Venkatesh, {\it A torsion Jacquet-Langlands correspondence},
  arXiv:1212.3847. 

 
\bibitem[Car]{Carter} R.W.\ Carter, {\it Finite Groups of Lie Type}, Pure Appl. Math. {\bf 44} (1985)

\bibitem[Ch]{Ch} H.-P. Chaudouard, {\it Sur certaines contributions unipotentes dans la formule des traces d'Arthur}, Amer. J. Math., {\bf 140}(3): 699--752, 2018. 


\bibitem[CD]{CD} L.\ Clozel, P.\ Delorme, {\it Le th\'eor\`eme de {P}aley-{W}iener invariant pour les groupes de
 {L}ie r\'eductifs}, Invent. Math. {\bf 77} (1984), no.\ 3, 427--453.


\bibitem[CLL]{CLL} L.\ Clozel, J.P.\ Labesse, R.P.\ Langlands, {\it Morning Seminar
on the Trace Fromuala}, mimeographed notes, IAS, Princeton, 1984.

\bibitem[Cl]{Cl} L. Clozel, {\it On limit multiplicities of discrete series 
representations in spaces of automorphic forms}, Invent. math. {\bf 83} (1986),
265 -- 284.

\bibitem[CoMc]{CoMc} D.H.\ Collingwood and W.M.\ McGovern, {\it Nilpotent orbits in semsimple Lie algebras}, CRC Press (1993).

\bibitem[FL16]{FL1} T.\ Finis, E.\ Lapid, {\it On the continuity of the geometric
side of the trace formula},  Acta Math. Vietnam. {\bf 41} (2016), no. 3, 
425--455.
 
\bibitem[FL17]{FL2} T. Finis, E. Lapid, {\it On the analytic properties of 
intertwining operators I: global normalizing factors}. Bull. Iranian Math. 
Soc. {\bf 43} (2017), no. 4, 235 -- 277.

\bibitem[FL18]{FL3}T. Finis, E. Lapid, {\it An approximation principle for 
congruence subgroups II: application to the limit multiplicity problem}.
 Math. Z. {\bf 289} (2018), no. 3-4, 1357 -- 1380. 

\bibitem[FL19]{FL4}T. Finis, E. Lapid, {\it On the analytic properties of 
intertwining operators II: Local degree bounds and limit multiplicities}. 
Israel J. Math. {\bf 230} (2019), no. 2, 771 -- 793. 



\bibitem[FLM11]{FLM1} T.\ Finis, E.\ Lapid and W.\ M\"uller, {\it On the spectral side of Arthur's trace formula - absolute convergence}, Ann. of Math., vol. 174 (2011), no. 1, 173--195.

\bibitem[FLM15]{FiLaMu} T.\ Finis, E.\ Lapid and W.\ M\"uller, {\it Limit multiplicities for principal congruence subgroups of $\GL(n)$ 
and $\SL(n)$}, J. Inst. Math. Jussieu, {\bf 14}, no.\ 3, 589--638, 2015

\bibitem[Fli]{Fli} Y. Z. Flicker, 
\newblock The Trace Formula and Base Change for $\GL(3)$,
\newblock {\em Lecture Notes in Mathematics}, {\bf 927}. Springer-Verlag, 
Berlin-New York, 1982. 

\bibitem[GaVa]{GaVa} R.\ Gangolli and V.S.\ Varadarajan, {\it Harmonic analysis of spherical functions on real reductive groups}, Springer (1988).


\bibitem[Hal]{Hales} T.\ C.\  Hales, {\it Unipotent Representations and Unipotent Classes ${{\rm SL} (N)}$}, Amer. J. of Math., vol. 115 no. 6 , 1347--1383, 1993


\bibitem[HC]{HC2} Harish-Chandra, {\it The Plancherel formula for reductive
p-adic groups}, In: Collected papers, IV, pp. 353-367, Springer-Verlag, 
New York-Berlin-Heidelberg, 1984.

\bibitem[Ho]{Ho} W. Hoffmann, {\it The trace formula and prehomogeneous vector 
spaces}. In: Families of automorphic forms and the trace formula, 175--215, 
Simons Symp., Springer, 2016. 

\bibitem[HW]{HW} W. Hoffmann, S. Wakatsuki, {\it On the geometric side of the 
Arthur trace formula for the symplectic group of rank 2}, Mem. Amer. Math. 
Soc. {\bf 255} (2018), no. 1224.

\bibitem[Kn]{Kn} A.W.\ Knapp, {\it Representation theory of semisimple groups},
Princeton University Press, Princeton and Oxford, 2001.

\bibitem[Lo]{Lo} J. Lott, \emph{Heat kernels on covering spaces and 
topological invariants}. J. Differential Geom. \textbf{35} (1992), no. 2, 
471 -- 510.

\bibitem[MaM]{MaM}  S.\ Marshall, W.\ M\"uller, {\it On the torsion in the 
cohomology of arithmetic hyperbolic 3-manifolds}. Duke Math. J. {\bf 162}
 (2013), no. 5, 863--888.

\bibitem[Mat]{MV} V. Mathai, 
\newblock $L^2$-analytic torsion,
{\em J. Funct. Anal.} {\bf 107}(2), (1992), 369 -- 386.

\bibitem[MM]{MM} Y.\ Matsushima, S.\ Murakami, 
{\it On vector bundle valued harmonic forms and automorphic forms on 
symmetric riemannian manifolds},  Ann. of Math. {\bf 78} (1963), 365--416. 

\bibitem[Ma1]{Ma1} Matz, J., \emph{Bounds for global coefficients in the fine geometric expansion of {A}rthur’s trace formula for ${{\rm GL}} (n)$}, Israel J. Math. {\bf 205}, no. 1,  (2015), 337--396.

\bibitem[MM17]{MM1} J.\ Matz and W. M\"uller, {\it Analytic torsion of arithmetic quotients of the symmetric space {$\SL_n(\R)/\SO(n)$}}, Geom. Funct. Anal. {\bf 27} no.\ 6 (2017), 1378--1449.


\bibitem[MM2]{MM2} J.\ Matz and W. M\"uller, {\it Approximation of $L^2$-analytic torsion for arithmetic quotients of the symmetric space $\mathrm{SL}(n,{\bf R})/\mathrm{SO}(n)$} to appear in J. Inst. Math. Jussieu.

\bibitem[Me]{Me} R.B.\ Melrose, {\it The Atiyah-Patodi-Singer index theorem},
Research Notes in Mathematics, 4. A K Peters, Ltd., Wellesley, MA, 1993.

\bibitem[Mn]{Mn} J. Mennicke, {\it Finite factor groups of the unimodular 
group}, Ann. of Math. {\bf 81} (1965), 31 -- 37.  


\bibitem[Mia]{Mia} R.J.\ Miatello, {\it The Minakshisundaram-Pleijel 
coefficients for the vector-valued heat kernel on compact locally symmetric 
spaces of negative curvature.}  Trans. Amer. Math. Soc. {\bf 260} (1980), 
1--33. 

\bibitem[MoS]{MoS} H. Moscovici and R.J. Stanton, {\it R-torsion and zeta 
functions for locally symmetric manifolds.} Invent. Math. {\bf 105} (1991), 
no. 1, 185 -- 216. 

\bibitem[Mu89]{Mu3} W.\ M{\"u}ller,
{\it The trace class conjecture in theory of automorphic forms},
Ann. of Math. {\bf 130} (1989), 473 -- 529.


\bibitem[Mu02]{Mu2} W.\ M{\"u}ller,
{\it On the spectral side of the {A}rthur trace formula},
 Geom. Funct. Anal. {\bf 12} (2002), 669--722.

 \bibitem[Mu07]{Mu1} W.\ M\"uller, {\it Weyl's law for the cuspidal spectrum of 
$\SL_n$}, Annals of Math. {\bf 165} (2007), 275--333.


\bibitem[Mu17]{Mu4} W.\ M\"uller, {\it On the analytic torsion of hyperbolic 
manifolds of finite volume}. Geometry, analysis and probability, 165–189, 
Progr. Math., 310, Birkhäuser/Springer, Cham, 2017.

\bibitem[MP12]{MP1} W.\ M\"uller, J.\ Pfaff, {\it Analytic torsion of complete 
hyperbolic manifolds of finite volume.} J. Funct. Anal. {\bf 263} (2012), 
no. 9, 2615--2675.

\bibitem[MP13]{MP2} W.\ M\"uller, J.\ Pfaff, {\it Analytic torsion and 
$L^2$-torsion of compact locally symmetric manifolds},
J. Diff. Geometry {\bf 95}, no.\ 1, (2013), 71 -- 119. 

\bibitem[MP14a]{MP3} W.\ M\"uller, J.\ Pfaff {\it The analytic torsion and its 
asymptotic behavior for sequences of hyperbolic manifolds of finite volume.}
 J. Funct. Anal. {\bf 267} (2014), no.\ 8, 2731--2786.


\bibitem[MP14b]{MP4} W.\ M\"uller, J.\ Pfaff, {\it On the growth of torsion in 
the cohomology of arithmetic groups.} Math. Ann. {\bf 359} (2014), no.\ 1-2, 
537--555. 

\bibitem[MR1]{MR1} W.\ M\"uller, F.\ Rochon, {\it Analytic torsion and Reidemeister torsion of hyperbolic manifolds with cusps},  	arXiv:1903.06199.

\bibitem[MR2]{MR2} W.\ M\"uller, F.\ Rochon, {\it Exponential growth of torsion in the cohomology of arithmetic hyperbolic manifolds},  	arXiv:1903.06207.

\bibitem[MS04]{MS} W.\ M\"uller, B.\ Speh, {\it Absolute convergence of the 
spectral side of the Arthur trace formula for $\GL(n)$
With an appendix by E.M.\ Lapid.}  Geom. Funct. Anal. {\bf 14} (2004), no.\ 1, 
58--93.

\bibitem[Pa]{Pa} J. Park, {\it Analytic torsion and Ruelle zeta functions for 
hyperbolic manifolds with cusps}, J. Funct. Anal. {\bf 257} (2009), no. 6, 
1713--1758. 

\bibitem[PlRa]{PlRa} V.P.\ Platonov and P.S.\ Rapinchuk, {\it Algebraic Groups and Number Theory}, Russian Mathematical Surveys {\bf 47}, no.\ 2, (1993).

\bibitem[PR]{PR} J.\ Pfaff, J.\ Raimbault, {\it The torsion in symmetric 
powers on congruence subgroups of Bianchi groups},  Trans. AMS {373} (2020), 
no. 1, 109--148.

\bibitem[Ra1]{Ra1} J.\ Raimbault,
{\it Asymptotics of analytic torsion for hyperbolic three--manifolds},
Comment. Math. Helv. {\bf 94} (2019), no. 3, 459--531.

\bibitem[Ra2]{Ra2} J.\ Raimbault,
{\it Analytic, Reidemeister and homological torsion for congruence 
three--manifolds},  Ann. Fac. Sci. Toulouse Math. {\bf 28} (2019),
no. 3, 417--469.


\bibitem[Rao]{Rao} R.\ Ranga Rao,{\it Orbital integrals on reductive groups},
 Ann. of Math., (2), {\bf 96} (1972), 505--510.
 
\bibitem[RS]{RS} D.B.\ Ray, I.M.\ Singer, {\it $R$-torsion and the Laplacian on 
Riemannian manifolds},  Advances in Math. {\bf 7}, (1971), 145--210.

\bibitem[Sh81]{Sha1} F.\ Shahidi, 
{\it On certain $L$-functions}, Amer. J. Math. {\bf 103} (1981), no.\ 2, 297--355.

\bibitem[Sh88]{Sha2} F.\ Shahidi, 
{\it On the Ramanujan conjecture and finiteness of poles for certain 
$L$-functions}, Ann. of Math. (2) {\bf 127} (1988), no.\ 3, 547--584.

\bibitem[Sh]{Sh} M.A. Shubin, {\it Pseudodifferential operators and spectral 
theory}, Second edition, Springer-Verlag, Berlin, 2001.

\bibitem[Si]{Si1} A.J.\ Silberger,
{\it Harish-Chandra's Plancherel Theorem for $p$-adic Groups}, Trans. Amer. Math. Soc. {\bf 348} (1996), no.\ 11, 4673 -- 4686.

\bibitem[Te]{Tenenbaum} G.\ Tenenbaum, {\it Introduction to Analytic and Probablistic Number Theory}, Graduate Studies in Mathematics, vol. 163, American Mathematical Society, 2015. 

\end{thebibliography}
\end{document}